\documentclass[11pt]{article}

\date{December 16, 2025}

\usepackage{mathpazo}   

\usepackage{tikz}
\definecolor{myseagreen}{HTML}{3FBC9D}
\usepackage[color=myseagreen]{todonotes}

\usepackage{soul}
\usepackage{nameref}
\usepackage{mathtools}
\usepackage{subcaption}
\usepackage{empheq}
\usepackage{comment}
\usepackage[shortlabels,inline]{enumitem}
\usepackage{wrapfig}
\usepackage{ragged2e}

\setlist[enumerate]{nosep}
\usepackage[colorlinks=true,
linkcolor=refkey,
urlcolor=lblue,
citecolor=magenta]{hyperref}
\usepackage[doc,wmm,hhb]{optional}
\usepackage{xcolor}

\usepackage{float}
\usepackage{soul}
\usepackage{graphicx}
\definecolor{labelkey}{rgb}{0,0.08,0.45}
\definecolor{refkey}{rgb}{0,0.6,0.0}
\definecolor{Brown}{rgb}{0.45,0.0,0.05}
\definecolor{lime}{rgb}{0.00,0.8,0.0}
\definecolor{lblue}{rgb}{0.5,0.5,0.99}
\definecolor{OliveGreen}{rgb}{0,0.6,0}
\definecolor{tyrianpurple}{rgb}{0.4, 0.01, 0.24}
\definecolor{darkgreen}{RGB}{0,100,0}

\colorlet{hlcyan}{cyan!30}

\usepackage{stmaryrd} 
\usepackage{amssymb}
\usepackage[title]{appendix}
\usepackage{amsthm}
\usepackage{aliascnt} 
\usepackage[capitalize,nameinlink]{cleveref}
\usepackage{relsize}

\makeatletter
\def\namedlabel#1#2{\begingroup
   \def\@currentlabel{#2}%
   \label{#1}\endgroup
}
\makeatother

\newcommand{\seppthree}{\setlength{\itemsep}{-3pt}}

\usepackage[margin=1in,footskip=0.25in]{geometry}
\newcommand{\weakly}{\ensuremath{\:{\rightharpoonup}\:}}

\newcommand{\nnn}{\ensuremath{{n\in{\mathbb N}}}}

\newcommand{\thalb}{\ensuremath{\tfrac{1}{2}}}
\newcommand{\menge}[2]{\big\{{#1}~\big |~{#2}\big\}}

\newcommand{\fenv}[1]%
{\ensuremath{\,\overrightarrow{\operatorname{env}}_{#1}}}
\newcommand{\benv}[1]%
{\ensuremath{\,\overleftarrow{\operatorname{env}}_{#1}}}

\newcommand{\scal}[2]{\left\langle{#1},{#2}  \right\rangle}

\newcommand{\RR}{\ensuremath{\mathbb R}}

\newcommand{\NN}{\ensuremath{\mathbb N}}

\newcommand{\inte}{\ensuremath{\operatorname{int}}}

\newcommand{\conv}{\ensuremath{\operatorname{conv}\,}}
\newcommand{\lspan}{\ensuremath{\operatorname{span}\,}}
\newcommand{\aff}{\ensuremath{\operatorname{aff}\,}}

\newcommand{\caff}{\ensuremath{\overline{\operatorname{aff}}\,}}

\providecommand{\fejer}{Fej\'{e}r}
{\begin{list}{}{%
\settowidth{\labelwidth}{\textrm{#1~}}%
\setlength{\leftmargin}{\labelwidth+\labelsep}}}
{\end{list}}
\usepackage{amsthm}
\makeatletter
\def\th@plain{%
	\thm@notefont{}
	\itshape 
}
\def\th@definition{%
	\thm@notefont{}
	\normalfont 
}
\makeatother
\usepackage[capitalize,nameinlink]{cleveref}
\crefname{equation}{}{equations}
\crefname{chapter}{Appendix}{chapters}
\crefname{item}{}{items}
\crefname{enumi}{}{}
\newtheorem{theorem}{Theorem}[section]
\newtheorem{lemma}[theorem]{Lemma}

\newtheorem{corollary}[theorem]{Corollary}

\newtheorem{proposition}[theorem]{Proposition}

\newtheorem{definition}[theorem]{Definition}

\newtheorem{example}[theorem]{Example}
\newtheorem{ex}[theorem]{Example}
\newtheorem{fact}[theorem]{Fact}
\newtheorem{remark}[theorem]{Remark}

\providecommand{\RR}{\mathbb{R}}

\providecommand{\aff}{\operatorname{aff}}
\providecommand{\conv}{\operatorname{conv}}

\providecommand{\NN}{\mathbb{N}}

\providecommand{\ri}{\operatorname{ri}}

\providecommand{\RR}{\mathbb{R}}
\providecommand{\NN}{\mathbb{N}}

\definecolor{myblue}{rgb}{0.9,0.9,0.98}

  \newcommand*\mybluebox[1]{%
    \colorbox{myblue}{\hspace{1em}#1\hspace{1em}}}

\allowdisplaybreaks 

\usepackage{relsize}

\newcommand{\crefpart}[2]{%
  \hyperref[#2]{\namecref{#1}~\labelcref*{#1}~\ref*{#2}}%
}

\author{
Aleksandr\ Arakcheev\thanks{
Mathematics, University
of British Columbia,
Kelowna, B.C.\ V1V~1V7, Canada. E-mail:
 \texttt{aleksandr.arakcheev@ubc.ca}.}~
 \;\;and\;\;
Heinz H.\ Bauschke\thanks{
Mathematics, University
of British Columbia,
Kelowna, B.C.\ V1V~1V7, Canada. E-mail:
\texttt{heinz.bauschke@ubc.ca}.}
}

\title{\textsf{
  \fejer\ and \fejer* Monotonicity: 
  New Results and Limiting Examples 
}
}

\makeatletter

\let\orig@label\label

\renewcommand{\label}[1]{%
  \begingroup
  \def\@currentlabelname{}%
  \ifx\current@theorem\relax\else
    \def\@currentlabelname{\current@theorem}%
  \fi
  \ifx\cref@currentlabel\undefined\else
    \let\@currentlabelname\cref@currentlabel
  \fi
  \orig@label{#1}%
  \endgroup
}

\makeatother

\begin{document}

\maketitle

\begin{abstract}

Many algorithms in convex optimization and variational analysis 
can be analyzed using Fej\'er monotone sequences. 
In 2024, Behling, Bello-Cruz, Iusem, Alves Ribeiro, and Santos 
introduced a new, more general, notion: Fej\'er* monotonicity. 
They obtained basic results and discussed applications in 
optimization. 

In this work, we complement Behling et al.'s work by presenting  
a thorough study of \fejer* monotonicity. 
We reveal striking similarities and differences between
these notions, including descriptions of the maximal \fejer* set. 
Moreover, we also touch upon Opial sequences and quasi-Fej\'er monotonicity. 
Throughout this paper, we provide numerous limiting examples 
and counterexamples. 

\end{abstract}
{ 
\small
\noindent
{\bfseries 2020 Mathematics Subject Classification:}
{Primary 
47H09, 
47J26, 
90C25; 
Secondary 
47H05, 
65K05. 
%
}

\fejer* monotone sequence,
Opial sequence, 
quasi \fejer\ monotone sequence. 
}

\section{Introduction}

\label{section1}

\subsection*{Motivation} 

Throughout, we assume that 
\begin{empheq}[box=\mybluebox]{equation*}
X \; \text{is a real Hilbert space}
\end{empheq} with inner product $\langle \cdot,\cdot \rangle$ and induced norm $\|\cdot\|$. 

Many convex optimization and monotone inclusion problems formulated 
in $X$ are solved by some (often splitting) algorithms that 
approximate a solution by generating a 
\emph{\fejer\ monotone sequence} (see, e.g., \cite{BC2017} 
and \cite{Comb}). 
Let us recall the relevant definition: 

\begin{definition}[\fejer\ monotonicity] 
\label{def-classic} Let $C$ be a nonempty subset of $X$. Then $(x_n)_{n \in \NN}$ is \emph{\fejer\ monotone with respect to} $C$ if
\begin{equation*}
    (\forall y\in C) (\forall n \in \NN) \quad
        \|x_{n+1} - y\| \leq \|x_n - y\|.
\end{equation*}
\end{definition}

Very recently, again in the study of optimization algorithms (see 
\cite{BBCIARS} and \cite{BBCILS}), 
Behling et al.\ proposed the following relaxed version: 

\begin{definition}[\fejer* monotonicity] \label{def} Let $M$ be a nonempty subset of $X$. We say that $(x_n)_{n \in \NN}$ is \emph{\fejer* monotone with respect to} $M$ if
\begin{equation*}
    (\forall y\in M) (\exists N(y) \in \NN) (\forall n \geq N(y)) \quad
        \|x_{n+1} - y\| \leq \|x_n - y\|.
\end{equation*}
\end{definition} 

\begin{remark}\label{sub-def}
According to \cref{def} (resp.\ \cref{def-classic}), if the sequence $(x_n)_{n \in \mathbb{N}}$ is \fejer* monotone with respect to a nonempty set, then every subsequence $(x_{n_k})_{k \in \mathbb{N}}$ also retains \fejer* (resp.\ \fejer) monotonicity with respect to the same set.
\end{remark} 

While \cite{BBCIARS} and \cite{BBCILS} contain some important results 
concerning \fejer* monotone sequence, a systematic comparison 
between the standard and the new notion is still missing, and our aim
is to close this gap: 

\emph{The goal of this paper is to systematically analyze \fejer* monotonicity, with a particular focus on comparison to the standard 
\fejer\ monotonicity. We also discuss quasi \fejer\ monotonicity as 
well as the recently introduced Opial property. Throughout, we 
present limiting examples that illustrate the sharpness of our results.}

\subsection*{Standing notation}

We now introduce notation that will be used consistently throughout the entire paper.
Given a sequence of points $(x_n)_{n \in \NN}$ in $X$, 
define a sequence of closed convex sets
by 
\begin{empheq}[box=\mybluebox]{equation}
\label{hs}
(\forall\nnn)\quad    C_n := \{y \in X \, | \, \|x_{n+1} - y\| \leq \|x_n - y\| \}. 
\end{empheq} 
For this sequence of closed convex sets $(C_n)_{n\in\NN}$, consider the  
\emph{limit inferior} of this sequence, defined by 
\begin{empheq}[box=\mybluebox]{equation}\label{linf}
    \underline{\lim} \, C_\NN  := 
    \bigcup_{N \in \NN}\Big(\bigcap_{n\geq N} C_n \Big).
\end{empheq} 
For a sequence $(x_n)_{n \in \mathbb{N}}$ of points in $X$, its sets of weak and strong cluster points are denoted by $\mathcal{W}((x_n)_{n \in \mathbb{N}})$ and $\mathcal{S}((x_n)_{n \in \mathbb{N}})$, respectively.

\subsection*{Summary of main results and limiting examples}

Let us now highlight some of the 
\textbf{new results and examples} obtained in this paper:

\begin{description}[labelindent=5pt, itemsep=3pt]
\item[(largest possible \fejer* set)] 
In \cref{eq-def}, we demonstrate that 
$\underline{\lim} \, C_\NN $
is the largest possible set with respect to which $(x_n)_\nnn$ 
is \fejer* monotone. 
(In contrast, for a standard \fejer\ monotone sequence, the largest 
possible set is $\bigcap_\nnn C_n$, which is closed.)
The set $\underline{\lim} \, C_\NN $ may fail to be closed 
(see \cref{closure_example}). 
\item[(striking differences between \fejer\ and \fejer* sequences)]
It is known that a \fejer\ monotone sequence $(x_n)_\nnn$ 
with respect to a nonempty closed convex subset $C$ of $X$ has 
a \emph{decreasing} sequence of distances $(d_C(x_n))_\nnn$, 
the shadow sequence $(P_C(x_n))_\nnn$ is \emph{strongly convergent}, 
and  that if $C$ is a closed affine subspace and $(x_n)_\nnn$ 
converges weakly to a point in $C$, 
then $x_n\weakly P_C(x_0)$. 
All this goes spectacularly wrong for \fejer* sequences:
suppose that $(x_n)_\nnn$ is \fejer* monotone with respect to 
a convex subset $M$ of $X$. 
Then $(d_{\overline{M}}(x_n))_\nnn$ may 
eventually be \emph{strictly increasing}  
(see \cref{example-non-decr}), 
$(P_{\overline{M}}(x_n))_\nnn$ may fail 
to converge weakly (see \cref{infdim-countex}\cref{infdim-countex-prop3}), 
and even if $\overline{M}$ is affine and $(x_n)_\nnn$ converges weakly
to some point in $\overline{M}$, 
then $(x_n)_\nnn$ may fail to converge weakly to 
$P_{\overline{M}}(x_0)$ 
(see \cref{aff-count}). 
\item[(\fejer* sequences behave better when the \fejer* set has nonempty 
relative interior)]\ \\ 
Suppose that $(x_n)_\nnn$ is \fejer* monotone with respect to 
a nonempty convex subset $M$ of $X$. 
We assume that furthermore $\ri(M)\neq\varnothing$, 
which turns out to lead 
to much better properties of $(x_n)_\nnn$: 
If $K$ is a nonempty compact subset of $\ri(M)$, 
then $(x_n)_\nnn$ is eventually \fejer\ monotone 
with respect to $K$ (see \cref{th-regular}); 
however, this fails if $K$ is assumed to be merely 
weakly compact (see \cref{int-nonempty-inf-example}). 
If $C$ is a nonempty closed convex subset of $\caff(M)$, 
then the shadow sequence $(P_C(x_n))_\nnn$ has a 
finite-length trajectory and thus converges strongly 
(see \cref{raik-*}) --- if $\ri(M)=\varnothing$, 
then the conclusion may fail to hold (see \cref{infdim-countex}). 
We stress that \cref{raik-*} appears to be new even in the standard \fejer\ setting! 
If $x_n\to z\in\ri(M)$, then the sequence of distances 
$(d_{\overline{M}}(x_n))_\nnn$ \emph{almost} behaves like 
the standard \fejer\ setting in the sense that 
$(d_{\overline{M}}(x_n))_\nnn$ is \emph{eventually} decreasing 
(see  \cref{c:251213b}). 
Relatedly, if $x_n\to z\in\caff(M)$, then 
eventually either $x_n$ is never equal to $z$ 
or always equal to $z$ (see \cref{cycl-prop-inf*-2cases}). 
\item[(directional asymptotic results for \fejer* sequences)] 
Suppose that $(x_n)_\nnn$ is \fejer* monotone with respect to a 
nonempty convex subset and suppose that 
$x_n\to z\in X$. 
Set $\mathcal{K} = \overline{\operatorname{cone}}(\overline{M} - z)$.  
If $(x_n)_\nnn$ has distinct consecutive terms, 
then $\mathcal{W}((x_n-x_{n+1})/\|x_n-x_{n+1}\|) 
\subseteq \mathcal{K}^\ominus$; 
while if $x_n\neq z$ for all $n$, we have 
then $\mathcal{W}((x_n-z)/\|x_n-z\|) 
\subseteq \mathcal{K}^\ominus$
(see \cref{direct-th}), which extends a known result 
in the \fejer\ setting. 
If $X$ is finite-dimensional 
and $z\in\operatorname{aff}(M)$, 
then $x_n\neq z$ eventually and 
$d_{\mathcal{K}}(x_n-z)/\|x_n-z\|\to 1$ 
(see \cref{equivalency-th}). 
If $z\in\caff(M)$ and $\ri(M)\neq\varnothing$, 
then there exists a constant $\Gamma>0$ such that 
eventually $\|x_n-z\|\leq\Gamma d_{\mathcal{K}}(x_n-z)$
 (see \cref{equiv-th-inf}). 
These results are fairly tight as the limiting examples 
\cref{int-nonempty-inf-example} and 
\cref{example-no_ineq} illustrate. 
We also present an ``exclusion'' result: 
if $(x_n)_\nnn$ has distinct consecutive terms, 
$\ri(M)\neq\varnothing$, and $z\in\caff(M)$, then
eventually 
$x_n\notin z+\mathcal{K}$ 
(see \cref{cycl-prop-inf*2}). 
These techniques can be put to good use to 
obtain linear convergence results (see \cref{QR-lin-remark}), 
and to 
provide a justification for a result published in the literature
with a somewhat obscure proof (see \cref{r:patch}).  
\item[(descriptions of the maximal \fejer* and Opial sets)] 
Suppose that $(x_n)_\nnn$ is \fejer* monotone 
with respect to some convex subset $M$ of $X$, 
and with distinct consecutive terms. 
If $\inte(M)\neq\varnothing$, then 
it is possible to describe $ \underline{\lim} \, C_\NN$, 
up to closure and inteior,
through the polar cone of all weak limits of 
the normalized sequence $((x_n-x_{n+1})/\|x_n-x_{n+1}\|)_\nnn$
(see \cref{lemma-cone-inf}). 
A similar result is presented for the maximal Opial set 
(see \cref{max-opial-set}). 
\item[(\fejer* vs quasi-\fejer\ types)] 
Let $(x_n)_\nnn$ be \fejer* monotone with respect to some 
nonempty convex subset $M$. 
It was previously known that $(x_n)_\nnn$ is automatically 
quasi-\fejer\ of Type~III. 
When $M$ is bounded and $\ri(M)\neq\varnothing$, 
then we show that $(x_n)_\nnn$ is 
quasi-\fejer\ of Type~II (see \cref{*implies2}); moreover, 
this result is sharp 
(see \cref{count-type2-prop} and \cref{prop-ind-count-quasi}). 
If $\inte(M)\neq\varnothing$, then
$(x_n)_\nnn$ is 
quasi-\fejer\ of Type~I (see \cref{nonep-int-type1}).
In fact, \fejer* monotonicity implies Type I under substantially more general assumptions (see \cref{*implies1}); in particular, there is only a very narrow window in which a \fejer*  monotone sequence fails to be quasi-\fejer\ Type I provided that $\ri(M) \neq \varnothing$. As \cref{c:*implies1}, we highlight that
if $M$ is closed and $\ri(M)\neq\varnothing$, then
$(x_n)_\nnn$ is 
quasi-\fejer\ of Type~I. 
\end{description}

\subsection*{Organization of the paper and notation}

The remainder of this paper is organized as follows. 
In \cref{s:aux}, we review Opial sequences, present the proof 
for the maximality of $\underline{\lim} \, C_\NN$ in the \fejer* context,
 discuss the relative interior, and recall properties of (solid and pointed) cones.
We construct various examples highlighting the stark differences between \fejer\ and \fejer* sequences in \cref{section2-f}. 
If the underlying \fejer* set has \emph{nonempty} 
relative interior, then we are able to obtain various 
results more in line with the standard \fejer\ setting in 
\cref{section2}. 
The topic of 
\cref{section4} are various quantitative estimates involving 
\fejer* monotone sequences. 
\cref{section3} 
presents descriptions of maximal \fejer* and Opial sets. 
In \cref{section6}, we discuss the relationships between 
\fejer, \fejer*, and quasi-\fejer\ monotonicity. 
Finally, we include some appendices containing details for some of the examples 
presented to make the manuscript more concise and readable. 

Our notation is fairly standard and follows largely \cite{BC2017}. 
We point out that 
the closed ball of radius $\rho \in \mathbb{R}_{++}$ centered at $x \in X$ is denoted by $B[x, \rho]$. 

\section{Auxiliary results}

\label{s:aux}

In this section, we briefly review properties of Opial sequences 
(introduced recently in \cite{Opial}), 
compare these to \fejer* monotone sequence, 
 present the largest possible \fejer* and \fejer\ sets, 
 discuss the relative interior as well as 
 (solid and pointed) cones. 

\subsection*{Opial sequences}

\begin{definition}[Opial sequence and Opial set]
    Let $C$ be a nonempty subset of $X$. We say that a sequence $(x_n)_{n \in \NN}$ in $X$ is Opial with respect to $C$ if \begin{equation*}
        (\forall y \in C) \quad \lim_{n \to \infty} \|x_n - y\| \; \text{exists}.
    \end{equation*} We occasionally refer to C as the corresponding Opial set. 
\end{definition}

\begin{proposition}[basic properties] \label{b-prop}
    Suppose that $(x_n)_{n \in \mathbb{N}}$ is \fejer* monotone with respect to a nonempty subset $M$ of $X$. Then the following hold: \begin{enumerate}
        \item\label{b-prop1} $(x_n)_{n \in \mathbb{N}}$ is bounded.
        \item\label{b-prop2} $(x_n)_{n \in \mathbb{N}}$ is Opial with respect to $M$.
        \item $(x_n)_{n \in \mathbb{N}}$ is \fejer* monotone with respect to $\operatorname{conv}(M)$.
    \end{enumerate}
\end{proposition} \begin{proof}
    See \cite[Proposition~3.1]{BBCIARS}.
\end{proof}

\begin{remark}\label{opial-prop}
   According to \cref{b-prop}\cref{b-prop2}, all the Opial properties discussed in \cite{Opial} are naturally inherited by \fejer* sequences. Specifically, both Opial's Lemma and its strong convergence counterpart (see \cref{opial-lemma-weak} and \cref{opial-lemma-strong} below) hold in the context of \fejer* monotonicity, just as they are known to apply in the classical \fejer\ setting (see \cite[Theorem 5.5]{BC2017} and \cite[Theorem 5.11]{BC2017}, respectively).
\end{remark}

We now briefly review some properties of Opial sequences relevant to the analysis in this paper. For more details on Opial sequences, see \cite{Opial}.

\begin{fact}[Opial's Lemma]\label{opial-lemma-weak}
    Let $(x_n)_{n \in \NN}$ be a sequence in $X$ that is Opial with respect to a nonempty subset $C$ of $X$. Then $(x_n)_{n \in \NN}$ converges weakly to some point in $C$ if and only if $$\mathcal{W}((x_n)_{n \in \mathbb{N}}) \subseteq C.$$
\end{fact}
\begin{proof}
    See \cite[Corollary 2.3]{Opial}.
\end{proof}

\begin{fact}[strong convergence of Opial sequences] 
\label{opial-lemma-strong}
    Let $(x_n)_{n \in \NN}$ be a sequence in $X$ that is Opial with respect to a nonempty subset $C$ of $X$. Then $(x_n)_{n \in \NN}$ converges strongly to some point in $C$ if and only if $$\mathcal{S}((x_n)_{n \in \mathbb{N}}) \cap C \neq \varnothing.$$
\end{fact}
\begin{proof}
    See \cite[Proposition 2.4]{Opial}.
\end{proof}

\begin{fact}[extending the Opial set] \label{opial-ext-aff}
    Let $(x_n)_{n \in \NN}$ be a sequence in $X$ that is Opial with respect to a nonempty subset $C$ of $X$. Then $(x_n)_{n \in \NN}$ is also Opial with respect to $\overline{\aff}(C)$.
\end{fact}

\begin{proof}
See \cite[Proposition 2.8]{Opial}.
\end{proof}

\subsection*{Alternative description of $C_n$}

We now provide a result that yields alternative 
descriptions of the sets $C_n$ defined in \cref{hs}: 

\begin{lemma}\label{hsp-lemma}
    Let $x, x_+ \in X$. Define the set \begin{align*}
    C(x, x_+) :&= \{ y \in X \mid \|x_+ - y\| \leq \|x - y\| \} \\  &= \{ y \in X \mid \langle x_+ - x, x_+ + x - 2y \rangle \leq 0 \} \\ &= \{ y \in X \mid \langle y, x_+ - x \rangle \geq (\|x_+\|^2 - \|x\|^2)/2\}
    \end{align*} If $x \neq x_+$, then $C(x, x_+)$ is a closed halfspace with its boundary given as the hyperplane \begin{align*} \operatorname{bdry}(C(x, x_+)) &= \{ y \in X \mid \|x_+ - y\| = \|x - y\| \} \\ &= \{ y \in X \mid \langle x_+ - x, x_+ + x - 2y \rangle = 0 \} \\ &= \{ y \in X \mid \langle y, x_+ - x \rangle = (\|x_+\|^2 - \|x\|^2)/2\},
    \end{align*} which passes through the midpoint $(x + x_+)/2$ and is orthogonal to the difference $x_+ - x$. If $x = x_+$, then $C(x, x_+) = X$.
\end{lemma}

\begin{proof}
    Note that for all $v, u \in X$, we have $\|v\|^2 - \|u\|^2 = \langle v - u, v + u \rangle$. In particular, we derive \begin{align*}
        \|x_+ - y\|^2 - \|x - y\|^2 &= \langle (x_+ - y) - (x - y), x_+ - y + x - y \rangle \\ &= \langle x_+ - x, x_+ + x - 2y \rangle.
    \end{align*} Due to that, we conclude \begin{equation*}
        C(x, x_+) = \{ y \in X \mid \langle x_+ - x, x_+ + x - 2y \rangle \leq 0 \},
    \end{equation*} which consequently implies all the other conclusions. \end{proof}

\begin{remark}\label{hsp-lemma-remark}
    For a given sequence $(x_n)_{n \in \NN}$ and $(C_n)_{n \in \NN}$ as in \cref{hs}, we have $C_n = C(x_n, x_{n+1})$ for all $n \in \NN$ 
    using the notation of \cref{hsp-lemma}. 
\end{remark}

\subsection*{Maximal \fejer* and \fejer\ sets}

\begin{proposition}[the maximal \fejer* set]
\label{eq-def} Let $(x_n)_{n \in \NN}$ be a sequence of points in $X$ and $M$ be a nonempty subset of $X$. Then, $(x_n)_{n \in \NN}$ is \fejer* monotone with respect to $M$ if and only if \begin{equation}\label{maxset_inclusion}
    M \subseteq \underline{\lim} \, C_\NN.
\end{equation} 
Consequently, $\underline{\lim} \, C_\NN$ is the largest possible set with respect to which $(x_n)_{n \in \NN}$ is \fejer* monotone.
\end{proposition}

\begin{proof}
    \cref{maxset_inclusion} follows from rewriting \cref{def} in terms of the sequence $(C_n)_{n\in\NN}$ and the set $\underline{\lim} \, C_\NN$ introduced above in \cref{hs} and \cref{linf}, respectively. 
\end{proof}

\begin{remark}
    We occasionally refer to the set $\underline{\lim} \, C_\NN$ as the \emph{maximal \fejer* set} for the sequence $(x_n)_{n \in \NN}$. If $\underline{\lim} \, C_\NN = \varnothing$, then $(x_n)_{n \in \NN}$ is not \fejer* monotone with respect to any nonempty subset of $X$. Comparing to the standard \fejer\ setting, we note that 
    \begin{equation*} 
    \bigcap_{\nnn} C_n
    \end{equation*} 
    is the largest possible set with respect to which $(x_n)_\nnn$ 
    is \fejer\ monotone (see \cite[Lemma~2.1.(i)]{BKX}).
\end{remark}

\subsection*{Relative interior}

We now recall some definitions of relative interiors and general topological facts about complete metric spaces that become essential when examining the relationship between \fejer* monotone sequences and standard \fejer\ monotonicity in infinite-dimensional spaces.

\begin{definition}[relative interior] \label{ri-defs}
    Let $A$ be a nonempty affine subspace of $X$ and take a convex subset $M$ of $A$. Then denote the interior of $M$ in $A$ induced with the strong topology from $X$ by \begin{equation*}
        \operatorname{int}_{A}(M) := \{x \in M \mid B[x, \rho] \cap A \subseteq M \text{ for some } \rho \in \mathbb{R}_{++}\}.
    \end{equation*} We then define the \emph{relative interior} of $M$ by     \begin{equation*}\ri(M) := \operatorname{int}_{\overline{\aff}(M)}(M).
    \end{equation*} 
    For more details on generalized interiors, see \cite{LoMordNam} and \cite{Zalinescu}.
\end{definition}

\begin{remark}
    If $X$ is finite-dimensional, then, because any affine subspace $A$ of $X$ is closed, $\operatorname{ri}(M)$ coincides with the standard \emph{relative interior} of $M$ given by \begin{equation*}
        \ri(M) = \operatorname{int}_{\aff(M)}(M) = \{x \in M \mid B[x, \rho] \cap \aff(M) \subseteq M \text{ for some } \rho \in \mathbb{R}_{++}\};
    \end{equation*} 
    moreover, $\ri(M)\neq \varnothing$ provided that $M$ is convex 
    (see \cite[Theorem 6.2]{Rockafellar}).
\end{remark}

\begin{fact}[Ursescu]\label{ursescu}
    Let $Y$ be a complete metric space. Suppose that $(Q_n)_{n \in \NN}$ is a sequence of closed subsets of $Y$. Then \begin{equation*}
        \overline{\bigcup_{n \in \NN} \operatorname{int}(Q_n)} = \overline{\operatorname{int}\Big(\bigcup_{n \in \NN}Q_n\Big)}.
    \end{equation*}
\end{fact}

\begin{proof}
    See \cite[Lemma 1.44(i)]{BC2017}. 
\end{proof}

\begin{corollary}\label{incr_inter_l-inf}
    Let $(x_n)_{n \in \NN}$ to be a sequence of points and $A$ be a closed affine subspace of $X$. Then 
    \begin{equation}\label{incr_inter_l-inf-eq}
        \bigcup_{n \in \NN} \operatorname{int}_{A}\Big(\bigcap_{n \geq N} C_n\Big) = \operatorname{int}_{A}\big(\underline{\lim} \, C_\NN\big).
    \end{equation}
\end{corollary}

\begin{proof}
    Indeed, $A$, with the induced strong topology from $X$, represents a complete metric space, whereas $(Q_n)_{n \in \NN}$, where $Q_N := (\cap_{n \geq N} C_n) \cap A$, is an increasing sequence of embedded closed convex subset of $A$. That said, due to \cref{linf}, we have \begin{equation*}
        \cup_{n \in \NN} Q_n = \underline{\lim} \, C_\NN \cap A.
    \end{equation*} However, in light of \cref{ursescu}, we conclude \begin{equation*}
        \overline{\cup_{n \in \NN} \operatorname{int}_{A}(Q_n)} = \overline{\operatorname{int}_{A}(\underline{\lim} \, C_\NN \cap A)}.
    \end{equation*} Finally, \cref{incr_inter_l-inf-eq} follows since $\operatorname{int}_{A}(Q \cap A) = \operatorname{int}_{A}(Q)$ for any $Q \subseteq X$ and the equality of the closures of two convex sets necessarily implies that their interiors coincide.
\end{proof}

\subsection*{Cones}

\begin{fact}[Moreau]\label{moreau_decomp}
    Let $K$ be a nonempty closed convex cone in $X$ and let $x \in X$. Then the following hold: \begin{enumerate}
        \item $x = P_{K}(x) + P_{{K}^{\ominus}}(x)$.
        \item $P_{{K}}(x) \perp P_{{K}^{\ominus}}(x)$.
    \end{enumerate}
\end{fact}
\begin{proof}
    See \cite[Theorem~6.30]{BC2017}.
\end{proof}

\begin{definition}[solid and pointed cones]
Let $K$ be a nonempty closed convex cone in $X$. We say that 
\begin{enumerate}
\item $K$ is \emph{solid} if $\inte K \neq\varnothing$;
\item $K$ is \emph{pointed} if $K \cap (-K)=\{0\}$. 
\end{enumerate}
\end{definition}

The following result is part of the folklore; for the reader's convenience, 
we include the short proof.
\begin{fact}
Let $K$ be a nonempty closed convex cone in $X$. 
Then we have the implications
\begin{equation}
\label{e:251211b}
K \text{ is solid }
\Rightarrow 
K-K = X
\Rightarrow 
\overline{K-K}=X 
\Rightarrow 
K^\ominus \text{ is pointed,} 
\end{equation}
and 
\begin{equation}
\label{e:251211c}
K \text{ is pointed}
\Rightarrow 
\overline{K^\ominus-K^\ominus} = X, 
\end{equation}
and 
\begin{equation}
\label{e:251212a}
\text{if $X$ is finite-dimensional, then: }
K \text{ is solid }
\Leftrightarrow 
K-K = X
\Leftrightarrow 
\overline{K-K}=X. 
\end{equation}
\end{fact}
\begin{proof}
Suppose first that $K$ is solid. 
Let $x\in X$ and $k_0\in\inte(K)$. 
Then there exists $\delta>0$ such that 
$k_0 + \delta x =: k_1 \in K$. 
Hence $x = k_1/\delta - k_0/\delta \in K-K$ and 
thus $X\subseteq K-K\subseteq X$. 

Now suppose that $K-K=X$.  
Then $\overline{K-K}=\overline{X}=X$.

Now suppose that $\overline{K-K}=X$.
Let $u\in K^\ominus \cap (-K^\ominus) = K^\perp$. 
Let $\varepsilon>0$. Then 
there exists $k_1,k_2$ in $K$ such that 
$\|u-(k_1-k_2)\| < \varepsilon$. 
Then $\|u\|^2=\scal{u}{u} = 
\scal{u-(k_1-k_2)}{u} + \scal{k_1-k_2}{u} 
\leq \|u-(k_1-k_2)\|\|u\|\leq \varepsilon\|u\|$.
Because $\varepsilon>0$ was chosen arbitrarily, we deduce that 
$u=0$. Therefore, $K^\ominus$ is pointed. 
We have thus proved \cref{e:251211b}.

Now assume that $K$ is pointed.
Then $(K^\ominus)^\perp = K^{\ominus\ominus} \cap -K^{\ominus\ominus} 
= K \cap -K = \{0\}$. 
Hence 
$(K^\ominus)^{\perp\perp} = X$ 
$\Leftrightarrow$
$\overline{\lspan(K^\ominus)} = X$
$\Leftrightarrow$
$\overline{K^\ominus-K^\ominus}=X$. 
We've verified \cref{e:251211c}. 

Finally, assume that $X$ is finite-dimensional and 
$\overline{K-K} = X$. 
By \cite[Corolllary~6.6.2 and Corollary~6.3.1]{Rockafellar}, 
$\ri(K)-\ri(K)=\ri(K-K)=\ri(X)=X$. 
Hence $K-K=X$. 
Furthermore, letting $A$ be the affine hull of $K$, we have 
$X = \ri(K)-\ri(K)\subseteq A-A \subseteq X$.
Thus $A=X$ and $\varnothing \neq \ri(K)=\inte(K)$, i.e., $K$ is solid.
\end{proof}

Next, we provide a characterization of the interior of 
the dual cone $K^\oplus = -K^\ominus$:

\begin{proposition}[interior of the dual cone]\label{int-dual-cone-prop}
Let $K$ be a nonempty closed convex cone in $X$. 
Then 
\begin{equation}\label{int-dual-cone}
\inte(K^\oplus) = 
\menge{u\in X}{(\exists \delta>0)(\forall k\in K)\;\; 
\scal{k}{u} \geq \delta\|k\|}.
\end{equation}
\end{proposition}
\begin{proof}
Indeed, for $u\in X$, we obtain 
\begin{align*}
u\in \inte(K^\oplus)
&\Leftrightarrow 
(\exists \delta>0)\;\; B[u,\delta]\subseteq K^\oplus\\
&\Leftrightarrow 
(\exists \delta>0)\;\; u+B[0,\delta]\subseteq K^\oplus\\
&\Leftrightarrow 
(\exists \delta>0)(\forall b\in B[0,\delta])(\forall k\in K)\;\;
\scal{k}{u+b}\geq 0\\
&\Leftrightarrow 
(\exists \delta>0)
(\forall k\in K)
(\forall b\in B[0,\delta])
\;\;
\scal{k}{u}\geq \scal{k}{-b}\\
&\Leftrightarrow 
(\exists \delta>0)
(\forall k\in K)
\;\;
\scal{k}{u}\geq \delta\|k\|,
\end{align*}
as claimed.
\end{proof}

Pointedness and solidity 
can be characterized using the unit sphere, which 
we denote by $S_X$:

\begin{proposition}\label{pointed-sph-cond}
\label{p:251214a}
Let $K$ be a nonempty closed convex cone in $X$. 
Then we have the equivalences:
\begin{enumerate}
\item 
\label{p:251214a1}
$K$ is pointed
$\Leftrightarrow$
$0\notin \conv(K\cap S_X)$. 
\item 
\label{p:251214a2}
$K^{\oplus}$ is solid $\Leftrightarrow$ $0 \not \in \overline{\conv}(K \cap S_X)$.
\end{enumerate}
\end{proposition}
\begin{proof}
Abbreviate $S_X$ by $S$. 

\cref{p:251214a1}:
``$\Rightarrow$'': 
We show the contrapositive and thus assume that 
$0\in \conv(K\cap S)$, say 
$0 = \sum_{i\in I} \lambda_i k_i$,
where the $\lambda_i$ are positive and sum up to $1$, and 
each $k_i\in K\cap S$. 
Note that $I$ contains at least two elements.
Now partition $I$ into a singleton subset $I_1$ and 
$I_2 := I\setminus I_1$. 
Then 
$K \ni \sum_{i\in I_1}\lambda_i k_i = -\sum_{i\in I_2}\lambda_i k_i 
\in -K$; 
thus, $K\cap (-K)$ contains the nonzero element $\sum_{i\in I_1}\lambda_i k_i$. Hence $K$ is not pointed. 
``$\Leftarrow$'': We again show the contrapositive and assume that 
$K$ is not pointed, say $0\neq k\in K\cap (-K)$.
WLOG $k\in S$ (after scaling if necessary). 
But then 
$0 = \thalb k + \thalb(-k) \in \conv(K\cap S)$. 

\cref{p:251214a2}:
``$\Rightarrow$'': 
Let $u \in \operatorname{int}(K^{\oplus})$. 
Then \cref{int-dual-cone} implies that \begin{equation}\label{pointed_cones-eqint1}
    (\exists \delta>0)(\forall k\in K) \quad  
\scal{k}{u} \geq \delta\|k\|.
\end{equation} 
However, note that, for any convex combination $k = \sum^N_{i = 1} \lambda_i k_i$ with $k_i \in K \cap S$, \cref{pointed_cones-eqint1} yields: \begin{equation}\label{pointed_cones-eqint2}
    \scal{k}{u} = \sum^N_{i = 1} \lambda_i \scal{k_i}{u} \geq \sum^N_{i = 1} \lambda_i \delta = \delta > 0.
\end{equation} 
Furthermore, the inequality \cref{pointed_cones-eqint2} extends to all $k \in \overline{\conv}(K \cap S)$; therefore, 
$0 \notin \overline{\conv}(K \cap S)$.
``$\Leftarrow$'': Note that $\overline{\conv}(K \cap S)$ is 
a nonempty bounded closed convex (weakly compact) subset of $X$ that does not contain 
$0$. By the strong separation theorem 
(see, e.g., \cite[Theorem~3.50]{BC2017}), 
there exists $\delta > 0$ such that 
$0 = \scal{0}{u} < \delta < \min\scal{\overline{\conv}(K \cap S)}{u}$. 
Consequently, 
$(\forall k\in K)$
$\delta\|k\| \leq \scal{k}{u}$. 
Therefore, $u\in\inte(K^{\oplus})$ by  \cref{int-dual-cone-prop}.
\end{proof}

\begin{remark}[finite vs infinite dimensions]\label{pointed-sph-cond-remark}
If $X$ is finite-dimensional, then $\operatorname{conv}(C)$ is closed for any nonempty compact subset C (see \cite[Theorem~17.2]{Rockafellar}), and \cref{pointed-sph-cond} yields the classical duality (see also, e.g., \cite[Theorem~2.3]{Berman}): $K$ is pointed if and only if $K^{\oplus}$ is solid. This equivalence is attributed to Krein and Rutman. 
In contrast, this duality generally fails in infinite-dimensional spaces. Specifically, while the solidness of $K^{\oplus}$ still implies that $K$ is pointed, the converse need not hold. As a counterexample, let $X=\ell_2(\mathbb{N})$ and define
$$
K := \overline{\operatorname{cone}}\operatorname{conv}\{e_i\}_{i\in\mathbb{N}} .
$$
Then $0 \in \overline{\operatorname{conv}}(K \cap S_X)\setminus \operatorname{conv}(K \cap S_X)$; in particular, $K$ is pointed, but $K^{\oplus} = K$ is \emph{not} solid. Also, if we set 
$$K := \menge{(x_n)_\nnn}{x_0\geq x_1\geq \cdots \geq x_n \geq \cdots 
\geq 0},$$ then $K$ is a nonempty closed convex cone in $X$ with 
$\overline{K-K}=X$ but $K-K\neq X$.
\end{remark}


\section{Counterexamples illustrating the failure of standard properties} 

\label{section2-f}

In this section, we provide various examples highlighting 
the stark differences between \fejer\ and \fejer* sequences.

In preparation the first example, we introduce the following notation: \begin{align*}
    \mathbf{e}(\alpha) &:= (\cos(\alpha), \sin(\alpha)), \quad \alpha \in \mathbb{R} \\
    \mathbf{e}(I) &:= \{\mathbf{e}(\alpha) \mid \alpha \in I\}, \quad I \subseteq \mathbb{R} \\
    K(I) &:= \mathbb{R}_{+} \mathbf{e}(I) = \{\rho\,\mathbf{e}(\alpha) \mid \rho \geq 0, \, \alpha \in I\}.
\end{align*}

\begin{example}
    Suppose $X = \mathbb{R}^2$ and let $\alpha < \alpha_+ < \alpha + \pi$. Set $\beta = (\alpha + \alpha_+)/2$. Then \begin{equation*}
        C(\mathbf{e}(\alpha), \mathbf{e}(\alpha_+)) = K([\beta, \beta + \pi]).
    \end{equation*}
\end{example}

\begin{example}\label{closure_example}
    Suppose $X = \mathbb{R}^2$, and pick any scalar sequence $(\alpha_n)_{n \in \NN}$ such that \begin{equation*}
        -\frac{\pi}{2} < \alpha_0 < \alpha_1 < \cdots < \alpha_n < \alpha_{n+1} \rightarrow 0^-.
    \end{equation*} For $n \in \NN$, set $\beta_n := (\alpha_n + \alpha_{n+1})/2$, which is strictly increasing and $\beta_n \to 0^-$. Finally, define \begin{equation*}
        x_n := \mathbf{e}(\alpha_n), \quad n \in \NN.
    \end{equation*} Then the maximal \fejer* set for the sequence $(x_n)_{n \in \NN}$ is given by \begin{equation*}
        \underline{\lim} \, C_\NN = K\left(\left[0, \pi\right)\right) = (\mathbb{R}\times\mathbb{R}_{++}) \cup (\mathbb{R}_+ \times\{0\}),
    \end{equation*} which is not closed.
\end{example}

\begin{proof}
    Indeed, we have \begin{equation*}
        (\forall n \in \NN) \quad C(x_n ,x_{n+1}) = K([\beta_n, \beta_n + \pi]) = K\left(\left[\beta_n, \pi/2 \right]\right) \cup K\left(\left[\pi/2, \beta_n + \pi\right]\right).
    \end{equation*} Here $\left( K\left(\left[\beta_n, \pi/2\right]\right) \right)_{n \in \NN}$ decreases, while $\left( K\left(\left[\pi/2, \beta_n + \pi\right]\right) \right)_{n \in \NN}$ increases; see \cref{fig:examples1a-pic}. Then, for all $N \in \NN$, it follows that \begin{align*}
        \bigcap_{n \geq N} C(x_n, x_{n+1}) &= \bigcap_{n \geq N} K\left(\left[\beta_n, \pi/2\right]\right) \quad \cup \quad  \bigcap_{n \geq N} K\left(\left[\pi/2, \beta_n + \pi\right]\right) \\ &= K\left(\left[0, \pi/2\right]\right) \; \cup \; K\left(\left[\pi/2, \beta_N + \pi\right]\right) \\ &= K\left(\left[0, \beta_N + \pi\right]\right).
    \end{align*} Hence \begin{equation*}
        \underline{\lim} \, C_\NN = \liminf C(x_n, x_{n+1}) = \bigcup_{N \in \NN} K\left(\left[0, \beta_N + \pi\right]\right) = K\left(\left[0, \pi\right)\right),
    \end{equation*} which completes the proof. \end{proof}

    \begin{figure}[t]
    \centering
    \includegraphics[width=.50\linewidth]{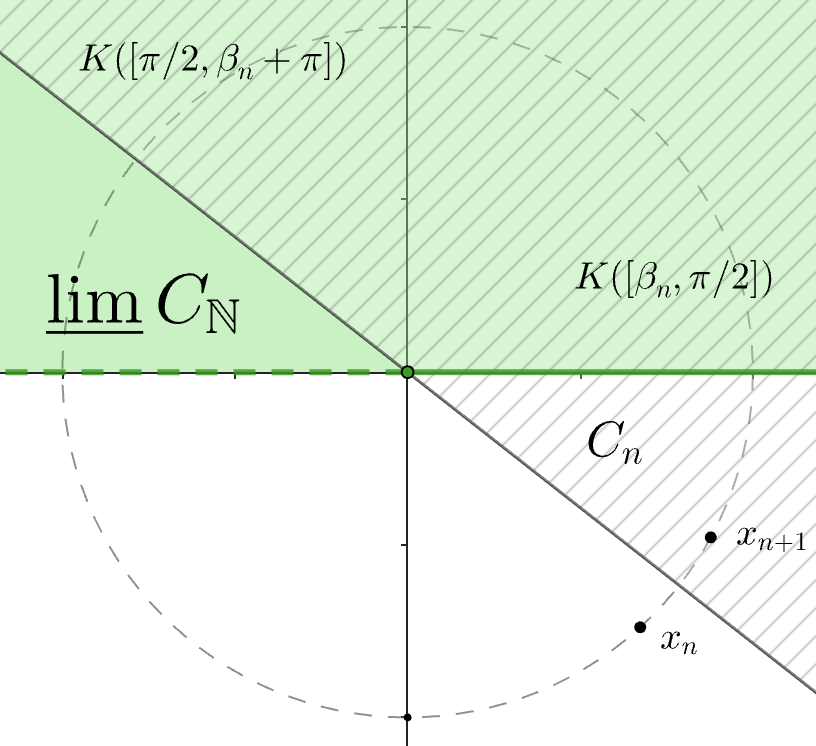}
    \caption{A single iteration step for the sequence $(x_n)_{n \in \NN}$ defined in \cref{example-non-decr}.}
    \label{fig:examples1a-pic}
\end{figure}


\begin{remark}[\fejer* only extends to the convex hull, not to the closed
convex hull]
\label{problem-closure}
    In particular, unlike in the standard \fejer\ setting, the \fejer* property cannot be extended to $\overline{\operatorname{conv}}(M)$, but only to $\operatorname{conv}(M)$. Henceforth, we assume that the set $M$ is always convex. Behling, Bello-Cruz, Iusem, Alves Ribeiro, and Santos provided a more complicated example in \cite{BBCIARS} to demonstrate this issue, and their construction inspired all finite-dimensional counterexamples developed in this paper. For more details, see \cref{authors-example} below.
\end{remark}

To further emphasize the difference between the standard \fejer\ and \fejer* settings, we review other established properties implied by \fejer\ monotonicity when the set in question is a closed affine subspace.

\begin{fact}\label{Proposition 5.9}
    Let $C$ be a closed affine subspace of $X$ and let $(x_n)_{n \in \NN}$ be a sequence in $X$. Suppose that $(x_n)_{n \in \NN}$ is \fejer\ monotone with respect to $C$. Then the following hold: \begin{enumerate}
        \item\label{Proposition 5.9(1)} $(\forall n \in \NN)$ $P_C(x_n) = P_C(x_0)$.
        \item\label{Proposition 5.9(2)} Suppose that every weak sequential cluster point of the sequence $(x_n)_{n \in \NN}$ belongs to $C$. Then $x_n \rightharpoonup P_C(x_0)$.
    \end{enumerate}
\end{fact}

\begin{proof}
    See \cite[Proposition 5.9]{BC2017}.    
\end{proof}

Unfortunately, the analogue of \cref{Proposition 5.9} fails to hold for a \fejer* monotone sequence $(x_n)_{n \in \NN}$ in general, as the following example shows.

\begin{example}[\fejer* is not nice with respect to closed affine subspaces]
\label{aff-count}
Suppose $X = \mathbb{R}^2$ and set \begin{equation}\label{aff-count-defM}
    M := \mathbb{R}\times\{0\}, \quad x_0 := (0, 1).
\end{equation} Then, for all $n \in \NN$, define $x_{n+1}$ inductively as follows: 
    
    Construct the circle $S_n$ centered at the point $z_n :=(-n, 0)$ passing through the point $x_n$. Denote by $y_n$ the rightmost point where $S_n$ intersects $M$. Then, define $x_{n+1} \in S_n$ as the unique point of $S_n$ lying on the internal angle bisector of the angle $\angle x_nz_ny_n$. Equivalently, $x_{n+1}$ is the midpoint of the minor arc connecting $x_n$ and $y_n$ on $S_n$.

    Then the following hold: \begin{enumerate}
        \item\label{aff-count-p1} $(P_M(x_n))_{n \in \NN}$ is strictly increasing in $M$, where $M$ is given by \cref{aff-count-defM} and identified with $\RR$.
        \item\label{aff-count-p2} $(x_n)_{n \in \NN}$ is \fejer* monotone with respect to $M$ and $x_n \to z = (\zeta,0) \in M$, where $\zeta>0$. 
        \item\label{aff-count-p3} $\underline{\lim} \, C_\NN = \mathbb{R}\times\mathbb{R_-}$.
    \end{enumerate}
\end{example}

\begin{figure}[t]
    \centering
    \includegraphics[width=.60\linewidth]{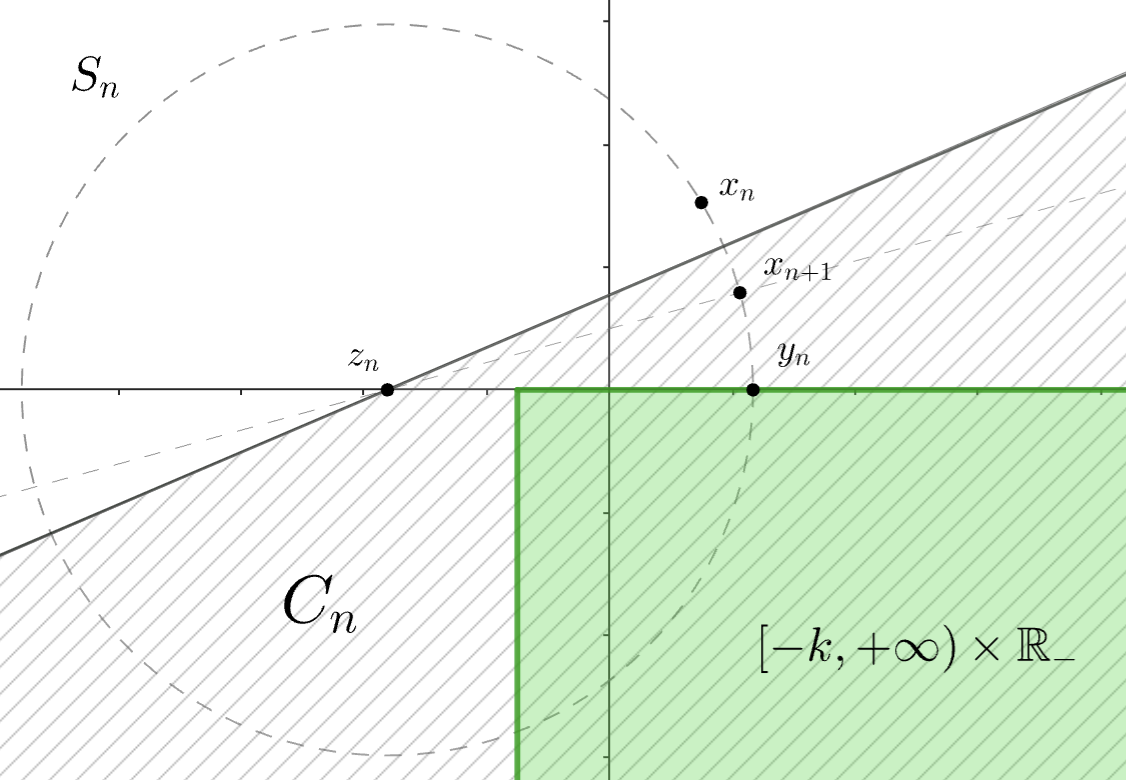}
    \caption{A schematic iteration step for the sequence $(x_n)_{n \in \NN}$ given by \cref{example-non-decr}.}
    \label{fig:examples1b-pic}
\end{figure}

\begin{proof}
    See Appendix~\ref{app-aff-count-p}.
\end{proof}
    
\begin{remark}\label{fig1-remark}
    \cref{fig:examples1a-pic} and \cref{fig:examples1b-pic} illustrate single iteration steps for the two previously defined sequences, as introduced in \cref{closure_example} and \cref{aff-count}, respectively. The hatched region indicates the halfspace $C_n$ corresponding to a given index $n \in \mathbb{N}$. In \cref{fig:examples1a-pic}, the green region illustrates the maximal \fejer* set $\underline{\lim} \, C_\NN$, while the solid and dashed dark green lines on its boundary indicate which boundary points are included in the sets, as specified in \cref{closure_example}. By contrast, in \cref{fig:examples1b-pic}, the green region denotes the auxiliary set $[-k, +\infty)\times\mathbb{R}_-$, used in the proof for \cref{aff-count}, for some $k \in \mathbb{N}$ with $k \leq n$.
\end{remark}

\begin{remark}
    Although, the affine projections $P_{\overline{\operatorname{aff}}(M)}(x_n)$ may shift and are not necessarily fixed in the \fejer* framework, the fact of their weak convergence, in light of \cref{b-prop}\cref{b-prop2} and \cref{opial-prop}, is quite general and holds for Opial sequences as \cref{OPL} below demonstrates.
\end{remark}

\begin{definition}[asymptotic center]
    Let $(u_n)_{n \in \NN}$ be a bounded sequence in $X$ and let $C$ be a nonempty closed convex subset of $X$. Then there is a unique point $\hat{c} \in C$ such that \begin{equation*}
        (\forall c \in C\setminus\{\hat{c}\}) \quad \overline{\lim_{n}} \|u_n - \hat{c}\| < \overline{\lim_{n}} \|u_n - c\|.
    \end{equation*} The point $\hat{c}$ is called the \emph{asymptotic center} of $(u_n)_{n \in \NN}$ with respect to $C$, and we write $\hat{c} = A_C(u_n)_{n \in \NN}$, or, more succinctly, $\hat{c} = A_C(u_\NN)$.
\end{definition}

\begin{fact}\label{OPL}
    Let $(x_n)_{n \in \NN}$ be an Opial sequence in $X$ with respect to a nonempty subset $M$ of $X$. Then $P_{\overline{\operatorname{aff}}(M)}(x_n)$ is weakly convergent to $A_{\overline{\operatorname{aff}}(M)}(x_\NN)$.
\end{fact}

\begin{proof}
    Combine \cite[Proposition 2.8]{Opial} and \cite[Theorem 3.12]{Opial}.
\end{proof}

\begin{corollary}\label{OPL-cor}
    Let $(x_n)_{n \in \NN}$ be an Opial sequence with respect to a nonempty subset $M$ of $X$. Suppose that every weak sequential cluster point of the sequence $(x_n)_{n \in \NN}$ belongs to $M$. Then $x_n \rightharpoonup A_{\overline{\operatorname{aff}}(M)}(x_\NN)$.
\end{corollary}

\begin{proof}
    Due to \cref{opial-lemma-weak}, we conclude $x_n \rightharpoonup z$ for some $z \in M$. Since the projection $P_{\overline{\operatorname{aff}}(M)}(\cdot)$ is weakly continuous on $X$, we derive that $$P_{\overline{\operatorname{aff}}(M)}(x_n) \rightharpoonup P_{\overline{\operatorname{aff}}(M)}(z) = z.$$ However, due to \cref{OPL}, it is known that $P_{\overline{\operatorname{aff}}(M)}(x_n) \rightharpoonup A_{\overline{\operatorname{aff}}(M)}(x_\NN)$.
\end{proof}

Next, we summarize key facts about the convergence and monotonicity of the distance sequence $(d_{C}(x_n))_{n \in \mathbb{N}}$ for standard \fejer\ and Opial sequences. Afterwards, we highlight what kinds of behaviors may occur, and what can plausibly be anticipated in the \fejer* framework.

\begin{fact}
    Let $C$ be a nonempty convex subset of $X$ and let $(x_n)_{n \in \NN}$ be a \fejer\ monotone sequence with respect to $C$. Then the distances $(d_{C}(x_n))_{n \in \mathbb{N}}$ decrease and, therefore, always converge.
\end{fact}

\begin{proof}
    See \cite[Proposition 5.4(iii)]{BC2017}.
\end{proof}

\begin{fact}\label{conver-dist}
Let $(x_n)_{n \in \NN}$ be a sequence in $X$ that is Opial with respect to a nonempty closed, convex subset $C$ of X. Suppose 
$\{P_{C}(x_n)\}_{n \in \NN}$ is relatively compact. Then $(d_C(x_n))_{n \in \NN}$ is convergent and
\begin{equation*}\label{conver-dist-details}
    \lim_{n \to \infty}d_C(x_n) = \lim_{n \to \infty}\|x_n - A_C(x_{\NN})\| = \min_{c \in C}\lim_{n \to \infty}\|x_n - c\|.
\end{equation*}
\end{fact}

\begin{proof}
    See \cite[Proposition 3.10]{Opial}.
\end{proof}

\begin{remark} As \cref{example-non-decr} below demonstrates, the scalar sequence $(d_{\overline{M}}(x_n))_{n \in \NN}$ may fail to be eventually decreasing in the \fejer* framework, even in finite dimensions.
\end{remark}

\begin{example}[distances may strictly \emph{increase} for \fejer*]
\label{example-non-decr}
    Pick a positive, strictly decreasing scalar sequence $(a_n)_{n \in \NN}$ such that $a_n \to 0$. Choose any $x_0 \in \mathbb{R}_{++}^2$. Then, depending on the given index $n \in \NN$, define recurrently \begin{equation}\label{example-non-decr-def}
        R_n := \|x_n + (a_n, 0)\| \quad \text{and} \quad x_{n + 1} := \left(a_n2^{-n}, \sqrt{R^2_n - a^2_n2^{-2n}}\right).
    \end{equation} Equivalently, we define $x_{n+1}$ to be the unique point located in $\mathbb{R}_{++}^2$ on the circle $S_n$, which is centered at $z_n := (-a_n, 0)$ and passes through $x_n$, so that $d_{\{0\}\times\mathbb{R}}(x_{n+1}) = a_n 2^{-n}$ (see \cref{fig:examples5-pic}).

    Define $M:=\mathbb{R}_{--}^2$. Then the following hold: \begin{enumerate}
        \item\label{example-non-decr-prop1} $(x_n)_{n \in \NN}$ is \fejer* monotone with respect to $M$. Moreover, \begin{equation*}
            x_n \to z \in \{0\}\times\mathbb{R}_{++} \quad \text{and} \quad\underline{\lim} \, C_\NN = \left(\mathbb{R}_{--}\times\mathbb{R}\right) \, \cup \, \left(\{0\}\times[\|z\|, +\infty)\right).
        \end{equation*}
        \item\label{example-non-decr-prop2} $(d_{\overline{M}}(x_n))_{n \in \NN}$ is strictly increasing. More specifically, $d_{\overline{M}}(x_n) = \|x_n\| \nearrow \|z\|$.
    \end{enumerate}
\end{example}

\begin{figure}[t]
    \centering
    \includegraphics[width=.60\linewidth]{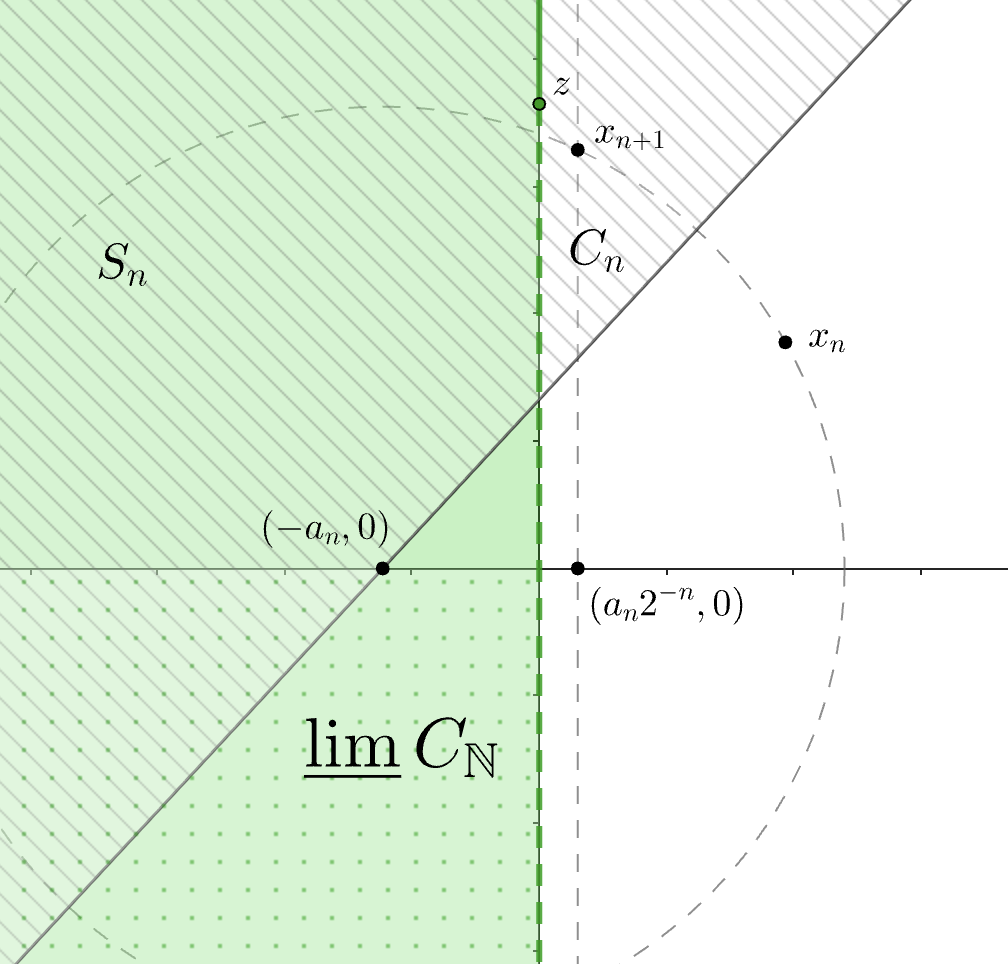}
    \caption{Intermediate iterations of the sequence $(x_n)_{n \in \NN}$ defined in \cref{example-non-decr}.}
    \label{fig:examples5-pic}
\end{figure}

\begin{proof}
    See Appendix~\ref{app-example-non-decr-prop}.
\end{proof}

\begin{remark}
    \cref{fig:examples5-pic} shows two iterations of $(x_n)_{n \in \NN}$, as introduced in \cref{example-non-decr}. Similar to \cref{fig1-remark}, the hatched region represents the halfspace $C_{n}$. As in \cref{fig:examples1a-pic}, the green region illustrates the maximal \fejer* set $\underline{\lim} \, C_\NN$, while the solid and dashed dark green lines on its boundary indicate which boundary points are included in the sets, as specified in \cref{example-non-decr}. The area filled with green dots represents the chosen set $M = \mathbb{R}_{--}^2$.
\end{remark}

\begin{remark}
    \cref{example-non-decr} can be adapted by intermittently introducing descending terms of the form $x_{n+1} = \delta x_n$, where $\delta \in (0, 1)$ is suitably chosen. Doing so enables the construction of a sequence $(x_n)_{n \in \mathbb{N}}$ converging to the origin $(0, 0) \in \overline{M} \setminus M$ with the distances $(d_{\overline{M}}(x_n))_{n \in \NN}$ failing to be eventually neither monotone increasing nor decreasing. 
\end{remark}

If the orbit of the shadow sequence $(P_{C}(x_n))_{n \in \mathbb{N}}$ is not relatively compact in the setting of \cref{conver-dist}, then, as shown by \cref{infdim-countex}, the distances $(d_C(x_n))_{n \in \mathbb{N}}$ may fail to converge in general. This infinite-dimensional example, which is essentially a variant of Zarantonello’s example \cite{Zarantonello} adjusted to our purposes, was originally introduced in \cite{Opial} to demonstrate the failure of $(P_{C}(x_n))_{n \in \mathbb{N}}$ to converge even weakly for Opial sequences. But prior to \cref{infdim-countex}, to additionally approach this issue in the \fejer* case, we highlight \cref{fejer-strong-shadow} and \cref{opial-weak-or-strong-shadow}; also, 
recall previously-mentioned \cref{OPL}.

\begin{fact}\label{fejer-strong-shadow}
    Let $(x_n)_{n \in \NN}$ be a sequence of points in $X$ that is \fejer\ monotone with respect to a nonempty closed convex subset $C$ of $X$. Then the shadow sequence $(P_C(x_n))_{n \in \NN}$ converges strongly to a point $z$ in $C$.
\end{fact}

\begin{proof}
    See \cite[Proposition 5.7]{BC2017}.
\end{proof}

\begin{fact}\label{opial-weak-or-strong-shadow}
    Let $(x_n)_{n \in \NN}$ be a sequence in $X$ that is Opial with respect to a nonempty closed convex subset $C$ of $X$. Suppose that $(P_C(x_n))_{n \in \NN}$ is relatively compact. Then $P_C(x_n) \to A_C(x_\NN)$.
\end{fact}

\begin{proof}
    See \cite[Corollary 3.8]{Opial}.
\end{proof}

\begin{example}[shadows and distance may fail to converge for \fejer*]
\label{infdim-countex}
    Suppose that $X = \ell^2$ with the standard basis of unit vectors $(e_n)_{n \in \NN}$. Define the sequence $(x_n)_{n \in \NN}$ in $X$ by \begin{equation}
    \label{e:251213a}
        (\forall n\in \NN) \quad x_n := \begin{cases} e_1 + e_{n/2}, & \text { if } n \text{ is even}; \\ e_0 + e_1, & \text { if } n \text{ is odd}.\end{cases}
    \end{equation}
    Define $Y := \{ e_0 \}^{\perp}$, $C:= B[0, 1]\cap Y$ and set \begin{equation*}
        M := \operatorname{span}((e_{n+1})_{n \in \NN}) = \{\lambda_1e_1 + \lambda_2e_2 + \cdots + \lambda_me_m \mid m \in \NN_{++}, \; \lambda_1, \ldots , \lambda_m \in \mathbb{R}\}.
    \end{equation*} Then, for the sequence $(x_n)_{n \in \NN}$ given by \cref{e:251213a}, the following hold: 
    \begin{enumerate}
        \item\label{infdim-countex-prop0} $M = \underline{\lim} \, C_\NN$; 
        consequently, $(x_n)_{\nnn}$ is \fejer* with respect to $M$. 
        \item\label{infdim-countex-prop1} $Y = \overline{M}$ is the maximal Opial set for the sequence $(x_n)_{n \in \NN}$, i.e. \begin{equation*}
            Y = \{y \in X \mid \lim_{n \to \infty} \|x_n - y\| \; \text{exists}\}.
        \end{equation*}
        \item\label{infdim-countex-prop2} $(P_Y(x_n))_{n \in \NN} = (e_1, e_1, 2e_1, e_1, e_1 + e_2, e_1, e_1 + e_3, \ldots)$ converges weakly to $e_1 = A_Y(x_{\NN})$, but the scalar sequence $(d_Y(x_n))_{n \in \NN} = (1, 1, 1, 1, 0, 1, 0, \ldots)$ is not convergent.
        \item\label{infdim-countex-prop3} $(P_C(x_n))_{n \in \NN} = (e_1, e_1, e_1, e_1, 2^{-1/2}(e_1 + e_2), e_1, 2^{-1/2}(e_1 + e_3), \ldots)$ does not converge weakly, and the distances $(d_C(x_n))_{n \in \NN} = (1, 1, 1, 1, \sqrt{2} - 1, 1, \sqrt{2} - 1, \ldots)$ are not convergent either.
    \end{enumerate}
\end{example}

\begin{proof}
    Although the sequence \((x_n)_{n \in \mathbb{N}}\) presented in \cref{infdim-countex} differs slightly from the one in \cite[Example 3.13]{Opial}, the two are essentially the same. Consequently, the arguments used in \cite{Opial} can be similarly applied to establish \cref{infdim-countex-prop1}, \cref{infdim-countex-prop2}, and \cref{infdim-countex-prop3}.

    Therefore, it remains only to prove \cref{infdim-countex-prop0} to complete the proof. That said, pick any $y \in Y$. Then, as $x_{2k + 1} = e_0 + e_1$ for all $k \in \NN$, we have \begin{equation}\label{infdim-countex-prop-claim}
        y \in \underline{\lim} \, C_\NN \; \Longleftrightarrow \; \|x_n - y\| = \|e_0 + e_1 - y\| \; \text{for all $n$ sufficiently large.}
    \end{equation} But, if $y = (0, y_1, y_2, \ldots)$ and $k \geq 2$, we have \begin{equation*}
    \begin{aligned}
        &\|x_{2k} - y\|^2 = \|e_1 + e_k - y\|^2 = (1 - y_1)^2 + \sum_{\substack{m = 2}}^{k - 1} y_m^2 + (1 - y_k)^2 + \sum_{\substack{m \geq k + 1}} y_m^2, \\
        &\|x_{2k + 1} - y\|^2 = \|e_0 + e_1 - y\|^2 = 1 + (1 - y_1)^2 + \sum_{\substack{m = 2}}^{k - 1} y_m^2 + y_k^2 + \sum_{\substack{m \geq k + 1}} y_m^2.
    \end{aligned}
    \end{equation*} Thus \cref{infdim-countex-prop-claim} holds if and only if $$\|x_{2k} - y\|^2 = \|x_{2k + 1} - y\|^2 \quad \text{for all $k$ sufficiently large,}$$ but the last happens only when $(1 - y_k)^2 = 1 + y_k^2$, i.e., when $y_k = 0$. To summarize, we have shown that $M = \underline{\lim} \, C_\NN \cap Y$, but, in light of \cref{infdim-countex-prop1} and \cref{b-prop}\cref{b-prop2}, it is known that $\underline{\lim} \, C_\NN \subseteq Y$, which is why \cref{infdim-countex-prop0} follows. \end{proof}

\begin{remark}
    To summarize, the introduced \fejer* monotonicity does not, on its own, yield any improvements in \cref{infdim-countex} regarding the convergence of the shadow sequence $(P_C(x_n))_{n \in \NN}$ or the distances $(d_C(x_n))_{n \in \NN}$ and $(d_Y(x_n))_{n \in \NN}$ compared to the inherited Opial properties. However, as will become clear in \cref{section2}, one of the possible issues with \cref{infdim-countex} is that $\ri(M) = \varnothing$ (see \cref{ri-defs}). 
    If this relative interior is nonempty, then many irregularities no longer occur and various standard \fejer\ properties continue to hold in slightly relaxed forms; see, e.g., 
    \cref{raik-*} and \cref{fejer-opial-lemma}.
\end{remark}

\section{Comparison to standard \fejer\ monotonicity}

\label{section2}

In this section, we demonstrate that various positive results are 
possible provided that the underlying 
\fejer* set has \emph{nonempty} relative interior 
(see \cref{ri-defs}).

\begin{theorem}\label{th-regular}
    Let $M$ be a convex subset of $X$ and $(x_n)_{n \in \NN}$ be \fejer* monotone with respect to $M$. Suppose that $\ri(M) \neq \varnothing$ and $K$ is a nonempty compact subset of $\ri(M)$. Then \begin{equation}\label{eventual-fejer-mean}
        (\exists N \in \NN)
        \quad (x_{n + N})_{n \in \NN}\text{ is \fejer\ monotone with respect to } K.
    \end{equation}
\end{theorem}

\begin{proof}
    For brevity, denote $A := \overline{\aff}(M)$. Fix any $y \in \ri(M)$. Then, due to \cref{eq-def} and \cref{incr_inter_l-inf}, we conclude that $M \subseteq \underline{\lim} \, C_\NN$ and \begin{equation*}
        y \in \ri(M) \subseteq \operatorname{int}_{A}(\underline{\lim} \, C_\NN) = \cup_{n \in \NN} \operatorname{int}_{A}(\cap_{n \geq N} C_n).
    \end{equation*} Hence $y \in \operatorname{int}_{A}(\cap_{n \geq N(y)} C_n)$ for some $N(y) \in \NN$, i.e. there exists some open neighborhood $U(y)$ of $y$ in $A$ such that $U(y) \subseteq C_n$ for all $n \geq N(y)$. That said, we cover the compact subset $K$ of $\ri(M)$ with the introduced family of open sets $(U(y))_{y \in \ri (M)}$: \begin{equation*}
        K \subseteq \cup_{y \in K} \ U(y).
    \end{equation*} Choosing a finite subcover implies that, for some $k \in \NN$, \begin{equation*}
        (\exists y_1, \ldots, y_k \in K) \quad K \subseteq \cup_{i = 1}^k \ U(y_i) \subseteq C_n
    \end{equation*} for all $n \geq N:= \max_{1 \leq i \leq k}\{N(y_i)\}$. In particular, $K \subseteq \cap_{n \geq N} C_n$ which literally means eventual \fejer\ monotonicity in the sense of \cref{eventual-fejer-mean}.
\end{proof}

\begin{remark}
    Although the weak topology is metrizable on bounded sets when $X$ is separable, this fact does not assist in extending \cref{th-regular} to the case where $K \subseteq \ri(M)$ is only \emph{weakly} compact, as demonstrated in \cref{int-nonempty-inf-example}.
\end{remark}

\begin{example}\label{int-nonempty-inf-example}
    Let $X = \ell_2(\NN)$. Pick any sequence of positive scalars $(\alpha_n)_{n \in \NN}$ so that $\alpha_n \to 0$ linearly, i.e., $\alpha_{n+1} \leq r\alpha_n$ for some $0 < r< 1$ and all $n \in \NN$, and define $(x_n)_{n \in \NN}$ by \begin{equation*}
        x_n := \alpha_n(e_0 + \gamma e_n),
\end{equation*} where $\gamma > 2$. Then the following hold: \begin{enumerate}
        \item\label{int-nonempty-inf-example-1} $(x_n)_{n \in \NN}$ is \fejer* monotone with respect to $M := \{y \in X \mid \langle y, e_0\rangle < 0\}$.
        \item\label{int-nonempty-inf-example-2} However, $(x_n)_{n \in \NN}$ is \emph{not} eventually \fejer\ monotone with respect to $B[-2e_0, 1] \subseteq M$ in the sense of \cref{eventual-fejer-mean}.
    \end{enumerate}
\end{example}

\begin{proof}
    ``\cref{int-nonempty-inf-example-1}'': Indeed, take any $y \in M$, i.e., $y = (\beta_0, \beta_1, \beta_2, \ldots) \in \ell_2(\NN)$ with $\beta_0 < 0$. Then \begin{equation*}
        \|x_{n+1} - y\|^2 - \|x_n - y\|^2 = (\gamma^2 + 1)(\alpha_{n+1}^2 - \alpha_n^2) + 2((-\beta_0 - \gamma\beta_{n+1})\alpha_{n+1} - (-\beta_0 - \gamma\beta_{n})\alpha_n),
    \end{equation*} where both terms are strictly negative for all $n$ sufficiently large as $\alpha_n \to 0$ geometrically and $-\beta_0 > 0$.

    ``\cref{int-nonempty-inf-example-2}'': Let $y_n := -2e_0 + e_n \in B[-2e_0, 1]$ for $n \in \NN$. Then, since $\alpha_n \downarrow 0$, \begin{equation*}
        \|x_n - y_n\|^2 = (\gamma^2 + 1)\alpha_n^2 + (4 - 2\gamma)\alpha_n + 5
    \end{equation*} strictly increases for all $n$ sufficiently large provided that $\gamma > 2$. However, this contradicts eventual \fejer\ monotonicity of $(x_n)_{n \in \NN}$ with respect to $B[-2e_0, 1]$ as it have to ensure the eventual decrease of this sequence instead.
\end{proof}

Next, we present a slightly more general analogue of the known Raik's proposition \cite[Lemma 1]{Raik} in the context of \fejer* monotone sequences. The proof is mostly analogous to the one given in \cite{Raik}.

\begin{lemma}\label{lemma4raik}
    Let $A$ be a closed nonempty affine subspace of $X$ and $x, x_+ \in X$, $y \in A$. Then TFAE: \begin{enumerate}
        \item\label{lemma4raik-1} $B[y, \rho] \cap A \subseteq C(x, x_+)$ where $C(x, x_+)$ is given as in \cref{hsp-lemma}.
        \item\label{lemma4raik-2} $\|x_+ - y\|^2 \leq \|x - y\|^2 - 2\rho\|P_A(x) - P_A(x_+)\|$.
    \end{enumerate}
\end{lemma}

\begin{proof}
    Indeed, without loss of generality, assume $x \neq x_+$. Denote $L:= A - y$ which is the closed linear subspace of $X$ parallel to $A$. Then the inclusion \cref{lemma4raik-1} can be rewritten as \begin{equation*}
        (B[y, \rho] \cap A) - y = B[0, \rho] \cap L \subseteq C(x - y, x_+ - y) = C(x, x_+) - y.
    \end{equation*} Therefore, item \cref{lemma4raik-1} takes place if and only if, for all $z \in B[0, \rho] \cap L$, we have \begin{align*}
        0 \geq \|x_+ - y - z\|^2 - \|x - y - z\|^2 &= \langle x_+ - x, x_+ + x - 2y - 2 z\rangle \\ &= \langle x_+ - x, x_+ + x - 2y\rangle - 2 \langle x_+ - x, z\rangle \\ &= \|x_+ - y\|^2 - \|x - y\|^2 - 2 \langle x_+ - x, z\rangle.
    \end{align*} Finally, note that the last inequality holds if and only if it holds for $$z_0 := -\rho \frac{P_L(x_+ - x)}{\|P_L(x_+ - x)\|} \in B[0, \rho] \cap L$$ as, due to Cauchy-Schwarz, \begin{equation*}
        -2 \langle x_+ - x, z\rangle = -2 \langle P_L(x_+ - x), z\rangle\leq 2\rho \|P_L(x_+ - x)\|\ = -2 \langle x_+ - x, z_0\rangle.
    \end{equation*} Thus, after rearranging the terms, we derive \begin{align*}
        \|x_+ - y\|^2 \leq \|x - y\|^2 + 2 \langle x_+ - x, z_0\rangle &= \|x - y\|^2 - 2\rho \|P_L(x_+ - x)\| \\ &= \|x - y\|^2 - 2\rho \|P_A(x_+) - P_A(x)\|,
    \end{align*} which is exactly what item \cref{lemma4raik-2} states.
\end{proof}

\begin{proposition}
\label{raiks-ineqs-statem}
    Let $(x_n)_{n \in \NN}$ be \fejer* monotone with respect to a nonempty convex subset $M$ of $X$. Then, for any $y \in \ri(M)$, there exists some $N(y) \in \NN$ and some $\rho > 0$ such that \begin{equation}\label{raiks-ineqs}\begin{aligned}
         \|x_n-y\|^2 - \|x_{n+m + 1}-y\|^2 &\geq  2\rho\sum_{k = 0}^m \|P_{\overline{\aff}(M)}(x_{n+k + 1})-P_{\overline{\aff}(M)}(x_{n+k}) \| \\ &\geq 2\rho\|P_{\overline{\aff}(M)}(x_{n+m+1})-P_{\overline{\aff}(M)}(x_{n}) \|
    \end{aligned}
    \end{equation} for all $n \geq N(y)$, $m \in \NN$.
\end{proposition}

\begin{proof}
    Similarly to the proof of \cref{th-regular}, denote $A := \overline{\aff}(M)$ and fix $y \in \ri(M)$. Due to \cref{incr_inter_l-inf}, it follows that $y \in \operatorname{int}_{A}(\cap_{n \geq N(y)} C_n)$ for some $N(y) \in \NN$, i.e., there exists some $\rho \in \mathbb{R}_{++}$ such that \begin{equation*}\label{raiks-ineqs-1}
        B[y , \rho] \cap A \subseteq \cap_{n \geq N(y)} C_n.
    \end{equation*} Consequently, \cref{lemma4raik}\cref{lemma4raik-2} implies:
    \begin{equation*}
        (\forall n \geq N(y))(\forall k \in \NN) \quad \|x_{n+k+1}-y\|^2 \leq \|x_{n + k}-y\|^2-2 \rho\|P_{A}(x_{n+k+1})-P_{A}(x_{n+k}) \|,
    \end{equation*} where summing over $0 \leq k \leq m$ gives the inequalities in \cref{raiks-ineqs}.
\end{proof}

\begin{remark}\label{raiks-ineqs-statem-remark}
    If the sequence $(x_n)_{n \in \NN}$ is \fejer\ monotone in the classical sense of \cref{def-classic}, then the inequalities \cref{raiks-ineqs} hold with $N(y) \equiv 0$ as $\ri(M) \subseteq \cap_{n \in \NN} C_n$. Furthermore, \cref{lemma4raik} implies that $\rho = \rho(y) > 0$ can be chosen for a standard \fejer\ sequence as 
    $$\rho(y) = d_{(X\setminus M)\cap \overline{\aff}(M)}(y) > 0,$$ 
    i.e., the distance between the given $y \in \ri(M)$ and the boundary of the set $M$ in $\overline{\aff}(M)$.
\end{remark}

The following result appears to be new even in the standard \fejer\ setting:

\begin{corollary}[Raik -- an affine extension]
\label{raik-*}
    Let $(x_n)_{n \in \NN}$ be \fejer* monotone with respect to a nonempty convex subset $M$ of $X$ with $\ri(M) \neq \varnothing$. Then $(P_{\overline{\aff}(M)}(x_n))_{n \in \NN}$ converges strongly to a point in $\overline{\aff}(M)$ and 
    \begin{equation}\label{raik-*-series}
        \sum_{n \in \NN}\|P_{\overline{\aff}(M)}(x_{n+1}) - P_{\overline{\aff}(M)}(x_n)\|<+\infty.
    \end{equation} 
    Hence the shadow sequence $(P_{C}(x_n))_{n \in \NN}$ also converges strongly with a finite-length trajectory for any nonempty closed convex
    subset $C$ of $\overline{\aff}(M)$; in particular, the conclusions hold if let $C = \overline{M}$.
    
\end{corollary}

\begin{proof}
    Indeed, denote $A := \overline{\aff}(M)$ and take any $y \in \ri(M)$. Then the series in \cref{raik-*-series} converges because 
    \cref{raiks-ineqs} yields $$\sum_{n \geq N(y)} \|P_{A}(x_{n+1})-P_{A}(x_{n}) \| \leq \frac{1}{2\rho}\|x_{N(y)} - y\|^2 < +\infty.$$
    Now let $C$ be a nonempty closed convex subset of $A$. 
    Recall that $P_{C}$ is a nonexpansive operator on $X$ and $P_{C} \circ P_{A} = P_{C}$. Then \begin{equation*}\begin{aligned}
        (\forall n \in \NN) \quad \|P_{C}(x_{n+1}) - P_{C}(x_{n})\| &= \|P_{C}(P_{A}(x_{n+1})) - P_{C}(P_{A}(x_{n+1}))\| \\ &\leq \|P_{A}(x_{n+1}) - P_{A}(x_n)\|
    \end{aligned}
    \end{equation*} implying that $(P_{C}(x_n))_{n \in \NN}$ similarly has a finite-length trajectory in the sense of \cref{raik-*-series}. \end{proof}

\begin{remark}\label{*-shadow_convergence-remark}
    As established in \cref{raik-*}, the \fejer* property ensures the strong convergence of the shadow sequence $(P_{\overline{M}}(x_n))_{n \in \NN}$ whenever $\ri(M) \neq \varnothing$. (In the standard \fejer\ setting, the shadow sequence always converges
    strongly.) However, if this condition does not hold, the conclusions of \cref{raik-*} may completely break down as mentioned earlier in \cref{infdim-countex}.
    Furthermore, the same \cref{infdim-countex} shows that the sequence $(P_{\overline{M}}(x_n))_{n \in \NN}$ might even fail to converge weakly if $\ri(M) = \varnothing$, while the Opial property alone still guarantees the weak convergence of $(P_{\overline{\aff}(M)}(x_n))_{n \in \NN}$ accordingly to \cref{OPL}.
\end{remark}

Next, we provide the following technical observation, which on its own represents a good \fejer\ improvement of \cref{opial-lemma-strong}. The standard \fejer\ version of this statement was briefly mentioned in \cite{Raik} and can be traced back at least to \cite[Lemma~6]{GubPolRaik}.

\begin{lemma}\label{fejer-opial-lemma}
    Let $(x_n)_{n \in \NN}$ be a \fejer* monotone sequence with respect to a nonempty convex set $M$ with $\ri(M) \neq \varnothing$. 
    Let $C$ be a nonempty closed convex subset of\, $\overline{\aff}(M)$. Then $(d_C(x_n))_{n \in \NN}$ converges, and \begin{equation*}
        \text{either} \quad \lim_{n \to \infty} d_{C}(x_n) > 0 \quad \text{or} \quad (x_n)_{n\in\NN} \text{ converges to some } z \in C.
    \end{equation*}
\end{lemma}

\begin{proof}
    Due to \cref{b-prop}\cref{b-prop2} and \cref{opial-ext-aff} combined, $(x_n)_{n \in \NN}$ is Opial with respect to $C$. 
    However, \cref{raik-*} yields that the orbit of the corresponding shadow sequence $(P_{C}(x_n))_{n \in \NN}$ is relatively compact; 
    consequently, \cref{conver-dist} implies the convergence of the distances $(d_C(x_n))_{n \in \NN}$.
    
    If $\lim_{n\to\infty} d_C(x_n)>0$, then we are done.
    So assume that $\lim_{n\to\infty} d_C(x_n)=0$, which yields
    $ 0 = \lim_{n \to \infty}d_{C}(x_{n}) \geq  \lim_{n \to \infty}d_{\overline{\aff}(M)}(x_{n}) \geq  0$. Thus,
    \begin{equation}\label{fejer-opial-lemma-1}
      0 = \lim_{n \to \infty}d_{C}(x_{n}) = \lim_{n \to \infty}d_{\overline{\aff}(M)}(x_{n}). 
    \end{equation} Finally, due to \cref{raik-*}, $$P_{\overline{\aff}(M)} (x_n)\, \to\, z \in \overline{\aff}(M)$$ which combined with \cref{fejer-opial-lemma-1} yields that $x_{n} \to z \in C$.
\end{proof}

\begin{remark}
    The assumption $\ri(M) \neq \varnothing$ is essential for the conclusions of \cref{fejer-opial-lemma} to hold; a counterexample is given by \cref{infdim-countex}.
\end{remark}

As another consequence of \cref{incr_inter_l-inf}, we highlight the following observation complementing \cref{example-non-decr} as it provides a sufficient condition for the eventual decrease of distances $(d_{\overline{M}}(x_n))_{n \in \NN}$.

\begin{corollary}
\label{c:251213b}
    Let $(x_n)_{n \in \NN}$ be \fejer* monotone with respect to a nonempty convex subset $M$ of $X$. Suppose that $x_n \to z \in \ri(M)$. Then the distances $(d_{\overline{M}}(x_n))_{n \in \NN}$ are eventually decreasing, i.e., it holds that $d_{\overline{M}}(x_{n + 1}) \leq d_{\overline{M}}(x_{n})$ for some $N \in \NN$ and all $n \geq N$.
\end{corollary}

\begin{proof}
    Indeed, by \cref{incr_inter_l-inf}, pick any neighborhood $U(z) \subseteq \ri(M)$ of the limit point $z \in \ri(M)$ in $\overline{\aff}(M)$ so that $U(z) \subseteq C_n$ for all $n$ sufficiently large, say $n \geq N$. By setting $N \in \NN$ even larger if necessary, assume that $P_{\overline{M}}(x_n) \in U(z) \subseteq C_n$ for all $n \geq N$. 
    Then for some $n \geq N$, we have 
    \begin{equation*}
        d_{\overline{M}}(x_{n + 1}) = \|x_{n+1} - P_{\overline{M}}(x_{n+1})\| \leq \|x_{n+1} - P_{\overline{M}}(x_n)\| \leq \|x_n - P_{\overline{M}}(x_n)\| = d_{\overline{M}}(x_{n}),
    \end{equation*} where the first inequality is due to the definition of the projection $P_{\overline{M}}(x_{n+1})$ and the second 
    inequality holds because $P_{\overline{M}}(x_n) \in C_n$ 
    (recall \cref{hs}).
\end{proof}

To conclude \cref{section2}, we address the "non-cyclicity" property about the closed affine hull of the target set for \fejer* and, consequently, \fejer\ monotone sequences. These statements will play an important role later (see \cref{cycl-prop-inf*2} and \cref{closure-cor-rem}).

\begin{proposition}\label{cycl-prop*}
    Let $(x_n)_{n \in \NN}$ be \fejer* monotone with respect to a nonempty convex subset $M$ of $X$ with $\ri(M) \neq \varnothing$. Then, for some $N \in \NN$, the following hold: \begin{enumerate}[itemsep=0.4em]
        \item\label{cycl-prop*-1} Suppose that $x_n = x_{n + k}$ for some $n \geq N$ and some $k \in \NN$. Then \begin{equation*}
            P_{\overline{\aff}(M)}(x_n) = \cdots = P_{\overline{\aff}(M)}(x_{n+k}) \quad \text{and} \quad d_{\overline{\aff}(M)}(x_n) = \cdots = d_{\overline{\aff}(M)}(x_{n+k}).
        \end{equation*}

        In particular, if $x_n \in \overline{\aff}(M)$, then $x_n = x_{n+1} = \cdots = x_{n+k}$.

        \item\label{cycl-prop*-2} Suppose that $\lim_{n \to \infty} x_n$ exists and equals $z \in X$. If $x_n = z$ for some $n \geq N$, then \begin{equation*}
            (\forall m \in \NN) \quad P_{\overline{\aff}(M)}(x_{n+m}) = P_{\overline{\aff}(M)}(z), \;\; d_{\overline{\aff}(M)}(x_{n+m}) = d_{\overline{\aff}(M)}(z).
        \end{equation*}  Consequently, if $z \in \overline{\aff}(M)$, then $x_{n+m} = z$ for all $m \in \NN$.
    \end{enumerate}
\end{proposition}

\begin{proof}
    Again, we make use of \cref{raiks-ineqs-statem}: take any $y \in \operatorname{ri}(M)$ and fix $N := N(y) \in \NN$ and $\rho > 0$ so that the inequalities in \cref{raiks-ineqs} hold. 
    For brevity, set $A := \overline{\operatorname{aff}}(M)$.
    
    ``\cref{cycl-prop*-2}'': Suppose $x_n = z$ for some $n \geq N$. Then, by taking a limit over $m \to \infty$ in \cref{raiks-ineqs} and recalling that $\lim_{m \to \infty} x_{n + m} = x_n =  z$, we obtain \begin{equation*}
        \sum_{m = 0}^{\infty} \, 2 \rho \|P_{A}(x_{n+m+1})-P_{A}(x_{n+m})\| \leq \|x_n - y\|^2 - \|z - y\|^2 = 0, 
    \end{equation*} which yields $P_{A}(x_{n+m}) = P_{A}(z)$ for all $m \in \NN$. Furthermore, the equality of the affine projections $P_{A}(x_{n+1}) = P_{A}(x_n)$ immediately implies: \begin{equation*}\begin{aligned}
        d^2_A(x_{n+1}) = \|x_{n+1} - &P_A(x_{n+1})\|^2 = \|x_{n+1} - y\|^2 - \|P_A(x_{n+1}) - y\|^2 \\ & \leq \|x_{n} - y\|^2 - \|P_A(x_{n}) - y\|^2 = \|x_{n} - P_A(x_{n})\|^2 = d^2_A(x_{n}).
    \end{aligned}\end{equation*} Finally, repeating this argument iteratively for all $m \in \NN$ gives: \begin{equation*}
        d_A(x_{n}) \geq \cdots \geq d_A(x_{n+m}) \geq \cdots \geq d_A(z).
    \end{equation*} Hence all the conclusions in \cref{cycl-prop*-2} follow as $x_n = z$.

    ``\cref{cycl-prop*-1}'': Now, suppose $x_n = x_{n + k}$ for some $n \geq N$ and some $k \in \NN$. Then, define a new sequence $(\tilde{x}_n)_{n \in \NN}$ as \begin{equation*}
        (\forall m \in \NN) \quad \tilde{x}_m := x_{\min\{m, n+k\}}.
    \end{equation*} Hence $(\tilde{x}_n)_{n \in \NN}$ is \fejer* monotone with respect to the same set $M$, and $\lim_{n \to \infty} \tilde{x}_n = x_{n + k}$. In particular, \cref{cycl-prop*-1} follows directly as a special case of \cref{cycl-prop*-2}, completing the proof.
\end{proof}

\begin{remark}
    \cref{raiks-ineqs-statem-remark} implies that if the sequence $(x_n)_{n \in \NN}$ is \fejer\ monotone in the classical sense of \cref{def-classic}, then he conclusion of \cref{cycl-prop*} holds with $N = 0$.
    \end{remark}

\begin{corollary}[dichotomy] \label{cycl-prop-inf*-2cases}
    Let $(x_n)_{n \in \NN}$ be \fejer* monotone with respect to a nonempty convex subset $M$ of $X$. Suppose that $\ri(M) \neq \varnothing$ and $x_n \to z \in \overline{\aff}(M)$. Then, for some $N \in \NN$, it holds that \begin{equation*}
        \text{either} \quad \{n \geq N \mid x_n = z\} = \varnothing \quad \text{or} \quad \{n \geq N \mid x_n = z\} = \{N, N+1, \ldots\}.
    \end{equation*}
\end{corollary}

\begin{proof}
    \cref{cycl-prop*}\cref{cycl-prop*-2} yields that there exists some $\widetilde{N} \in \NN$ such that if $x_n = z$ for some $n \geq \widetilde{N}$, then $x_{n + m} = z$ for all $m \in \mathbb{N}$. That said, set $N:= \widetilde{N}$ if $\{n \geq \widetilde{N} \mid x_n = z\} = \varnothing$. Otherwise, pick $N:= \min \{n \geq \widetilde{N} \mid x_n = z\}$.
\end{proof}

\section{Directional asymptotics for \fejer* monotone sequences}

\label{section4}

In this section, we derive \fejer* results that mirror the \fejer\ results of \cite[Theorem 3.1]{BKX}.

\begin{proposition}\label{aux-direct-th}
    Let $M$ be a nonempty convex subset of $X$, $z \in X$, and $(x_n)_{n \in \mathbb{N}}$ be a sequence in $X$. Suppose that $(x_n)_{n \in \mathbb{N}}$ is \fejer* monotone with respect to $M$. Then for all $y \in M$ there exists $N(y) \in \NN$ such that the following hold: \begin{equation*}
        \begin{aligned}
            (\forall n \geq N(y))(\forall m \geq n+1) \quad \; \langle x_n-x_m, y-z\rangle &\leq 
            \tfrac{1}{2}\big(\|x_n-z\|^2-\|x_m-z\|^2\big) \\ &\leq \|x_n - x_m\|\|x_m - z\| + \tfrac{1}{2}\|x_n - x_m\|^2.
        \end{aligned}
    \end{equation*}
\end{proposition}

\begin{proof}
    Indeed, assume $y \in M$. Then, due to \cref{def}, we have \begin{equation*}
        (\exists N(y) \in \NN)(\forall n \geq N(y)) \quad \|x_{n+1} - y\| \leq \|x_n - y\|.
    \end{equation*}  Fix any $n \geq N(y)$ and $m \geq n+1$. Then 
    \begin{equation*}
        \begin{aligned}
            0 & \leq \|x_n-y\|^2-\|x_{m}-y\|^2 
            = \|x_n-x_{m}\|^2+2\langle x_n-x_{m}, x_{m}-y \rangle \\
            & = \|x_n-x_{m}\|^2+2\langle x_n-x_{m},(x_{m}-z)-(y-z)\rangle .
        \end{aligned}
    \end{equation*} 
    Hence 
    \begin{equation*}
    \begin{aligned}
        \langle x_n-x_{m}, y-z\rangle & \leq \tfrac{1}{2}\|x_n-x_{m}\|^2+\langle x_n-x_{m}, x_{m}-z\rangle 
        =\tfrac{1}{2}\big(\|x_n-z\|^2-\|x_m-z\|^2\big) \\ 
        & \leq \|x_n - x_m\|\|x_m - z\| + \tfrac{1}{2}\|x_n - x_m\|^2.
    \end{aligned}
\end{equation*} which completes the proof. \end{proof}

\begin{proposition}\label{limsup-sup-prop}
    Let $(v_n)_{n \in \NN}$ be a bounded sequence in $X$ and define $f: X \mapsto \mathbb{R}$ by \begin{equation*}
        f(x) := \limsup_{n \to \infty} \langle x, v_n\rangle.
    \end{equation*} 
    Then $f$ is a real-valued, sublinear, convex, and continuous function.
    For $x\in X$, we have 
    \begin{equation*}
        f(x) = \sigma_\mathcal{D} (x) = \sup_{w \in \mathcal{D}} \langle w, x\rangle, \quad \text{where } \; \mathcal{D} := \mathcal{W}((v_n)_{n \in \mathbb{N}}).
    \end{equation*}
\end{proposition} 

\begin{proof}
    Indeed, $f(x)$ and $\sigma_\mathcal{D}(x)$ are real numbers for all $x \in X$, since $(v_n)_{n \in \NN}$ is bounded.

    Now, let $x \in X$. Get $(v_{n_k})_{k \in \NN}$ such that $\langle x, v_{n_k} \rangle \to f(x)$. Additionally, assume that $v_{n_k} \rightharpoonup v \in \mathcal{D}$. Then \begin{equation*}f(x) \leftarrow \langle x, v_{n_k}\rangle \rightarrow \langle x, v\rangle \leq \sigma_\mathcal{D} (x).
    \end{equation*} So, $f(x) \leq \sigma_\mathcal{D} (x)$. However, choosing $(v_{n_k})_{k \in \NN}$ so that $\langle x, v_{n_k}\rangle \rightarrow \sigma_\mathcal{D} (x)$ ensures: \begin{equation*}
        \sigma_\mathcal{D} (x) = \lim_{k \to \infty} \langle x, v_{n_k}\rangle \leq \limsup_{n \to \infty} \,\langle x, v_{n}\rangle = f(x).
    \end{equation*} In particular, $f(x) \geq \sigma_\mathcal{D} (x)$. Altogether, $f(x) = \sigma_\mathcal{D} (x)$ for all $x \in X$. Finally, note that the support function $\sigma_{\mathcal{D}}(\cdot)$ is sublinear, convex, and continuous.
\end{proof}

\begin{definition}[distinct consecutive terms]
\label{DCT-def}
   From this point forward, we will sometimes refer to $(x_n)_{n \in \NN}$ as  having \emph{distinct consecutive terms} if it satisfies \begin{equation*}
        (\forall n \in \NN) \quad x_{n+1} \neq x_n.
    \end{equation*}
\end{definition}

\begin{remark}\label{dist-remark}
    For any given sequence $(x_n)_{n \in \NN}$, consider a (finite or infinite) subsequence of indices $(n_k)$ defined recurrently as follows: 
    
    Let $n_0 := 0$. For $k \in \NN$, assuming that $(n_\ell)_{\ell = 0}^k$ were previously chosen, set $$I_k := \{j > n_k \mid x_j \neq x_{n_k}\}.$$ Then, if $I_k \neq \varnothing$, let $n_{k+1} := \min I_k$.

    One can notice that, if the new derived subsequence $(n_k)_{k \in \NN}$ is infinite, then $(x_{n_k})_{k \in \NN}$ has distinct consecutive terms, and its maximal \fejer* set coincides with the one given by the initial sequence $(x_n)_{n \in \NN}$ (see \cref{hsp-lemma}). However, if $(n_k)$ is finite, then, for the initial sequence $(x_n)_{n \in \NN}$, it means that \begin{equation*}
        (\exists N \in \NN)(\forall n \geq \NN) \quad x_n = x_N,
    \end{equation*} making the sequence straightforward to analyze. Therefore, in what follows, we can always assume that $(x_n)_{n \in \NN}$ is a sequence with distinct consecutive terms without any loss of generality. Such an assumption will be useful later as it simplifies our notes to some extent (see \cref{cycl-prop-inf*2} below).
\end{remark}

For the remainder of this section, 
if $x_n\to z\in X$, then we introduce the nonempty closed convex cone 
\begin{empheq}[box=\mybluebox]{equation}\label{earlierK}
\mathcal{K} := \overline{\operatorname{cone}}(\overline{M} - z). 
\end{empheq} 

\begin{theorem}\label{direct-th}
    Let $(x_n)_{n \in \mathbb{N}}$ be a sequence of points in $X$ that is \fejer* monotone with respect to a nonempty convex subset $M$ of $X$. Suppose that $x_n \to z \in X$. Then the following hold: \begin{enumerate}
        \item\label{direct-th-1} If $(x_{n})_{n \in \mathbb{N}}$ has distinct consecutive terms, then \begin{equation}\label{direct-formula}
        \mathcal{W}\left(\left(\frac{x_{n}-x_{n+1}}{\left\|x_{n}-x_{n+1}\right\|}\right)_{n \in \mathbb{N}}\right) \subseteq (\overline{M}-z)^{\ominus} = \mathcal{K}^\ominus.
    \end{equation}
        \item\label{direct-th-2} If $x_{n} \neq z$ for all $n \in \NN$, then \begin{equation}\label{direct-formula-1}
        \mathcal{W}\left(\left(\frac{x_{n}-z}{\left\|x_{n}-z\right\|}\right)_{n \in \mathbb{N}}\right) \subseteq (\overline{M}-z)^{\ominus} = \mathcal{K}^\ominus.
    \end{equation}
    \end{enumerate} In particular, if $z \in \overline{M}$, then we may replace $(\overline{M}-z)^{\ominus}$ by $N_{\overline{M}}(z)$ in \cref{direct-formula} and \cref{direct-formula-1}.
\end{theorem}

\begin{proof}
    To simplify our notes, we introduce the following notation: \begin{equation*}
        \mathcal{D}_1 := \mathcal{W}\left(\left(\frac{x_n-x_{n+1}}{\left\|x_n-x_{n+1}\right\|}\right)_{n \in \mathbb{N}}\right), \quad \mathcal{D}_2 := \mathcal{W}\left(\left(\frac{x_n-z}{\left\|x_n-z\right\|}\right)_{n \in \mathbb{N}}\right).
    \end{equation*} According to \cref{aux-direct-th}, for any chosen $z \in M$, there exists $N(z) \in \NN$ such that the following hold: \begin{equation}\label{aux-direct-ineq}
    \begin{aligned}
        (\forall n \geq N(y))(\forall m \geq n+1) \quad \; \langle x_n-x_m, y-z\rangle &\leq \frac{1}{2}(\|x_n-z\|^2-\|x_m-z\|^2) \\ &\leq \|x_{m}-z\| \|x_n-x_{m}\|+\frac{1}{2}\|x_n-x_{m}\|^2.
    \end{aligned}
    \end{equation} ``\cref{direct-th-1}'': For $m = n+1$, dividing by $\|x_n - x_{n+1}\| > 0$ gives: \begin{equation}\label{aux-direct-ineq-3}
        \left\langle \frac{x_n-x_{n+1}}{\|x_n-x_{n+1}\|}, y-z\right\rangle \leq \|x_{n+1}-z\|+\frac{1}{2}\|x_n-x_{n+1}\|. 
    \end{equation} Consequently, by taking $\limsup$ of both sides in \cref{aux-direct-ineq-3}, we derive \begin{equation*}
        (\forall y \in M) \quad \limsup_{n \to \infty} \left\langle \frac{x_n-x_{n+1}}{\|x_n-x_{n+1}\|}, y-z \right\rangle \leq 0,
    \end{equation*} or, equivalently, due to \cref{limsup-sup-prop}, \begin{equation}\label{aux-direct-ineq-5-1}
        (\forall y \in M) \quad \sup_{w \in \mathcal{D}_1} \left\langle w, y-z \right\rangle \leq 0.
    \end{equation} However, by considering the limits over the sequences $z_n \to y \in \overline{M}$, where $z_n \in M$, we get \cref{aux-direct-ineq-5-1} for all $y \in \overline{M}$. 
    
    ``\cref{direct-th-2}'': Similarly, letting $m \to \infty$ in \cref{aux-direct-ineq} and dividing by $\|x_n - z\| > 0$ implies: \begin{equation}\label{aux-direct-ineq-4}
        \left\langle \frac{x_n-z}{\|x_n-z\|}, y-z\right\rangle \leq \frac{1}{2}\|x_n-z\|.
    \end{equation} Therefore, by taking $\limsup$ of both sides in \cref{aux-direct-ineq-4}, for all $y \in M$, we derive \begin{equation}\label{aux-direct-ineq-5-2}
        \limsup_{n \to \infty}  \left\langle \frac{x_n-z}{\|x_n-z\|}, y-z \right\rangle \leq 0 \quad \text{or, equivalently,} \quad \sup_{w \in \mathcal{D}_2} \left\langle w, y-z \right\rangle \leq 0.
    \end{equation} Likewise, by considering the limits over the sequences $z_n \to y \in \overline{M}$, where $z_n \in M$, we get \cref{aux-direct-ineq-5-2} for all $y \in \overline{M}$ completing the proof.
\end{proof}

\begin{remark}
    In view of \cref{sub-def}, we are able to apply \cref{direct-th} to any chosen subsequence $(x_{n_k})_{k \in \mathbb{N}}$ instead. 
    For instance, if a subsequence $(x_{n_k})_{k \in \NN}$ has distinct consecutive terms, then, due to \cref{direct-th}\cref{direct-th-1}, we have \begin{equation*}
        \mathcal{W}\left(\left(\frac{x_{n_k}-x_{n_{k+1}}}{\left\|x_{n_k}-x_{n_{k+1}}\right\|}\right)_{k \in \mathbb{N}}\right) \subseteq (\overline{M}-z)^{\ominus} = \mathcal{K}^\ominus.
    \end{equation*}
\end{remark}

Next, we establish asymptotic inequalities similar to those presented in \cite[Proposition 5.4(iv)]{BC2017}. We begin by examining the finite-dimensional setting and then discuss extensions to the infinite-dimensional case.

\begin{corollary}\label{equivalency-th}
    Assume that $X$ is finite-dimensional and let $(x_n)_{n \in \NN}$ be a \fejer* monotone with respect to a nonempty convex subset $M$ of $X$. Suppose $x_n \to z \in \aff(M)$. Then \begin{equation*}
        (\forall \varepsilon > 0)(\exists N(\varepsilon) \in \NN)(\forall n \geq N(\varepsilon)) \quad \|x_n - z\| \leq (1 + \varepsilon)d_{\mathcal{K}}(x_n - z),
    \end{equation*} where $\mathcal{K} = \overline{\operatorname{cone}}(\overline{M} - z)$. 
    Moreover, if $(x_n)_{n \in \NN}$ has distinct consecutive terms, then $x_n \neq z$ eventually, and \begin{equation*}
        \lim_{n \to \infty} \frac{d_{\mathcal{K}}(x_n - z)}{\|x_n - z\|} = 1.
    \end{equation*}
\end{corollary}

\begin{proof}
    For notational convenience, we assume that $z = 0$. 
    In view of \cref{dist-remark}, we also assume that $(x_n)_{n \in \mathbb{N}}$ has distinct consecutive terms (otherwise, $x_n=z$ eventually and the first inequality 
    turns into $0=0$). Combined with \cref{cycl-prop-inf*-2cases}, it follows that, for some $N \in \NN$, \begin{equation*}
          (\forall n \geq N) \quad x_n \neq 0. 
    \end{equation*} Suppose that, after possibly passing to a subsequence, $$\frac{x_{n_k}}{\|x_{n_k}\|} \to w \in \mathcal{K}^{\ominus} \quad \text{and} \quad \frac{P_{\mathcal{K}}(x_{n_k})}{\|x_{n_k}\|} \to v \in \mathcal{K},$$ where the first inclusion is due to \cref{direct-th}\cref{direct-th-2}. 
    That said, \cref{moreau_decomp} yields that 
    \begin{equation*}
        \|v\|^2 = \lim_{k \to \infty}\left\langle \frac{P_{\mathcal{K}}(x_{n_k})}{\|x_{n_k}\|}, \frac{P_{\mathcal{K}}(x_{n_k})}{\|x_{n_k}\|} \right\rangle = \lim_{k \to \infty}\left\langle \frac{x_{n_k}}{\|x_{n_k}\|}, \frac{P_{\mathcal{K}}(x_{n_k})}{\|x_{n_k}\|} \right\rangle = \langle w, v\rangle \leq 0.
    \end{equation*} In particular, $v = 0$ and $P_{\mathcal{K}}(x_n)/\|x_n\| \to 0$. Finally, \begin{equation*}
        1 \geq \frac{d_{\mathcal{K}}(x_n)}{\|x_n\|} 
        = \frac{\|x_n-P_{\mathcal{K}}(x_n)\|}{\|x_n\|}
        \geq  
        1 - \frac{\|P_{\mathcal{K}}(x_n)\|}{\|x_n\|}
        \to 1
    \end{equation*} 
    as $n\to\infty$. 
\end{proof}

\begin{remark}
    In the setting of \cref{equivalency-th}, 
    as $\overline{M} \subseteq z + \overline{\operatorname{cone}}(\overline{M} - z) = z + \mathcal{K}$, we have 
    \begin{equation}
    \label{e:251210a}
        d_{\mathcal{K}}(x_n - z) = d_{z + \mathcal{K}}(x_n) \leq d_{\overline{M}}(x_n).
    \end{equation} Now assume that $z \in \overline{M}$. Then $\mathcal{K} = T_{\overline{M}}(z)$ and \begin{equation*}
        \|x_n - z\| \leq (1 + \varepsilon) d_{T_{\overline{M}}(z)}(x_n - z) \leq (1 + \varepsilon) d_{\overline{M}}(x_n)
    \end{equation*} for any $\varepsilon > 0$ fixed and all $n$ sufficiently large. Moreover, because $z \in \overline{M} \subseteq z + \mathcal{K}$, 
    we estimate 
     \begin{align*}
        \|P_{\overline{M}}(x_n) - z\|^2 &\leq \|x_n - z\|^2 - \|x_n - P_{\overline{M}}(x_n)\|^2 \leq \|x_n - z\|^2 - \|x_n - P_{z + \mathcal{K}}(x_n)\|^2\\ 
        &= \|P_{z + \mathcal{K}}(x_n) - z\|^2,
    \end{align*} 
    where the first inequality is due to the firm nonexpansiveness 
    of $P_{\overline{M}}$, the second follows from \cref{e:251210a}, and the last equality follows 
    from the Moreau decomposition (\cref{moreau_decomp}). Finally, 
    \begin{subequations}
    \label{add-equiv-1}
    \begin{align}
            \|P_{\overline{M}}(x_n) - z\| 
            &\leq \|P_{z + \mathcal{K}}(x_n) - z\| = \|P_{\mathcal{K}}(x_n - z)\| = \left(\|x_n - z\|^2 -d^2_{\mathcal{K}}(x_n - z)\right)^{\frac{1}{2}} \\ 
            &= o(\|x_n - z\|) = o(d_{\mathcal{K}}(x_n - z)) = o(d_{\overline{M}}(x_n)),
    \end{align} 
    \end{subequations}
    which explicitly demonstrates how the shadow sequence $(P_{\overline{M}}(x_n))_{n \in \NN}$ outperforms the initial one if the latter has a limit $x_n \to z \in \overline{M}$.
\end{remark}

\begin{remark}
    Since $\langle \cdot, \cdot \rangle$ is not weakly–weakly continuous, \cref{equivalency-th} cannot be directly extended to infinite-dimensional spaces unless additional restrictive assumptions are imposed on the directional sequences — for instance, one might require that \begin{equation*}
        \left\{\frac{x_n - z}{\|x_n - z\|}\right\}_{n \in \NN} \quad \text{or} \quad \left\{\frac{P_{\mathcal{K}}(x_n - z)}{\|x_n - z\|}\right\}_{n \in \NN}
    \end{equation*} be relatively compact. As a matter of fact, the conclusions of \cref{equivalency-th} generally do not hold in the infinite-dimensional case. To see this, consider the \fejer* monotone sequence $(x_n)_{n \in \NN}$ in $\ell_2(\NN)$ given by \cref{int-nonempty-inf-example} and take the same target set $$M = \{y \in X \mid \langle y, e_0 \rangle < 0\} \subseteq \underline{\lim} \, C_\NN.$$ Then $x_n \to 0$, $\overline{M} = T_{\overline{M}}(0)$, and $P_{\overline{M}}(x_n) = \alpha_n \gamma \,e_n$ for $n\geq 1$. In particular, \begin{equation}\label{int-nonempty-inf-example-dist}
        (\forall n \geq 1) \quad \|x_n\| = \alpha_n \sqrt{1 + \gamma^2} = \sqrt{1 + \gamma^2} \,d_{\overline{M}}(x_n).
    \end{equation} 
    Therefore, the constant $\Gamma := \sqrt{1 + \gamma^2}$ in \cref{int-nonempty-inf-example-dist} depends explicitly on $\gamma$ and can be arbitrarily large. Additionally, note that changing $\gamma$ alters the sequence $(x_n)_{n \in \NN}$ but no the target set $M$; in particular, if there exists a constant $\Gamma \geq 1$ providing the inequality in \cref{int-nonempty-inf-example-dist}, 
    then $\Gamma$ must depend on the target set $M$ and on the sequence $(x_n)_{n \in \NN}$ as well. As it turns out, such a constant 
    $\Gamma = \Gamma(x_\NN, M)$ always exists provided that $\ri(M) \neq \varnothing$ as we show next.
\end{remark}

\begin{theorem}\label{equiv-th-inf}
    Let $(x_n)_{n \in \NN}$ be \fejer* monotone with respect to a nonempty convex subset $M$ of $X$ and $x_n \to z \in \overline{\aff}(M)$. Assume $\ri(M) \neq \varnothing$. 
    Then, for some constant $\Gamma = \Gamma(x_\NN, M) \geq1$, \begin{equation}\label{equiv-th-inf-not-a-goal}
        (\exists N \in \NN)(\forall n \geq N) \quad \|x_n - z\| \leq \Gamma d_{\mathcal{K}}(x_n - z),
    \end{equation} where $\mathcal{K}=\overline{\operatorname{cone}}(\overline{M} - z)$. Similarly, if $(x_n)_{n \in \NN}$ has distinct consecutive terms, then $x_n \neq z$ eventually, and \begin{equation}\label{equiv-th-inf-goal}
        \liminf_{n \to \infty} \frac{d_{\mathcal{K}}(x_n - z)}{\|x_n - z\|} > 0.
    \end{equation}
\end{theorem}

\begin{proof}
    Once again, following \cref{dist-remark}, we assume that $(x_n)_{n \in \NN}$ has distinct consecutive terms (otherwise, $x_n=z$ eventually and the first inequality turns into $0=0$). Then, \cref{cycl-prop-inf*-2cases} implies the existence of some $N \in \NN$ such that $x_n \neq z$ for all $n \geq N$. 
    For simplicity’s sake, we assume that $z = 0$ and that $N = 0$.
    It suffices to verify that \cref{equiv-th-inf-goal} holds with $z = 0$. To this end, assume to the contrary, and after passing to a subsequence if needed, that 
    \begin{equation}\label{equiv-th-inf-1}
        \lim_{n \to \infty} \frac{d_{\mathcal{K} }(x_n)}{\|x_n\|} = 0.
    \end{equation} Then, \cref{moreau_decomp} yields that $d_{\mathcal{K}}(x_n) = \|x_n - P_{\mathcal{K}}(x_n)\| = \|P_{\mathcal{K}^{\ominus}}(x_n)\|$ and \begin{equation}\label{equiv-th-inf-2}
        \frac{\|P_{\mathcal{K} }(x_n)\|^2}{\|x_n\|^2} = 1 - \frac{d^2_{\mathcal{K} }(x_n)}{\|x_n\|^2} \to 1 
        \quad \text{as } n \to \infty.
    \end{equation} Now, assume that, along some subsequence $(n_k)_{k \in \NN}$, $P_{\mathcal{K} }(x_{n_k})/\|x_{n_k}\| \rightharpoonup w$. Then $w \in \mathcal{K}$ and \begin{equation*}
    \left\langle \frac{P_{\mathcal{K} }(x_{n_k})}{\|x_{n_k}\|},  w\right\rangle = \left\langle \frac{P_{\mathcal{K} }(x_{n_k}) - x_{n_k}}{\|x_{n_k}\|},  w\right\rangle + \left\langle \frac{x_{n_k}}{\|x_{n_k}\|},  w\right\rangle, 
\end{equation*} where, due to \cref{equiv-th-inf-1}, the first term in the right hand side tends to $0$, and, for the second one, \cref{direct-th}\cref{direct-th-2} implies: \begin{equation*}
    \limsup_{k \to \infty}\left\langle\frac{x_{n_k}}{\|x_{n_k}\|},  w\right\rangle \leq 0.
\end{equation*} To summarize, \begin{equation*}
    \|w\|^2 = \lim_{k \to \infty}\left\langle \frac{P_{\mathcal{K} }(x_{n_k})}{\|x_{n_k}\|},  w\right\rangle = \limsup_{k \to \infty}\left\langle\frac{x_{n_k}}{\|x_{n_k}\|},  w\right\rangle \leq 0.
\end{equation*} In particular, $w = 0$ and $P_{\mathcal{K} }(x_{n})/\|x_{n}\| \rightharpoonup 0$ as $n \to \infty$. Due to \cref{equiv-th-inf-1}, we conclude that \begin{equation}\label{equiv-th-inf-3}
    \frac{x_n}{\|x_n\|} = \frac{P_{\mathcal{K}}(x_n)}{\|x_n\|} + \frac{P_{\mathcal{K}^{\ominus}}(x_n)}{\|x_n\|}\rightharpoonup 0.
\end{equation}

To conclude, pick any $y \in \ri(M)$ and abbreviate $A := \overline{\aff}(M)$. Then, since $x_n \to 0$, taking the limit over $m \to \infty$ in \cref{raiks-ineqs} and dividing by $\|x_n\| > 0$ yields for some $\rho > 0$ and all $n$ sufficiently large: \begin{equation}\label{equiv-th-inf-4}
    \left\langle \frac{x_n}{\|x_n\|}, x_n - 2y\right\rangle = \frac{1}{\|x_n\|}(\|x_n - y\|^2 - \|y\|^2) \geq 2\rho \frac{\|P_{A}(x_n)\|}{\|x_n\|}.
\end{equation} 
However, due to \cref{equiv-th-inf-3} and the fact that $x_n \to 0$, the left hand side of \cref{equiv-th-inf-4} converges to $0$; thus 
$P_A(x_n)/\|x_n\|\to 0$ as $n\to\infty$. Because $0\in \mathcal{K} \subseteq A$ and hence 
$P_{\mathcal{K}} = P_{\mathcal{K}}\circ P_A$, 
we have 
$\|P_{\mathcal{K}}(x_n)\| 
= \|P_{\mathcal{K}}(P_Ax_n)\|
\leq \|P_A(x_n)\|$.
Thus 
\begin{equation*}
0 \leq \frac{\|P_{\mathcal{K}}(x_n)\|}{\|x_n\|}
\leq \frac{\|P_{A}(x_n)\|}{\|x_n\|} \to 0
\end{equation*}
as $n\to \infty$. 
Hence $P_{\mathcal{K}}(x_n)/\|x_n\|\to 0$
as $n\to\infty$; however, this 
contradicts \cref{equiv-th-inf-2}.
\end{proof}

\begin{remark}
    If the assumption $\ri(M) \neq \varnothing$ is not satisfied, the conclusions of \cref{equiv-th-inf} may fail to hold in general as the following \cref{example-no_ineq} demonstrates.
\end{remark}

\begin{example}\label{example-no_ineq}
    Set $X = \ell_2(\NN)$. Pick any sequence of positive scalars $(\alpha_n)_{n \in \NN}$ so that $\alpha_n \to 0$ decreases and define $(x_n)_{n \in \NN}$ by $x_n := \alpha_n e_n$. Then the following hold: \begin{enumerate}
        \item\label{example-no_ineq-1} $(x_n)_{n \in \NN}$ is \fejer* monotone with respect to 
        $$M := \operatorname{span}\{e_n\}_{n\in \NN} = \{\lambda_0 e_0 + \ldots + \lambda_n e_n \mid \lambda_i \in \mathbb{R}, n \in \NN\}.$$
        \item\label{example-no_ineq-2} However, $(x_n)_{n \in \NN}$ has distinct consecutive terms, $x_n \to 0$, and $x_n \in M$ for all $n \in \NN$.
    \end{enumerate}
\end{example}

\begin{proof}
    Indeed, for any chosen $y \in M$, $x_n - x_{n+1} \perp y$ for all $n$ sufficiently large; hence, \begin{equation*}\begin{aligned}
        \|x_n - y\|^2 - \|x_{n+1} - y&\|^2 = \langle x_n - x_{n+1}, x_n + x_{n+1} - 2y\rangle \\ &= \langle x_n - x_{n+1}, x_n + x_{n+1} \rangle = \|x_n\|^2 - \|x_{n+1}\|^2 = \alpha^2_n - \alpha^2_{n+1} \geq 0,
    \end{aligned}
    \end{equation*} where the last inequality holds since $\alpha_n \to 0$ decreases.
\end{proof}

Next, we point out some ``exclusion'' properties of \fejer* monotone sequences with distinct consecutive terms, complementing previously given \cref{cycl-prop-inf*-2cases}.

\begin{corollary}[exclusion properties]
\label{cycl-prop-inf*2}
    Let $(x_n)_{n \in \NN}$ be a sequence with distinct consecutive terms that is \fejer* monotone with respect to a nonempty convex set $M$ with $\ri(M) \neq \varnothing$. Then the following hold: \begin{enumerate}
        \item\label{cycl-prop-inf*2-1} If $x_n \rightarrow z \in \overline{\aff}(M)$, then $x_n \not \in z + \overline{\operatorname{cone}}(\overline{M} - z)\, $ for all $n$ sufficiently large.
        \item\label{cycl-prop-inf*2-2} $x_n \not \in \overline{M}\,$ for all $n$ large enough.
    \end{enumerate}
\end{corollary}

\begin{proof}
    ``\cref{cycl-prop-inf*2-1}'': \cref{cycl-prop-inf*-2cases} yields that $x_n \neq z$ for all $n$ sufficiently large. Hence \cref{equiv-th-inf-not-a-goal} implies \cref{cycl-prop-inf*2-1}.
    
    ``\cref{cycl-prop-inf*2-2}'': Assume the opposite holds; that is, there exists a subsequence $(n_k)_{k \in \mathbb{N}}$ such that $x_{n_k} \in \overline{M}$. Then, \cref{fejer-opial-lemma} 
    (with $C=\overline{M}$) implies that $x_n \to z$ for some $z \in \overline{M}$. But this leads to a contradiction when considered in light of \cref{cycl-prop-inf*2}\cref{cycl-prop-inf*2-1} as $\overline{M} \subseteq z + \overline{\operatorname{cone}}(\overline{M} - z)$.
\end{proof}

\begin{remark}
    If $\ri(M) = \varnothing$, then \cref{cycl-prop-inf*2}\cref{cycl-prop-inf*2-2} may fail to hold as demonstrated by \cref{example-no_ineq}.
\end{remark}

\begin{remark}\label{closure-cor-rem}
    The fact that $x_n \not \in \overline{M}$ for all $n \in \NN$ is trivial for standard \fejer\ sequences $(x_n)_{n \in \NN}$ with distinct consecutive terms as it follows directly from \cref{def-classic}: indeed, suppose to the contrary that $x_n = y \in \overline{M}$ for some $n \in \NN$. Then, in particular, since $\overline{M} \subseteq \cap_{n \in \NN} C_n$, we have \begin{equation*}
        (\forall k \in \NN) \quad 0 = \|x_n - y\| \geq \|x_{n+k} - y\|,
    \end{equation*} which contradicts distinct consecutiveness. However, in the \fejer* case, we \emph{cannot} guarantee 
    that 
    \begin{equation*}
        P_{\overline{M}}(x_n) \in \bigcap_{k \geq n} C_k \quad \text{for all $n$ sufficiently large.}
    \end{equation*} (Indeed, one can show that \cref{example-non-decr} is 
    a counterexample.) Therefore, a more subtle approach is needed.
\end{remark}

Next, we prove the following inequality \cref{add-equiv-2}, 
which will be further discussed in \cref{section6}.

\begin{corollary}\label{add-equiv}
    In the setting of \cref{equiv-th-inf}, we additionally have: \begin{equation}\label{add-equiv-2}
        \|z - y\| \leq (\Gamma+2)\|x_n - y\|
    \end{equation} for all $n \geq N$ and all $y \in \overline{M}$.
\end{corollary}

\begin{proof} 
    Following \cref{dist-remark} once again, we assume that $(x_n)_{n \in \NN}$ has distinct consecutive terms; otherwise, $x_n=z$ eventually and the inequality turns into 
    $$\|z - y\| \leq (\Gamma+2) \|z - y\|.$$ 
    \cref{cycl-prop-inf*2}\cref{cycl-prop-inf*2-2} implies that $x_n \not \in \overline{M}$ for all $n$ sufficiently large. That said, for all $y \in \overline{M}$, we have \begin{equation*}
        \frac{\|z - y\|}{\|x_n - y\|} \leq \frac{\|P_{\overline{M}}(x_n) - y\|}{\|x_n - y\|} + \frac{\|z - P_{\overline{M}}(x_n)\|}{\|x_n - y\|} \leq \frac{\|P_{\overline{M}}(x_n) - y\|}{\|x_n - y\|} + \frac{\|z - P_{\overline{M}}(x_n)\|}{d_{\overline{M}}(x_n)}.
    \end{equation*} However, since $P_{\overline{M}}$ is a nonexpansive operator and $y = P_{\overline{M}}(y)$, we have \begin{equation}\label{add-equiv-2-fn}
        \|P_{\overline{M}}(x_n) - y\| = \|P_{\overline{M}}(x_n) - P_{\overline{M}}(y)\| \leq \|x_n - y\|.
    \end{equation} 
    Moreover, due to \cref{equiv-th-inf} and \cref{e:251210a},
     we have \begin{equation}\label{add-equiv-2-fn-1}
        \frac{\|z - P_{\overline{M}}(x_n)\|}{d_{\overline{M}}(x_n)} \leq \frac{\|x_n - z\|}{d_{\overline{M}}(x_n)} + \frac{\|x_n - P_{\overline{M}}(x_n)\|}{d_{\overline{M}}(x_n)} 
        \leq \Gamma + 1.
    \end{equation} Altogether, \cref{add-equiv-2-fn} and \cref{add-equiv-2-fn-1} combined imply \cref{add-equiv-2}.
\end{proof}

To conclude this section, we provide the \fejer* counterpart of \cite[Theorem 5.12]{BC2017}.

\begin{corollary}\label{linear-conv*}
Let $(x_n)_{n \in \NN}$ be \fejer* monotone with respect to a nonempty convex subset $M$ of $X$ such that $\ri(M) \neq \varnothing$. Suppose that, for some $r \in [0, 1)$, \begin{equation}\label{kappa-ineq-dist1}
(\forall n \in \NN) \quad d_{\overline{M}} (x_{n+1}) \leq r d_{\overline{M}} (x_{n}).
\end{equation} 
Then $(x_n)_{n \in \NN}$ converges linearly to a point in $\overline{M}$ in the sense that, 
for some $\Gamma = \Gamma(x_{\NN}, M) \geq 1$ and some $z \in \overline{M}$, \begin{equation}\label{kappa-ineq-dist2}
(\exists N \in \NN)(\forall n \geq N) \quad \|x_n - z\| 
\leq \Gamma r^n d_{\overline{M}} (x_{0}).
\end{equation}
\end{corollary}

\begin{proof}
Indeed, due to \cref{kappa-ineq-dist1}, we have \begin{equation*}
(\forall n \in \NN) \quad d_{\overline{M}} (x_{n}) \leq r^n d_{\overline{M}} (x_{0})
\end{equation*} and $\lim_{n \to \infty} d_{\overline{M}} (x_n) = 0$. In particular, \cref{fejer-opial-lemma} implies that $x_n \rightarrow z \in \overline{M}$. 
Therefore, due to \cref{equiv-th-inf} and \cref{e:251210a}, we have \begin{equation*}
(\exists N \in \NN)(\forall n \geq N) \quad \|x_n - z\| 
\leq \Gamma d_{\mathcal{K}}(x_n-z) \leq \Gamma d_{\overline{M}}(x_n) 
\leq \Gamma r^n d_{\overline{M}}(x_0),
\end{equation*} which yields \cref{kappa-ineq-dist2}. \end{proof}

\begin{remark}\label{QR-lin-remark}
    As one can see, the inequalities outlined in \cref{equivalency-th} or \cref{equiv-th-inf} allow us to obtain any results on linear convergence similar to \cref{linear-conv*}. Following Behling, Bello-Cruz, Iusem, Liu, and Santos in \cite[Theorem 2.6(ii)]{BBCILS}, we can assert that under the assumptions of \cref{equivalency-th}, as soon as the scalar sequence of distances $(d_{\overline{M}} (x_{n}))_{n \in \NN}$ converges to zero $Q-$ or $R-$linearly with asymptotic constant $\eta$, the sequence of points $(x_{n})_{n \in \NN}$ converges to $z \in \overline{M}$ $Q-$ or $R-$linearly, respectively, with the same asymptotic constant (for the definitions of these types of linear convergence, see \cite[Appendix A.2, Rates of Convergence]{NoWr}.)
\end{remark}

\begin{remark}[patching a gap in the literature]
\label{r:patch}
    In our view, the proof of \cite[Theorem 2.6(ii)]{BBCILS} contains a significant flaw. To avoid any confusion, we recall the notation and assumptions made by the authors. Specifically, their sequence $(y^k)_{k \in \NN}$ is assumed to be \fejer* monotone with respect to $\operatorname{int}(C)$, where $C \subseteq \mathbb{R}^n$ is nonempty, closed, and convex, and having a limit $y^* \in C$. Then, the step where the existence of some $\hat{y} \in \operatorname{int}(C)$ satisfying \begin{equation}\label{authors-SIAM-problem}
        \|\hat{y} - z^k\| \leq \|y^k - z^k\|,
    \end{equation} with $z^k := P_C(y^k)$ is asserted seems ambiguous to us.
    
    Indeed, it is not clear whether such choice of $\hat{y} \in \operatorname{int}(C)$ has to depend on $k$. 
    
    \textit{Case 1}: If $\hat{y} \in \operatorname{int}(C)$ is allowed to vary with $k$, then the choice of $\hat{y} = \hat{y}(k)$ is indeed possible and not unique. However, the following arguments break down: the key inequalities in the proof (see \cite[(2.4)–(2.5)]{BBCILS}) are stated for $k \geq N(\hat{y})$, but it is not clear whether $k \geq N(\hat{y}(k))$ holds for all sufficiently large $k$.

    \textit{Case 2}: Conversely, suppose $\hat{y}$ is chosen as in \cref{authors-SIAM-problem} for all sufficiently large $k \in \NN$, i.e., in a way that does not depend on $k$. Then passing to the limit in \cref{authors-SIAM-problem} would give \begin{equation*}
        \|\hat{y} - P_C(y^*)\| = \|\hat{y} - y^*\| \leq \lim_{k \to \infty} \|y^k - z^k\| = 0,
    \end{equation*} implying that $\hat{y} = y^*$ making this choice of $\hat{y}$ predetermined. Moreover, if the sequence is assumed to be converging to $y^* \in C$ from outside of the set $C$, i.e. $y^k \not \in C$, then $y^* \in \operatorname{bdry}(C)$ making the choice of $\hat{y} \in \operatorname{int}(C)$ satisfying \cref{authors-SIAM-problem} impossible at all. 
    
    While the inequality \cref{authors-SIAM-problem} with $\hat{y} = y^*$ and for all sufficiently large $k$ is consistent with \cref{add-equiv-1}, its use in their proof still requires justification. Once established, however, this inequality streamlines their proof and makes the outlined conclusions solid and fully justified.
\end{remark}

\section{On the structure of maximal \fejer* sets  with nonempty interior}\label{section3}

In this section, 
assuming that the \fejer* set $M$ has a nonempty interior, we are able to explicitly characterize the entire maximal \fejer* set $\underline{\lim} \, C_\NN$ up to its boundary. 
We also present an analogous result for the maximal Opial set. 

\begin{definition}\label{def-non-repeat}
    For a given sequence $(x_n)_{n \in \NN}$ of points in $X$, we introduce the following notation: 
    \begin{empheq}[box=\mybluebox]{equation}\label{cone-notation}
        (\forall n\in \NN) \quad v_n := \begin{cases} 
        \displaystyle \frac{x_n - x_{n+1}}{\|x_n - x_{n+1}\|}, & \text { if } x_{n+1} \neq x_n; \\[+5mm] 0, & \text { otherwise, }\end{cases}
    \end{empheq} 
and 
    \begin{empheq}[box=\mybluebox]{equation}
    \label{laterK}
    \mathcal{D} := \mathcal{W}((v_n)_{n \in \mathbb{N}}), 
    \;\;\text{ and }\;\;
    \mathcal{K} := (\overline{\operatorname{cone}}(\mathcal{D}))^{\ominus}. 
    \end{empheq} 
\end{definition}
Note that \cref{laterK} differs from the definition 
\cref{earlierK} in the previous section. 

Prior to presenting \cref{lemma-cone-inf}, we highlight that \cref{direct-th} already provides a localization of the maximal \fejer* set $\underline{\lim} \, C_\NN$ in terms of the directional asymptotics $\mathcal{D}$: 

\begin{corollary}\label{cone-th}
    Let $(x_n)_{n \in \NN}$ be a sequence in $X$ such that $x_n \to z \in X$. Then \begin{equation}\label{cone-inclusion}
            \underline{\lim} \, C_\NN \subseteq (\overline{\operatorname{cone}}(\mathcal{D}))^{\ominus} + z 
            =\mathcal{K}+z.
        \end{equation}
\end{corollary}

\begin{proof}
    If there exists $N \in \NN$ such that $x_n = x_N$ for all $n \geq N$, then \cref{cone-inclusion} holds indeed since $$\mathcal{D} = \{0\} \, \text{ and } \, \left(\overline{\operatorname{cone}}(\mathcal{D})\right)^{\ominus} = X.$$ Otherwise, following \cref{dist-remark}, assume $(x_n)_{n \in \NN}$ has distinct consecutive terms. Since the inclusion \cref{cone-inclusion} clearly holds for the case when $\underline{\lim} \, C_\NN = \varnothing$, we may assume the opposite, i.e., $M:= \underline{\lim} \, C_\NN \neq \varnothing$. By \cref{eq-def}, $(x_n)_{n \in \NN}$ is \fejer* monotone with respect to $M$. By \cref{direct-th}\cref{direct-th-1}, \begin{equation*}
        \mathcal{D} \subseteq (\overline{M} - z)^{\ominus}.
    \end{equation*} Taking polar cones implies that \begin{equation*}
        M - z \subseteq \overline{\operatorname{cone}}(\overline{M} - z) \subseteq (\overline{M} - z)^{\ominus\ominus} \subseteq \mathcal{D}^{\ominus} = (\overline{\operatorname{cone}}(\mathcal{D}))^{\ominus}.
    \end{equation*} Thus, $M \subseteq (\overline{\operatorname{cone}}(\mathcal{D}))^{\ominus} + z$ and \cref{cone-inclusion} holds. \end{proof}

\begin{theorem}\label{lemma-cone-inf}
    Let $(x_n)_{n \in \NN}$ be a sequence of points in $X$ with distinct consecutive terms, and recall \cref{laterK}.
    Then TFAE:
    
    \begin{enumerate}
        \item\label{lemma-cone-inf1} $(x_n)_{n \in \NN}$ is \fejer* monotone with respect to a convex set $M$ with $\operatorname{int}(M) \neq \varnothing$.
        \item\label{lemma-cone-inf2} $(x_n)_{n \in \NN}$ is bounded, $\mathcal{K}$ is solid, and $0 \not \in \mathcal{D}$.
    \end{enumerate} Moreover, if 
    \cref{lemma-cone-inf1} or \cref{lemma-cone-inf2} holds, 
    then $(x_n)_{n \in \NN}$ has a limit $z \in X$ and \begin{equation}\label{cone-int-equal-inf}
        \operatorname{int}(\underline{\lim} \, C_\NN) = \operatorname{int}(\mathcal{K}) + z,\quad \overline{\underline{\lim} \, C_\NN} = \mathcal{K} + z.
    \end{equation}
\end{theorem}

\begin{proof} ``\cref{lemma-cone-inf1}$\Rightarrow$\cref{lemma-cone-inf2}'':
\cref{raik-*} implies $x_n \to z \in X$. 
Let $y \in \operatorname{int}(M)$. 
Dividing \cref{raiks-ineqs} 
by $\|x_{n+1} - x_n\| > 0$ gives
\begin{equation}\label{lemma-cone-inf-1}
    \left\langle \frac{x_{n} - x_{n+1}}{\|x_{n} - x_{n+1}\|}, x_n + x_{n+1} - 2y\right\rangle = \frac{1}{\|x_{n} - x_{n+1}\|} (\|x_n - y\|^2 - \|x_{n+1} - y\|^2) \geq 2\rho > 0.
\end{equation} Now, by taking the limits over all possible subsequences $n_k \to \infty$ in the left hand side of \cref{lemma-cone-inf-1}, we conclude that $\langle w, z - y\rangle \geq \rho$ for all $w \in \mathcal{D}$. Therefore, $0 \not \in \mathcal{D}$, $y-z\in\inte(\mathcal{K})$ by 
\cref{int-dual-cone-prop}, and $\mathcal{K}$ is solid.

``\cref{lemma-cone-inf2}$\Rightarrow$\cref{lemma-cone-inf1}'': 
Let $u \in \operatorname{int}(\mathcal{K})$, i.e., $\langle k, u\rangle \geq \delta \|k\|$ for all $k \in \mathcal{K}^{\oplus} = -\mathcal{K}^{\ominus}$ and some $\delta \in \mathbb{R}_{++}$ according to \cref{int-dual-cone-prop}. Because $0 \not \in \mathcal{D}$ and $\mathcal{K} = (\overline{\operatorname{cone}}(\mathcal{D}))^{\ominus}$, we conclude \begin{equation*}
    \left\langle x_{n + 1} - x_n, u\right\rangle \geq \frac{\delta}{2} \|x_{n + 1} - x_{n }\|
\end{equation*} for all $n$ sufficiently large, say $n \geq N$. 
Consequently, 
\begin{equation*}
        (\forall n \geq N) \quad \langle x_{n+1} - x_{N}, u \rangle = \sum_{\ell = N}^{n}\langle x_{\ell+1} - x_{\ell}, u\rangle \geq \frac{\delta}{2} \sum_{\ell = N}^{n} \|x_{\ell+1} - x_{\ell}\|.
    \end{equation*} 
    Next, as $(x_n)_{n \in \NN}$ is bounded, pick any weakly convergent subsequence $x_{n_k} \rightharpoonup z$. Then \begin{equation*}
        \frac{\delta}{2} \sum_{\ell = N}^{\infty} \|x_{\ell+1} - x_{\ell}\| \leq \lim_{k \to \infty}\sum_{\ell = N}^{n_k - 1}\langle u, x_{\ell+1} - x_{\ell}\rangle = \lim_{k \to \infty} \langle u, x_{n_k} - x_{N}\rangle = \langle u, z - x_N\rangle < +\infty,
    \end{equation*} which implies that $x_n \to z$. Finally, for all $b \in B[0, 1]$, we have \begin{equation*}
        \left\langle x_{n+1} - x_n , \, u + \frac{\delta}{4} b + z -  \frac{x_{n+1} + x_n}{2}\right\rangle \geq \left(\frac{\delta}{4} - \frac{\|x_{n+1} + x_{n} - 2z\|}{2} \right)\|x_{n+1} - x_n\| > 0
    \end{equation*} for all $n$ sufficiently large. In particular, due to \cref{hs} and \cref{hsp-lemma}, the sequence $(x_n)_{n \in \NN}$ is \fejer* monotone with respect to $M := u + B[0, \delta/4]+ z$.

    That said, as evident from the proof, the same reasoning used to establish the implication ``\cref{lemma-cone-inf2}$\Rightarrow$\cref{lemma-cone-inf1}'' also verifies the inclusion
    \begin{equation*}
        \operatorname{int}(\mathcal{K}) + z \subseteq \operatorname{int}(\underline{\lim} \, C_\NN),
    \end{equation*} which, combined with \cref{cone-th}, consequently establishes \cref{cone-int-equal-inf}. \end{proof}

    \begin{remark}
        Recalling \cref{limsup-sup-prop}, consider the support function of the set $\mathcal{D}$ given by \begin{equation*}
        \sigma_\mathcal{D}(x) = \sup_{w \in \mathcal{D}} \langle w, x\rangle = \limsup_{n \to \infty} \langle x, v_n\rangle.
    \end{equation*} Then the closed convex cone $\mathcal{K}$ in \cref{cone-int-equal-inf} can be found as \begin{equation*}
        \mathcal{K} = (\overline{\operatorname{cone}}(\mathcal{D}))^{\ominus} = \operatorname{lev}_{\leq 0}\sigma_\mathcal{D} = \{y \in X \mid \limsup_{n \to \infty} \langle y, v_n\rangle\leq 0\}.
    \end{equation*}
    \end{remark}

\begin{remark}
The assumption of distinct consecutiveness imposed on $(x_n)_{n \in \NN}$ is purely decorative and can be neglected. Specifically, if $x_n = x_N$ for all $n \geq N$, then \cref{cone-int-equal-inf} clearly holds with $\mathcal{K}=X$. 
Furthermore, 
if $X$ is finite-dimensional, then distinct consecutiveness always implies that $0 \not \in \mathcal{D}$, 
and, in light of \cref{pointed-sph-cond-remark}, it suffices to assume that $\mathcal{K}^{\ominus} = \operatorname{cone}\operatorname{conv} \mathcal{D}$ is pointed instead of the solidness of $\mathcal{K}$. 
\end{remark}

\begin{remark}
    If a sequence $(x_n)_{n \in \NN}$ fails to satisfy \cref{lemma-cone-inf1} or \cref{lemma-cone-inf2} in \cref{lemma-cone-inf}, the corresponding  maximal \fejer* set may not be a cone at all as demonstrated next.
\end{remark}

\begin{example}\label{example3}
    Consider the sequence $(x_n)_{n \in \NN}$ defined recurrently in $\mathbb{R}^2$ as follows: 

    Set $x_0 := (0, 1)$. Then, depending on the currently given index $n \in \NN$: \begin{itemize}
            \item If $n \equiv 0 \: (\operatorname{mod} 3)$, construct a circle $S_n$ centered at the point $z_n :=((-1)^n, 0)$ such that it passes through the point $x_n$. When $n$ is odd, let $y_n$ be the rightmost point where $S_n$ intersects the $x$-axis; if $n$ is even, let $y_n$ be the leftmost such point. Then, define $x_{n+1} \in S_n$ as the unique point on the internal angle bisector of the angle $\angle x_nz_ny_{n}$ that lies on the circle. Equivalently, $x_{n+1}$ is the midpoint of the minor arc between $x_n$ and $y_n$ on $S_n$.

            \item otherwise, set $x_{n+1}$ to be the reflection of the previous point $x_{n}$ with respect to the $x$-axis.
        \end{itemize} Then the following hold: \begin{enumerate}
        \item\label{example3-prop-1} $\underline{\lim} \, C_\NN = \cap_{n \in \NN} C_n= [-1, 1] \times \{0\}$.
        \item\label{example3-prop-2} $x_n \to z \in \ri(\underline{\lim} \, C_\NN)$ and $\mathcal{K} = \mathbb{R} \times \{0\}$. In particular, $\operatorname{int}(\mathcal{K}) = \varnothing$.
    \end{enumerate}
\end{example}

\begin{proof}
    See Appendix~\ref{app-example3-prop}.
\end{proof}

\begin{figure}[t]
    \centering
    \includegraphics[width=.90\linewidth]{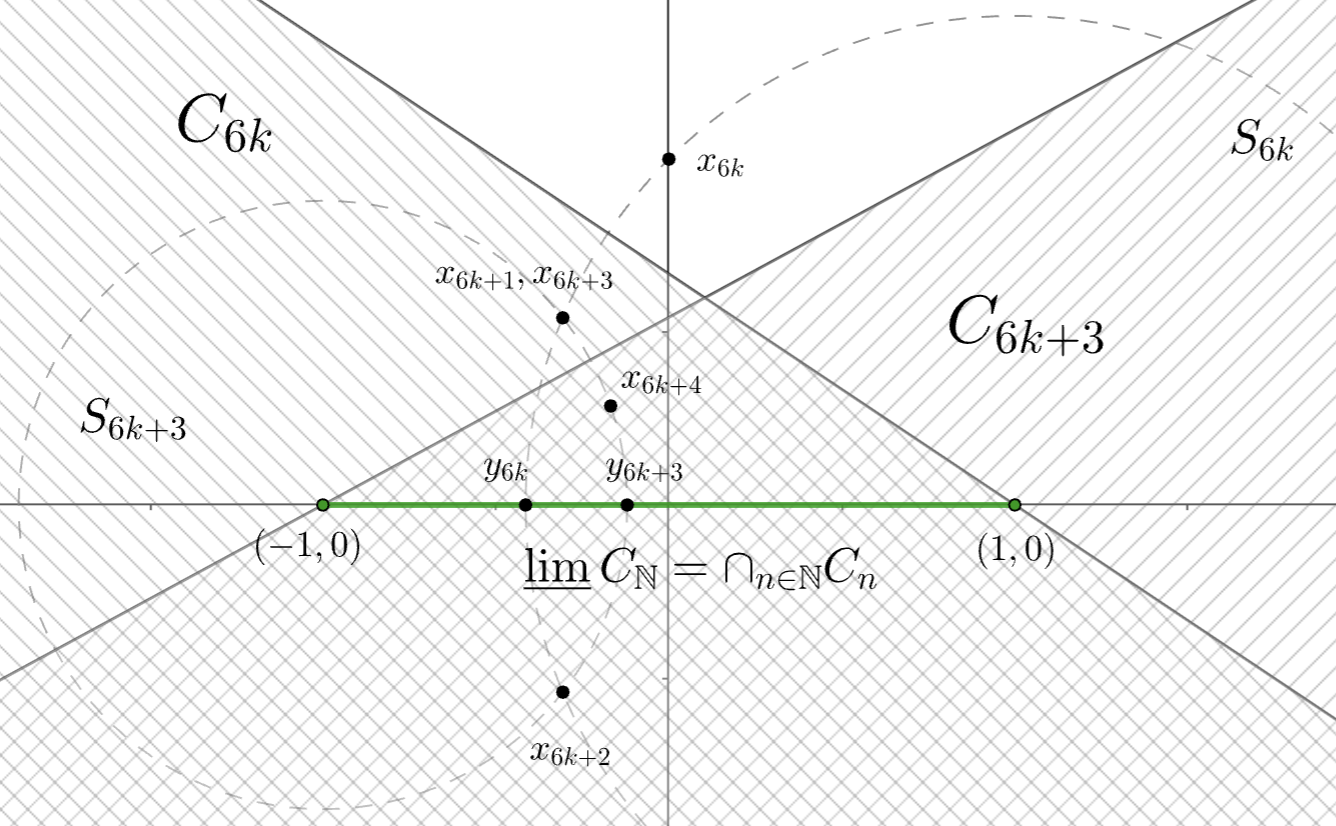}
    \caption{First four iterations of the sequence $(x_n)_{n \in \NN}$ as introduced in \cref{example3}.}
    \label{fig:examples3-pic}
\end{figure}

\begin{remark}\label{example3-remark}
    \cref{fig:examples3-pic} shows the first four iterations (i.e., for $k = 0$) of the sequence $(x_n)_{n \in \NN}$, as introduced in \cref{example3}. The hatched regions represent the halfspaces $C_{6k}$ and $C_{6k+3}$, with different hatching angles used to distinguish them. As in \cref{fig:examples1a-pic}, the dark green line depicts the maximal \fejer* set $\underline{\lim} \, C_\NN$, which, in this case, coincides with the maximal \fejer\ set $\cap_{n \in \NN} C_n$.
\end{remark}

\begin{example}\label{closure-ex-cone}
    For the sequence $(x_n)_{n \in \NN}$ given by \cref{closure_example}, according to \cref{lemma-cone-inf}, we have $\lim_{n \to \infty} x_n = (1, 0)$ and \begin{equation*}
        \underline{\lim} \, C_\NN = (\mathbb{R}\times\mathbb{R}_{++}) \cup (\mathbb{R}_+ \times\{0\}), \quad \mathcal{K} = \overline{\underline{\lim} \, C_\NN} - (1, 0) = \{(x, y) \in \mathbb{R}^2 \mid y \geq 0\}.
    \end{equation*} Here, $\lim_{n \to \infty} x_n = (1, 0)$ belongs to the maximal \fejer* set $\underline{\lim} \, C_\NN$.
\end{example}

\begin{example}\label{authors-example}
    In the paper \cite{BBCIARS}, where the concept of \fejer* monotonicity was originally analyzed, the authors provided a different interesting example \cite[Example 3.1]{BBCIARS} to demonstrate the extension issue as outlined in \cref{problem-closure}. The sequence $(x_{n})_{n \in \NN}$ with distinct consecutive terms was given as follows: 
    
    Define the sequence $\left(x_n\right)_{n \in \mathbb{N}}$ in $\mathbb{R}^2$ by $x_0:=(0,2)$ and for $\ell \in \mathbb{N}$ $$
    x_{2 \ell+1}:=x_{2 \ell}+\left(1 / 2^{\ell}, 0\right), \quad x_{2 \ell+2}:=\left(0, \sqrt{\lVert x_{2 \ell+1}-(1,0)\rVert^2-1}\right)
    $$ If $n$ is even, $x_{n+1}$ is obtained by moving $1 / 2^{(n / 2)}$ horizontally to the right. If $n$ is odd, $x_{n+1}$ is the intersection between the vertical axis and the arc centered in $(1,0)$ passing through $x_n$.

    Then, the sequence $(x_n)_{n \in \mathbb{N}}$ is well-defined, and the maximal \fejer* set $\underline{\lim}\,C_{\NN}$ is given by \begin{equation*}
    \underline{\lim} \, C_\NN  = \left\{ (0, 2/\sqrt{3}) + \operatorname{int}(\mathcal{K}) \right\} \cup \left((0, 2/\sqrt{3}), \left(1,0\right)\right],
    \end{equation*} where the closed convex cone $\mathcal{K} = \left(\overline{\operatorname{cone}}(\mathcal{D})\right)^{\ominus}$ satisfies \begin{equation*}
        \mathcal{K} = \left\{(x, y) \in \mathbb{R}^2 \,\middle| \, x \geq 0, \, y \leq -\frac{2}{\sqrt{3}} x \right\}.
    \end{equation*} In particular, in contrast to \cref{closure-ex-cone}, $\lim_{n \to \infty} x_n = (0, 2/\sqrt{3})$ does not belong to the maximal \fejer* set $\underline{\lim} \, C_\NN$ in this case.
\end{example}

\begin{figure}[t]
    \centering
    \includegraphics[width=.70\linewidth]{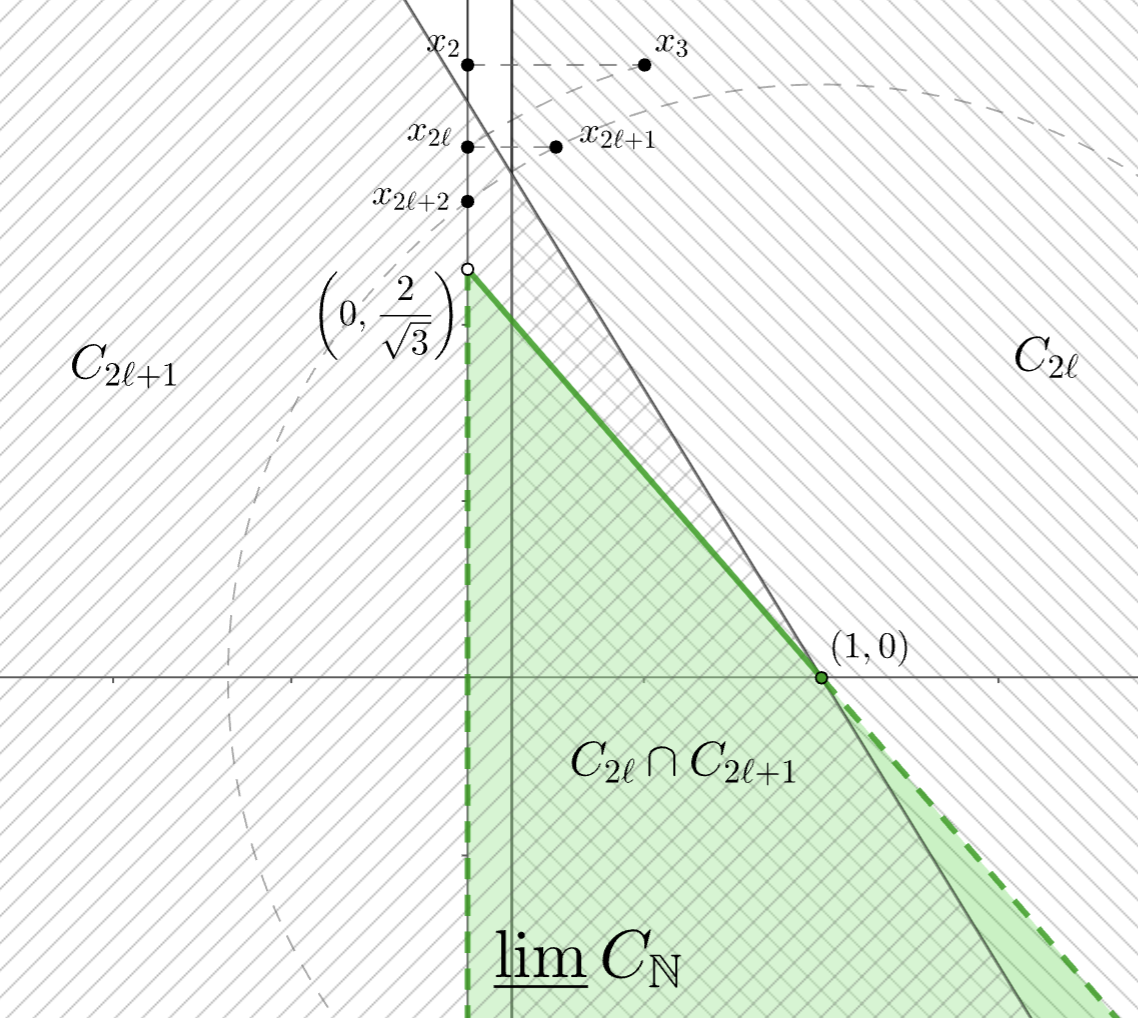}
    \caption{Intermediate iterations of the sequence $(x_n)_{n \in \NN}$ defined in \cref{authors-example}.}
    \label{fig:examples4-pic}
\end{figure}

\begin{proof} The fact that $(x_n)_{n \in \NN}$ is well-defined was established in \cite[Example 3.1]{BBCIARS}, where it was also shown that $\lim_{n \to \infty} x_n = (0, 2/\sqrt{3})$.

That said, the description of its maximal \fejer* set $\underline{\lim} \, C_\NN$ follows directly from the geometric construction of the sequence $(x_n)_{n \in \NN}$ as the halfspaces $C_n$ and their corresponding bounding hyperplanes $\operatorname{bdry}(C_n)$ are explicitly identifiable; see \cref{fig:examples4-pic}. 
\end{proof}

\begin{remark}
    \cref{fig:examples4-pic} shows a few iterations (i.e., for $\ell = 2$) of the sequence $(x_n)_{n \in \NN}$, as introduced in \cref{authors-example}. Similar to \cref{example3-remark}, the hatched regions represent the halfspaces $C_{2\ell}$ and $C_{2\ell + 1}$, with different hatching angles used to distinguish them. As in \cref{fig:examples1a-pic}, the green region illustrates the maximal \fejer* set $\underline{\lim} \, C_\NN$, while the solid and dashed dark green lines on its boundary indicate which boundary points are included in the sets, as specified in \cref{authors-example}.
\end{remark}

To conclude \cref{section3}, we address the Opial counterpart of \cref{lemma-cone-inf}. Set $0$ to be the origin in $X$ and denote by $\mathcal{W} := \mathcal{W}((x_n)_{n \in \NN})$ the set of weak cluster points for the given sequence $(x_n)_{n \in \NN}$. Let $\mathcal{L} := \overline{\aff}(\mathcal{W}) - \overline{\aff}(\mathcal{W})$ to be the closed linear subspace associated with $\overline{\aff}(\mathcal{W})$, i.e. \begin{empheq}[box=\mybluebox]{equation}\label{cluster-aff-decomp}
    \overline{\aff}(\mathcal{W}) = \mathcal{L} + a, \quad \text{where} \; a \in \overline{\aff}(\mathcal{W}).
\end{empheq}

\begin{fact}[location of weak cluster points]\label{opial-wperp}
    Let $(x_n)_{n \in \NN}$ be a sequence in $X$ that is Opial with
respect to a nonempty subset $C$ of $X$. Then \begin{equation*}
    (\forall w_1, w_2 \in \mathcal{W}) \quad w_1 - w_2 \in (C-C)^{\perp}.
\end{equation*}
\end{fact}

\begin{proof}
See \cite[Theorem 2.2]{Opial}.
\end{proof}

\begin{theorem}[maximal Opial set]\label{max-opial-set}
    Let $(x_n)_{n \in \NN}$ be an Opial sequence with respect to a nonempty subset $C$ of $X$. Then its maximal Opial set $\mathcal{O}$ is given by the closed affine subspace: \begin{equation*}
        \mathcal{O} = \mathcal{L}^{\perp} + z,
    \end{equation*} where $z = P_{\mathcal{O}}(0) + P_{\overline{\aff}(\mathcal{W})}(0)$ is the unique point of $\overline{\aff}(\mathcal{W})$ such that $(x_n)_{n \in \NN}$ is Opial with respect to~$z$.
\end{theorem}

\begin{proof}
    For brevity, we let $\mathcal{A} := \overline{\aff}(\mathcal{W})$. Note that, in light of \cref{opial-ext-aff}, $\mathcal{O}$ is a closed affine subspace of $X$. That said, let $V:= \mathcal{O}-\mathcal{O}$ to be the closed linear subspace of $X$ parallel to $\mathcal{O}$, i.e., \begin{equation*}\label{max-opial-set-eq}
        \mathcal{O} = \{y \in X \mid \lim_{n \to \infty} \|x_n - y\| \, \text{exists}\} =  V + b, \quad \text{for some} \; b \in \mathcal{O}.
    \end{equation*} Then, since $\|x_n - y\|^2 = \|x_n - b\|^2 - 2\langle x_n, y-b \rangle + \|y\|^2 - \|b\|^2$, where $\lim_{n \to \infty}\|x_n - b\|^2$ exists, we have \begin{equation*}
        V = \{v \in X \mid \lim_{n \to \infty} \langle x_n, v\rangle \; \text{exists}\}.
    \end{equation*} That said, note that \cref{opial-wperp} implies $\mathcal{L} \subseteq V^{\perp}$. We now show that $\mathcal{L}^\perp \subseteq V$; equivalently, $X \setminus V \subseteq X\setminus \mathcal{L}^\perp$. Take $v \in X\setminus V$. Then \begin{equation*}
        (\exists w_1, w_2 \in \mathcal{W}) \quad \langle w_1, v \rangle = \liminf_{n \to \infty}\langle x_n, v\rangle < \limsup_{n \to \infty}\langle x_n, v\rangle = \langle w_2, v \rangle. 
    \end{equation*} Hence $\langle w_2 - w_1, v\rangle > 0$, and so $v \not \in \mathcal{L}^{\perp}$. We have shown that \begin{equation}\label{max-opial-set-claim-proved}
        V = \mathcal{L}^{\perp}.
    \end{equation} 
    Finally, $\mathcal{O}$ and $\mathcal{A}$ are two closed affine subspaces of $X$ with parallel linear subspaces $V$ and $\mathcal{L}^{\perp}$ being orthogonal complements of each other. Next, we show that $\mathcal{O} \cap \mathcal{A}$ is the singleton $\{P_{\mathcal{O}}(0) + P_{\mathcal{A}}(0)\}$. Indeed, note that \begin{equation}\label{max-opial-origin-proj}
        P_{\mathcal{A}}(0) \in \mathcal{L}^{\perp}, \quad P_{\mathcal{O}}(0) \in \mathcal{L}.
    \end{equation} That said, the point $z := P_{\mathcal{O}}(0) + P_{\mathcal{A}}(0)$ lies in $\mathcal{O}\cap\mathcal{A}$ because \begin{equation*}
        \begin{aligned}
            P_{\mathcal{O}}(0) + P_{\mathcal{A}}(0) &\in \mathcal{L} + \mathcal{A} = \mathcal{A} - \mathcal{A} + \mathcal{A} = \mathcal{A}, \\
            P_{\mathcal{O}}(0) + P_{\mathcal{A}}(0) &\in \mathcal{L}^\perp + \mathcal{O} = V + \mathcal{O} = \mathcal{O} - \mathcal{O} + \mathcal{O} = \mathcal{O}.
        \end{aligned}
    \end{equation*} Now pick $p \in \mathcal{O}\cap\mathcal{A}$. Then, due to \cref{cluster-aff-decomp} and \cref{max-opial-set-claim-proved}, \begin{equation*}
        (\exists v \in \mathcal{L})(\exists v^{\perp} \in \mathcal{L}^{\perp}) \quad p = v + P_{\mathcal{A}}(O) = v^{\perp} + P_{\mathcal{O}}(O).
    \end{equation*} Hence, by applying \cref{max-opial-origin-proj} together with the uniqueness of the orthogonal decomposition, it follows that $v = P_{\mathcal{O}}(0)$ and $v^{\perp} = P_{\mathcal{A}}(0)$. Altogether, $p = z = P_{\mathcal{O}}(0) + P_{\mathcal{A}}(0)$. 
    \end{proof}

\section{Relationship to quasi-\fejer\ types}
\label{section6}

The question of relationships between quasi-\fejer\ types of monotonicity, analyzed by Combettes in his well-known paper \cite{Comb}, and the recently introduced \fejer* monotonicity was originally posed and partially analyzed by Behling, Bello-Cruz, Iusem, Liu, and Santos in \cite{BBCIARS}. In this section, we address these relationships highlighting previously-done progress in \cite{BBCIARS} as well as providing new results and counterexamples.

\begin{definition}[quasi-\fejer\ types] \label{def-quasi-types} Let $M$ be a nonempty subset of $X$. With respect to $M$, a given sequence $(x_n)_{n \in \mathbb{N}}$ of points in $X$ is called \emph{quasi-Fejér monotone of} \begin{itemize}
    \item \emph{Type I} if there exists $(\varepsilon_n)_{n \in \mathbb{N}} \in \ell^1_{+}(\mathbb{N})$ such that \begin{equation*}
(\forall y \in M)(\forall n \in \NN) \quad \|x_{n+1}-y\| \leq \|x_n-y\| + \varepsilon_n,
\end{equation*}

    \item \emph{Type II} if there exists $(\varepsilon_n)_{n \in \mathbb{N}} \in \ell^1_{+}(\mathbb{N})$ such that \begin{equation*}
(\forall y \in M)(\forall n \in \NN) \quad \|x_{n+1}-y\|^2 \leq \|x_n-y\|^2 + \varepsilon_n,
\end{equation*}
    \item \emph{Type III}  if for any given $y \in M$ there exists $(\varepsilon_n(y))_{n \in \mathbb{N}} \in \ell^1_{+}(\mathbb{N})$ such that \begin{equation*}
(\forall n \in \NN) \quad \|x_{n+1}-y\|^2 \leq \|x_n-y\|^2 + \varepsilon_n(y).
\end{equation*}
\end{itemize}
\end{definition}

\begin{remark}\label{connections-quasi}
    Some connections between quasi-types are known. For instance, Type III is the most general one, as both Type I and Type II imply Type III. Moreover, if the set $M$ under consideration is bounded, then Type I implies Type II. Standard \fejer\ monotonicity is clearly stronger than any of these quasi-types. For more details, see \cite[Proposition 3.2]{Comb}.  In addition, we highlight that every property established in \cite{Comb} for quasi-\fejer\ Type III sequences is also valid for Opial sequences.
\end{remark}

\begin{definition}[tolerance functions]\label{def-tol-fn}
    Let $(x_n)_{n \in \mathbb{N}}$ be a sequence of points in $X$ and consider functions $\varepsilon_{n}^{(1)}, \varepsilon_{n}^{(2)}: X \rightarrow \mathbb{R}_{+}$ given for all $n \in \NN$ by \begin{empheq}[box=\mybluebox]{equation*}\begin{aligned}
        \varepsilon_n^{(1)}(y) &:= \max \{0, \|x_{n+1} - y\| - \|x_n - y\|\}, \\ \varepsilon_n^{(2)}(y) &:= \max \{0, \|x_{n+1} - y\|^2 - \|x_n - y\|^2\}.
    \end{aligned}
\end{empheq}
\end{definition}

\begin{remark}\label{remark-tol-fn}
    It is natural to consider these tolerance functions since, due to their definition, we have \begin{equation*}
    (\forall y \in X\setminus C_n)(\forall n \in \NN) \quad \|x_{n+1} - y\| = \|x_n - y\| + \varepsilon_n^{(1)}(y),
\end{equation*} and, moreover, \begin{equation*}
    y \in C_n \quad \text{if and only if} \quad \varepsilon_n^{(1)}(y) = 0 \; (\text{or, equivalently,} \; \varepsilon_n^{(2)}(y) = 0)
\end{equation*} Naturally, the same applies to $\varepsilon_n^{(2)}$ for the corresponding equalities with squared norms instead.
\end{remark}

\begin{remark}\label{remark-eps-1}
For any chosen set $M$ and a sequence $(x_n)_{n \in \NN}$, one always has \begin{equation*}
(\forall n \in \NN) \quad \sup_{y \in M} \varepsilon_n^{(1)}(y) \leq \|x_{n+1} - x_n\| < + \infty,
\end{equation*} which is entirely contrary to what may occur with $\varepsilon_n^{(2)}$ as demonstrated by \cref{count-type2-prop}.
\end{remark}

\begin{proposition}\label{prop-rewritten-tol}
Let $M$ be a nonempty subset of $X$ and $(x_n)_{n \in \mathbb{N}}$ be a sequence of points in $X$. Then: \begin{enumerate}
\item\label{fejer*-rewritten} $(x_n)_{n \in \mathbb{N}}$ is \fejer* monotone with respect to $M$ if and only if \begin{equation*}
        (\forall y \in M)(\exists N(y) \in \NN)(\forall n \geq N(y)) \quad \varepsilon_n^{(1)}(y) = \varepsilon_n^{(2)}(y) = 0.
    \end{equation*}
\item\label{12-rewritten} $(x_n)_{n \in \mathbb{N}}$ is quasi-\fejer\ monotone of Type I (of Type II, respectively) with respect to $M$ if and only if \begin{equation*}
        \sum_{n = 0}^{\infty} \; \sup_{y\in M}\varepsilon_n^{(1)}(y) < +\infty \quad \left(\sum_{n = 0}^{\infty} \; \sup_{y\in M}\varepsilon_n^{(2)}(y) < +\infty, \; \text{respectively}\right).
    \end{equation*}
\item\label{3-rewritten} $(x_n)_{n \in \mathbb{N}}$ is quasi-\fejer\ monotone of Type III with respect to $M$ if and only if \begin{equation*}
        (\forall y \in M) \quad \sum_{n = 0}^{\infty} \varepsilon_n^{(2)}(y) < +\infty.
    \end{equation*}
\end{enumerate}
\end{proposition}

\begin{proof}
The result follows from \cref{def-quasi-types}, \cref{def-tol-fn}, and \cref{remark-tol-fn}. \end{proof}

Following Behling, Bello-Cruz, Iusem, Liu, and Santos (see \cite[Theorem 4.1]{BBCIARS}), we highlight the following observation.

\begin{corollary}\label{*implies3}
    Let $(x_n)_{n \in \mathbb{N}}$ be a \fejer* monotone sequence with respect to a nonempty subset $M$ of $X$. Then $(x_n)_{n \in \mathbb{N}}$ is quasi-\fejer\ of Type III with respect to $M$.
\end{corollary}

\begin{proof}
    Indeed, the series \begin{equation*}
        \sum_{n = 0}^{\infty} \varepsilon_n^{(2)}(y) = \sum_{n < N(y)}\varepsilon_n^{(2)}(y)
    \end{equation*} converges for all $y \in M$ as it contains at most $N(y)$ nonzero terms (see \cref{prop-rewritten-tol}\cref{fejer*-rewritten}). Therefore, the result follows due to \cref{prop-rewritten-tol}\cref{3-rewritten}.
\end{proof}

Next, due to \cref{raik-*}, we highlight the following particular case in which \fejer* monotonicity implies quasi-\fejer\ Type II.

\begin{proposition}\label{*implies2}
    Let $(x_n)_{n \in \NN}$ be a \fejer* monotone sequence with respect a nonempty convex subset $M$ of $X$. Suppose that $M$ is bounded and $\ri(M) \neq \varnothing$. Then $(x_n)_{n \in \NN}$ is quasi-\fejer\ monotone of Type II with respect to $M$.
\end{proposition}

\begin{proof}
    To see this, rewrite the definition of $\varepsilon_n^{(2)}$ as follows:
    \begin{equation}
    \label{e:251211a}
        \varepsilon_n^{(2)}(y) = \max\{0, \langle x_{n+1} - x_n, \,  x_{n+1} + x_n - 2y \rangle\}.
    \end{equation}
    Pick any $y_0 \in M$. Then, for all $n \geqslant N := N(y_0)$ and all $y \in M$, \cref{prop-rewritten-tol}\cref{fejer*-rewritten} implies
    \begin{equation*}
    \begin{aligned}
        0 \leq \varepsilon_n^{(2)}(y) = \varepsilon_n^{(2)}(y) - \varepsilon_n^{(2)}(y_0) &\leq 2 |\langle y - y_0, \, x_{n+1} - x_n\rangle| \\ &\leq 2 \operatorname{diam}(M)\|P_{\overline{\aff}(M)}(x_{n+1}) - P_{\overline{\aff}(M)}(x_n)\|,
    \end{aligned}
    \end{equation*}
    where the last step is due to the fact that $y, y_0 \in \overline{\aff}(M)$ and $\|y - y_0\| \leq \operatorname{diam}(M)$. Hence \cref{raik-*} implies that \begin{equation*}
        (\exists N \in \NN) \quad \sum_{n \geq N} \, \sup_{y \in M} \varepsilon_{n}^{(2)}(y) \leq 2 \operatorname{diam}(M) \sum_{n \geq N} \|P_{\overline{\aff}(M)}(x_{n+1}) - P_{\overline{\aff}(M)}(x_n)\| < +\infty.
    \end{equation*} However, as the given set $M$ is assumed to be convex and bounded (so $\overline{M}$ is weakly compact) and the tolerance functions $\varepsilon_{n}^{(2)}$ are weakly continuous (see \cref{e:251211a}), we get for all $n \in \{0, \ldots, N-1\}$ \begin{equation*}
        \sup_{y \in M} \varepsilon_{n}^{(2)}(y) < +\infty. 
    \end{equation*} Therefore, the result follows due to \cref{prop-rewritten-tol}\cref{12-rewritten}.
\end{proof}

\begin{remark}\label{count-type2-remark}
    As \cref{count-type2-prop} and \cref{prop-ind-count-quasi} demonstrate next, both of the assumptions, on the boundedness of $M$ and nonemptiness of $\ri(M)$, are essential for the conclusion of \cref{*implies2} to hold.
\end{remark}

\begin{example}\label{count-type2-prop}
    For the sequence $(x_n)_{n \in \NN}$ given by \cref{closure_example}, we have \begin{equation*}
        (\forall n \in \NN) \quad \sup_{y \in \underline{\lim} \, C_\NN} \varepsilon_n^{(2)} (y) = + \infty.
    \end{equation*} In particular, $(x_n)_{n \in \NN}$ is \emph{not} quasi-\fejer\ monotone of Type II with respect to its \emph{unbounded} maximal \fejer* set $\underline{\lim} \, C_\NN$, which has nonempty interior.
\end{example}

\begin{proof}
     For $\rho > 0$ and $n \in \NN$ fixed, set \begin{equation}\label{parr-lines}
         l_n(\rho) := \{y \in X \mid \|x_{n+1} - y\|^2 = \|x_n - y \|^2 + \rho\}.
     \end{equation} We claim that $l_n(\rho) \cap \underline{\lim} \, C_\NN \neq \varnothing$ for any $\rho > 0$ and $n \in \NN$ chosen. Indeed, by taking squares and rearranging the terms in \cref{parr-lines}, we conclude that \begin{equation*}
         l_n(\rho) = \left\{ y \in X \,\middle|\, \langle y, x_{n+1} - x_{n} \rangle = \frac{1}{2}(\|x_{n+1}\|^2 - \|x_n\|^2) - \frac{\rho}{2}\right\},
     \end{equation*} which represents a line parallel to the line $\operatorname{bdry}(C_n)$ but shifted, according to \cref{hsp-lemma}. In particular, as $l_n(\rho)$ has the same negative slope, no matter how big the gap between these two lines is, $l_n(\rho)$ will always intersect $\mathbb{R}\times\mathbb{R_{++}} \subset \underline{\lim} \, C_\NN$, where the last inclusion is due to \cref{closure_example}. Therefore, our claim is proved, and \begin{equation*}
         (\forall n \in \NN) \quad \sup_{y \in \underline{\lim} \, C_\NN} \varepsilon_n^{(2)} (y) \geq \rho.
     \end{equation*} However, as the parameter $\rho > 0$ was chosen arbitrarily, the result follows.
\end{proof}

\begin{example}\label{prop-ind-count-quasi}
    For the sequence $(x_n)_{n \in \NN}$ given in \cref{infdim-countex}, consider the same target set \begin{equation*}
        M = \operatorname{span}\,\{e_n\}_{n \geq 1} = \{\lambda_1e_1 + \lambda_2e_2 + \cdots + \lambda_me_m \mid m \in \NN_{++}, \; \lambda_1, \ldots , \lambda_m \in \mathbb{R}\}.
    \end{equation*} Then $M = \underline{\lim} \, C_\NN$. However, \begin{equation}\label{tol-func-inf-count}
        (\forall n \in \NN) \quad \max_{y \in B[0, 1] \cap M} \varepsilon_n^{(2)} (y) = 2.
    \end{equation} In particular, $(x_n)_{n \in \mathbb{N}}$ is neither quasi-\fejer\ monotone of Type I nor II with respect to $B[0, 1] \cap M$, which is a bounded set with empty relative interior.
\end{example}

\begin{proof}
    The equality $M = \underline{\lim} \, C_\NN$ was covered previously in \cref{infdim-countex}. Now, take any $y \in M$. Repeating the same arguments, we get \begin{equation*}
    (\forall k \in \NN) \quad \begin{aligned}
        &\|x_{2k+1} - y\|^2 - \|x_{2k} - y\|^2 
    = 1 + y_k^2 - (1 - y_k)^2 = 2y_k, \\ &\|x_{2(k+1)} - y\|^2 - \|x_{2k+1} - y\|^2 
    = (1 - y_{k+1})^2 - 1 - y_{k+1}^2 = -2y_{k+1}.
    \end{aligned}
    \end{equation*} Hence, for all $k \in \NN$ and $n = 2k, 2k+1$, the maximum in \cref{tol-func-inf-count} is equal to $2$ and attained at $y = e_k, -e_{k+1} \in M$, respectively. Consequently, the final claim follows immediately by combining \cref{connections-quasi} with \cref{prop-rewritten-tol}\cref{12-rewritten}.
\end{proof}

In view of \cref{connections-quasi}, \cref{*implies2} indicates that \fejer* monotonicity may be related in some way to Type I of quasi-\fejer\ monotonicity.

\begin{corollary}\label{nonep-int-type1}
    Let $(x_n)_{n \in \NN}$ be a \fejer* monotone sequence with respect to a nonempty convex subset $M$ of $X$. Suppose that $\operatorname{int}(M) \neq \varnothing$. Then $(x_n)_{n \in \NN}$ is quasi-\fejer\ monotone of Type I with respect to $X$. In particular, $(x_n)_{n \in \NN}$ is quasi-\fejer\ monotone of Type II with respect to any bounded subset $C$ of $X$.
\end{corollary}

\begin{proof}
    Since $\operatorname{int}(M) \neq \varnothing$, \cref{raik-*} together with \cref{remark-eps-1} implies \begin{equation*}
       \sum_{n = 0}^{\infty} \, \sup_{y \in M} \varepsilon_n^{(1)}(y) \leq \sum_{n = 0}^{\infty} \|x_{n+1} - x_n\| < +\infty.
    \end{equation*} Consequently, the results follow directly from \cref{prop-rewritten-tol}\cref{12-rewritten} and \cref{connections-quasi}.
\end{proof}

\begin{remark}
    In light of \cref{lemma-cone-inf}, we can use any sequence $(x_n)_{n \in \NN}$ with distinct consecutive terms that is \fejer* monotone with respect to the set $M$ with $\operatorname{int}(M) \neq \varnothing$ to get counterexamples showing that quasi-\fejer\ monotonicity of Types I or II does not generally imply \fejer* monotonicity. Indeed, take any such sequence $(x_n)_{n \in \NN}$ (for instance, as presented in \cref{authors-example}). Then, due to \cref{nonep-int-type1}, we have: \begin{equation*}
    \begin{aligned}
        (x_n)_{n \in \mathbb{N}} \text{ is quasi-\fejer\ monotone of Type I with respect to } 
        X \setminus \overline{\underline{\lim} \, C_\NN} \neq \varnothing \\
        \text{but \emph{not} \fejer* monotone with respect to any point } x \in X \setminus \overline{\underline{\lim} \, C_\NN}.
    \end{aligned}
\end{equation*} Moreover, for any nonempty \emph{bounded} subset $C$ of $X\setminus\overline{\underline{\lim} \, C_\NN}$, we get: \begin{equation*}
    \begin{aligned}
        (x_n)_{n \in \mathbb{N}} \text{ is quasi-\fejer\ monotone of Type II with respect to } 
        C \subsetneq X \setminus \overline{\underline{\lim} \, C_\NN} \\
        \text{but \emph{not} \fejer* monotone with respect to any point } x \in C.
    \end{aligned}
\end{equation*} 
We highlight that, in the setting of \cref{authors-example}, the sequence $(x_n)_{n \in \NN}$ is quasi-\fejer\ Type II with respect to $M := (0,1]\times{0}$, even though $\overline{M} \not\subseteq \underline{\lim} \, C_{\NN}$. This does not lead to any contradiction, contrary to the claim made by Behling et al. in \cite{BBCIARS}.
\end{remark}

By combining Raik’s proposition with the corollaries obtained from directional asymptotics, we now show that \fejer* monotonicity implies Type I under a fairly general condition:

\begin{theorem}\label{*implies1}
    Let $(x_n)_{n \in \NN}$ be a \fejer* monotone sequence with respect to a nonempty convex subset $M$ of $X$ such that $\ri(M) \neq \varnothing$. Suppose that \begin{equation*}
        \text{either} \quad \lim_{n \to \infty} d_{\overline{M}}(x_n) > 0 \quad \text{or} \quad (x_n)_\nnn\; \text{ converges to some } z \in \overline{M} \cap \underline{\lim} \, C_{\NN}.
    \end{equation*} Then $(x_n)_{n \in \NN}$ is quasi-\fejer\ monotone of Type I with respect to $M$. 
\end{theorem}

\begin{proof}
According to \cref{remark-eps-1} and \cref{prop-rewritten-tol}\cref{12-rewritten}, it suffices to show that \begin{equation}\label{*implies1-series}
\sum_{n \geq N} \, \sup_{y \in M} \varepsilon_{n}^{(1)}(y) < +\infty
\end{equation} for some $N \in \NN$. In view of \cref{dist-remark}, we assume WLOG that $(x_n)_{n \in \NN}$ has distinct consecutive terms (for otherwise $\varepsilon_n^{(1)}\equiv 0$ eventually and we are done). Next, \cref{cycl-prop-inf*2}\cref{cycl-prop-inf*2-2} provides some $N$ such that $x_n \not \in \overline{M}$ for all $n \geq N$. That said, for all $n \geq N$ and all $y \in \overline{M} \setminus C_n$, \begin{equation*}
0 < \varepsilon_n^{(1)}(y) = \frac{\|x_{n+1} - y\|^2 - \|x_{n} - y\|^2}{\|x_{n+1} - y\| + \|x_{n} - y\|} = \frac{\langle x_{n+1} - x_n, x_{n+1} + x_n - 2y \rangle}{\|x_{n+1} - y\| + \|x_{n} - y\|}.
\end{equation*} 

\textit{Case 1}: Suppose $$x_n \to z \in \overline{M} \cap \underline{\lim} \, C_{\NN}.$$ Then, due to \cref{def}, \cref{hsp-lemma}, and \cref{hsp-lemma-remark}, \begin{equation}\label{type1-th-ineq1}
(\exists N(z) \in \NN)(\forall n \geq  N(z)) \quad \langle x_{n+1} - x_n, x_{n+1} + x_n - 2z \rangle \leq 0.
\end{equation} Now, adding and subtracting $2z$ yields for all $n \geq N$ and all $y \in \overline{M} \setminus C_n$ \begin{equation*}\label{type1-th-ineq2}
0 < \varepsilon_n^{(1)}(y) = \frac{2\langle x_{n+1} - x_n, z - y \rangle}{\|x_{n+1} - y\| + \|x_{n} - y\|} + \frac{\langle x_{n+1} - x_n, x_{n+1} + x_n - 2z \rangle}{\|x_{n+1} - y\| + \|x_{n} - y\|}.
\end{equation*} Without loss of generality, let $N \geq N(z)$. Then \cref{type1-th-ineq1} implies \begin{equation*}
0 < \varepsilon_n^{(1)}(y) \leq \frac{2\langle x_{n+1} - x_n, z - y \rangle}{\|x_{n+1} - y\| + \|x_{n} - y\|} = \frac{2\langle P_{\overline{\aff}(M)}(x_{n+1}) - P_{\overline{\aff}(M)}(x_n), z - y \rangle}{\|x_{n+1} - y\| + \|x_{n} - y\|}.
\end{equation*} 
Finally, \cref{add-equiv} provides the following upper bound 
for some $\Gamma \in \mathbb{R}_{++}$: \begin{equation}\label{*implies1-ub}
    (\exists \widetilde{N} \in \NN)(\forall n \geq \widetilde{N})(\forall y \in \overline{M}) \quad \frac{\|z - y\|}{\|x_{n+1} - y\| + \|x_n - y\|} \leq \frac{\|z - y\|}{\|x_n - y\|} \leq \Gamma,
\end{equation} where we similarly may let $N \geq \widetilde{N}$ by taking $N$ even larger if needed. Then, the Cauchy-Schwarz inequality combined with \cref{raik-*} yields 
\begin{equation*}
\sum_{n \geq N} \, \sup_{y \in M} \varepsilon_{n}^{(1)}(y) \leq 2\Gamma  \sum_{n \geq N} \|P_{\overline{\aff}(M)}(x_{n+1}) - P_{\overline{\aff}(M)}(x_n)\| < + \infty
\end{equation*} completing the proof in this case.

\textit{Case 2}: Let $$2\delta := \lim_{n \to \infty} d_{\overline{M}}(x_n) > 0$$ and pick any $z \in M \subseteq\underline{\lim} \, C_\NN$. Proceeding as in the first case, we can likewise verify the convergence of the series \cref{*implies1-series}. What remains 
to be done is to establish an upper bound analogous to \cref{*implies1-ub} for this case as well. To achieve this, suppose that, after enlarging $N$ if needed, \begin{equation*}
(\forall n \geq N)(\forall y \in \overline{M}) \quad \|x_n - y\| \geq \delta > 0.
\end{equation*} Then, \cref{b-prop}\cref{b-prop1} implies that $R := \sup\{(2\|x_n\|)_{n \in \NN}, \|z\|\} < +\infty$, and 
\begin{equation*}
(\forall y \in X\setminus B[0, R]) \quad \frac{\|z - y\|}{\|x_{n+1} - y\| + \|x_n - y\|} \leq \frac{\|z - y\|}{\|x_n - y\|} \leq \frac{\|z\| + \|y\|}{\|y\| - \|x_n\|} \leq \frac{2\|y\|}{\|y\|/2} = 4
\end{equation*} for all $n \geq N$. However, inside the chosen ball $B[0, R]$, we use a different upper bound: \begin{equation*}
(\forall y \in \overline{M} \cap B[0, R]) \quad \frac{\|z - y\|}{\|x_{n+1} - y\| + \|x_n - y\|} \leq \frac{\operatorname{diam}(B[0, R] \cup \{z\})}{2\delta} < + \infty.
\end{equation*} To summarize, for 
\begin{equation*}
\Gamma := \max\left\{\frac{\operatorname{diam}(B[0, R] \cup \{z\})}{2\delta}, \,4\right\},
\end{equation*} 
the upper bound \cref{*implies1-ub} holds in the second case as well completing the proof. \end{proof}

\begin{corollary}\label{c:*implies1}
    Let $(x_n)_{n \in \NN}$ be a \fejer* monotone sequence with respect to a nonempty closed convex subset $M$ of $X$ with 
    $\ri(M) \neq \varnothing$. Then $(x_n)_{n \in \NN}$ is quasi-\fejer\ monotone of Type I with respect to $M$. 
\end{corollary}

\begin{proof}
Because $M$ is closed, we know that 
$\overline{M} = M \subseteq \underline{\lim} \, C_{\NN}$. 
\cref{fejer-opial-lemma} yields that $\lim_{n \to \infty} d_{\overline{M}}(x_n)$ always exists. If $\lim_{n \to \infty} d_{\overline{M}}(x_n) > 0$, then the conclusion follows by \cref{*implies1}. However, if $\lim_{n \to \infty} d_{\overline{M}}(x_n) = 0$, then it still follows as $x_n \to z \in \overline{M} = M = M \cap \underline{\lim} \, C_{\NN} = \overline{M} \cap \underline{\lim} \, C_{\NN}$.
\end{proof}

\begin{remark}
    Without the assumptions of \cref{nonep-int-type1} or \cref{*implies1}, \fejer* monotonicity may not imply quasi-\fejer\ Type I, even in finite dimensions, as shown next, whereas the necessity of the condition $\ri(M)\neq\varnothing$ was previously examined in \cref{prop-ind-count-quasi}.
\end{remark}

\begin{example}\label{type1-countex}
    Suppose $X = \mathbb{R}^2$ and define the sequence $(x_n)_{n \in \NN}$ in $X$ by \begin{equation*}
        (\forall k \in \NN) \quad x_{2k} := \left(\frac{1}{k+1}, 0\right), \; x_{2k + 1} := \left(\frac{1}{k+2}, \sqrt{\frac{2}{(k+1)(k+2)\ln(k+1)}}\right).
    \end{equation*}

    Then the following hold: \begin{enumerate}
        \item\label{type1-countex-prop1} $\underline{\lim} \, C_\NN = \mathbb{R}_{--} \times\{0\}$ and $x_n \to (0, 0) \in \operatorname{bdry}(\underline{\lim} \, C_\NN) \setminus \underline{\lim} \, C_\NN.$
        \item\label{type1-countex-prop2} $(x_n)_{n \in \NN}$ is \emph{not} quasi-\fejer\ monotone of Type I with respect to $\underline{\lim} \, C_\NN$.
    \end{enumerate}
\end{example}

\begin{proof}
    Indeed, fix any $x \in \mathbb{R}$ and denote $M := \mathbb{R}_{--} \times\{0\}$. Then $(x,0) \in C_{2k+1}$ and \begin{equation*}
        \|x_{2k+1} - (x,0)\|^2 - \|x_{2k} - (x,0)\|^2 = \frac{2}{(k+1)(k+2)}\left(\frac{1}{\ln(k+1)} - \frac{2k + 3}{2(k+1)(k+2)} + x\right).
    \end{equation*} In particular, for all $k$ large enough, $\operatorname{sign}(\|x_{2k+1} - (x,0)\|^2 - \|x_{2k} - (x,0)\|^2) = \operatorname{sign}(x)$; hence $(x,0) \in \underline{\lim} \, C_\NN$ if and only if $x \in \mathbb{R}_{--}$, and $M = \mathbb{R} \times \{0\} \cap \underline{\lim} \, C_\NN$. Note that, the maximum in \cref{type1-countex-prop-eq1} over all $x \in \mathbb{R}_{-}$ is attained at $x = 0$; specifically, \begin{equation}\label{type1-countex-prop-eq1}\begin{aligned}
        \max_{x \in \mathbb{R}_{-}}\, (\|x_{2k+1} - (x, 0)&\|^2 - \|x_{2k} - (x, 0)\|^2) = \|x_{2k+1}\|^2 - \|x_{2k}\|^2 \\ &=  \frac{2}{(k+1)(k+2)}\left(\frac{1}{\ln(k+1)} - \frac{2k + 3}{2(k+1)(k+2)}\right) \sim \frac{2}{k^2\ln(k)}.
    \end{aligned} 
    \end{equation} Moreover, by plugging in the same $x = 0$, we additionally achieve \begin{equation}\label{type1-countex-prop-eq2}\begin{aligned}
        \min_{x \in \mathbb{R}_{-}} (\|x_{2k+1} - (x, 0)\| + \|&x_{2k} - (x, 0)\|) = \|x_{2k+1}\| + \|x_{2k}\| \\ &=  \sqrt{\frac{2}{(k+1)(k+2)\ln(k+1)} + \frac{1}{(k+2)^2}} + \frac{1}{k + 1} \sim \frac{2}{k}.
    \end{aligned} 
    \end{equation} That said, \cref{type1-countex-prop-eq1} and \cref{type1-countex-prop-eq2} combined yield for some $N \in \NN$ and all $k \geq N$  \begin{equation}\label{type1-countex-prop-eq3}
        0 < \sup_{x \in \mathbb{R}_{--}} (\|x_{2k+1} - (x, 0)\| - \|x_{2k} - (x, 0)\|) = \frac{\|x_{2k+1}\|^2 - \|x_{2k}\|^2}{\|x_{2k+1}\| + \|x_{2k}\|} \sim \frac{1}{k \ln(k)},
    \end{equation} which consequently implies \begin{equation*}
         \sum_{n = 0}^{\infty} \; \sup_{y\in M}\varepsilon_n^{(1)}(y) \geq  \sum_{k \geq N} \; \sup_{x \in \mathbb{R}_{--}} (\|x_{2k+1} - (x, 0)\| - \|x_{2k} - (x, 0)\|) = + \infty,
    \end{equation*} where the divergence of the last series follows from \cref{type1-countex-prop-eq3} and the fact that $$\sum_{k \geq 2} \frac{1}{k \ln(k)} = +\infty.$$ Finally, \cref{prop-rewritten-tol}\cref{12-rewritten}, together with the equality $M = \mathbb{R} \times \{0\} \cap \underline{\lim} \, C_\NN$, implies that $(x_n)_{n \in \NN}$ fails to be quasi-\fejer\ monotone of Type~I with respect to $M$ or $\underline{\lim} \, C_\NN$. Specifically, due to \cref{nonep-int-type1}, we derive that $\operatorname{int}(\underline{\lim} \, C_\NN) = \varnothing$ and, moreover, $M = \underline{\lim} \, C_\NN$, 
    and we're done. \end{proof}

\section*{Acknowledgments}
The research of HHB was supported by a Discovery Grant from 
 the Natural Sciences and Engineering Research Council of Canada.

\begin{appendices}

\section{The proof for \cref{aff-count}}\label{app-aff-count-p}
\begin{proof}
     First, we prove by induction that all points of the sequence $(x_n)_{n \in \NN}$ belong to $\mathbb{R}_{+}\times\mathbb{R}_{++}$. Indeed, one can verify that the initial points $x_0$ and $x_1$ lie in $\mathbb{R}_{+}\times\mathbb{R}_{++}$. Assume that $x_n$ belongs to $\mathbb{R}_{+}\times\mathbb{R}_{++}$ for some $n \in \NN$. Then, as the right point of $S_n \cap M$, $y_n$ lies in the positive $x$-axis, whereas the point $z_n =(-n, 0)$ belongs to the negative $x$-axis. Consequently, the angle bisector of $\angle x_nz_ny_n$ intersects $S_n$ within the first quadrant, which completes the inductive step and proves the claim.
    
    Next, to simplify our following notes, denote $a_n := \|P_M(x_n)\|$ and note that, by our construction, \begin{equation*}
        (\forall n \in \NN) \quad x_n = (a_n, d_M(x_n)) \quad \text{and} \quad d_M(x_{n + 1}) < d_M(x_{n}).
    \end{equation*} However, since the angle $\alpha_n := \angle x_nz_ny_n$ does not exceed $\pi/2$ in its absolute value and \begin{equation*}
            \left(\forall x \in \left[-\frac{\pi}{2}, \frac{\pi}{2}\right]\right) \quad \frac{\sin(x/2)}{\sin(x)} = \frac{1}{2\cos(x/2)} \leq \frac{1}{2\cos(\pi/2)} = \frac{1}{\sqrt{2}},
        \end{equation*} we have \begin{equation*}
        (\forall n \in \NN) \quad d_M(x_{n + 1}) = \frac{\sin(\alpha_n/2)}{\sin(\alpha_n)}d_M(x_{n}) \leq \frac{1}{\sqrt{2}}d_M(x_{n}),
    \end{equation*} which implies that $d_M(x_{n}) \leq 2^{-n/2} \to 0$.

    ``\cref{aff-count-p1}'': Since $y_n$ is the right point of $S_n \cap M$, and $x_{n+1}$ is defined as the midpoint of the minor arc connecting $x_n$ and $y_n$ on $S_n$, it follows that the projection of $x_{n+1}$ onto $M$ lies strictly between $P_M(x_n)$ and $P_M(y_n)$. More specifically, $a_n < a_{n+1} < \|P_M(y_n)\|$.

    ``\cref{aff-count-p2}'': Note that, by \cref{hsp-lemma}, since $x_n \neq x_{n+1}$ for all $n \in \mathbb{N}$, the set $C_n$ is a closed 
    halfspace whose boundary is the line passing through the midpoint $(x_n + x_{n+1})/2$ and orthogonal to the difference $x_{n+1} - x_n$. However, since both points $x_n$ and $x_{n+1}$ lie on the circle $S_n$ centered at $z_n = (-n, 0)$, this boundary line coincides then with the angle bisector of $\angle x_nz_nx_{n+1}$. In particular, for every $n \in \mathbb{N}$, the boundary line $\operatorname{bdry}(C_n)$ has positive slope since both points $x_n$ and $x_{n+1}$ belong to the first quadrant and $z_n = (-n, 0)$.

    That said, we conclude that $[-k, +\infty)\times\mathbb{R}_- \subseteq C_n$ for any $k \in \NN$ and $n \geq k$ (see \cref{fig:examples1b-pic}). Therefore, \begin{equation}\label{aff-count-incl}
        M \subset \mathbb{R}\times\mathbb{R}_- \subseteq \underline{\lim} \, C_\NN,
    \end{equation} where $\operatorname{int}(\mathbb{R}\times\mathbb{R}_-) \neq \varnothing$. In particular, due to \cref{eq-def} and \cref{raik-*}, the sequence $(x_n)_{n \in \NN}$ is \fejer* monotone with respect to $M$ and $\mathbb{R}\times\mathbb{R}_-$, and $(x_n)_{n \in \NN}$ has a finite-length trajectory, i.e. $$\sum_{n\in \NN} \|x_n - x_{n+1}\| < +\infty.$$ But, since $d_M(x_n) \leq 2^{-n/2} \to 0$, $(x_n)_{n\in \NN}$ converges to some point $z = (\zeta,0) \in M$, where $\zeta>0$. 
    
    ``\cref{aff-count-p3}'': To show the reverse for the right inclusion in \cref{aff-count-incl}, denote $\delta_n := a_{n+1} - a_n$. Since $z_n \in \operatorname{bdry}(C_n)$, note that \begin{equation}\label{aff_count-perp}
        x_n + x_{n+1} - 2z_n  \; \perp \; x_n - x_{n+1},
    \end{equation} where \begin{equation*}
    \begin{aligned}
        &x_n + x_{n+1} - 2z_n = (a_n + a_{n+1} + 2n, d_M(x_n) + d_M(x_{n+1})), \\ &x_n - x_{n+1} = (-\delta_n, d_M(x_n) - d_M(x_{n+1})).
    \end{aligned}
    \end{equation*} Consequently, solving $\langle x_n + x_{n+1} - 2z_n, x_n - x_{n+1}\rangle = 0$, we conclude \begin{equation*}
        (\forall n \in \NN) \quad \delta_n = \frac{d^2_M(x_n) - d^2_M(x_{n+1})}{a_n + a_{n+1} + 2n}.
    \end{equation*} Due to \cref{aff_count-perp}, by adding and subtracting $2z_n$, we can rewrite the explicit form for $C_n$ in \cref{hsp-lemma} as \begin{equation*}
        C_n = \{y \in X \mid \langle y - z_n, x_{n+1} - x_n \rangle \geq 0\}.
    \end{equation*} However, pick any $y = (a, b)$ with $b > 0$, then \begin{equation*}
    \begin{aligned}
        \langle y - z_n, x_{n+1} - x_n \rangle &= \delta_n(a + n) - b(d_M(x_n) - d_M(x_{n+1})) \\ &= (d_M(x_n) - d_M(x_{n+1}))\left[\frac{(a + n)(d_M(x_n) + d_M(x_{n+1}))}{a_n + a_{n+1} + 2n} - b\right] < 0,
    \end{aligned}
    \end{equation*} when $n$ becomes sufficiently large since $d_M(x_n) - d_M(x_{n+1}) > 0$ and \begin{equation*}
        \frac{a + n}{a_n + a_{n+1} + 2n} \rightarrow \frac{1}{2}, \quad d_M(x_n) + d_M(x_{n+1}) \rightarrow 0.
    \end{equation*} Consequently, $y \not \in C_n$ for all $n$ sufficiently large. Finally, we conclude that $y \not \in \underline{\lim} \, C_\NN$, and, since the point $y \in \mathbb{R}\times\mathbb{R}_{++}$ was chosen arbitrarily, the converse to the right inclusion in \cref{aff-count-incl} holds indeed. \end{proof}

\section{The proof for \cref{example-non-decr}}\label{app-example-non-decr-prop}
\begin{proof}
    Denote $O := (0, 0)$. In light of \cref{example-non-decr-def}, one can notice that $(x_n)_{n \in \NN}$ is well-defined and has distinct consecutive terms. Moreover, $x_n \in \mathbb{R}_{++}^2$ for all $n \in \NN$.

    That said, according to \cref{hsp-lemma}, $\operatorname{bdry}(C_n)$ is a line passing through $z_n$ and having a positive slope $c_n:= \tan((\angle x_{n+1}z_nO + \angle x_{n}z_nO)/2)$ tending to $+ \infty$ as $n \to \infty$. To see the last part, note that $\operatorname{bdry}(C_n)$ intersects the circle $S_n$ between $x_n$ and $x_{n + 1}$ (see \cref{fig:examples5-pic}), due to which, for all $n \geq 1$, \cref{example-non-decr-def} implies: \begin{equation}\label{example-non-decr-ineqs}
        \tan^2(\angle x_{n}z_nO) = \frac{R^2_{n - 1} - a^2_{n - 1}2^{-2(n - 1)}}{(2^{-(n-1)}a_{n - 1} + a_n)^2} \leq c^2_n \leq \frac{R^2_n - a^2_n2^{-2n}}{(1 + 2^{-n})^2a^2_n} = \tan^2(\angle x_{n + 1}z_nO),
    \end{equation} However, from our construction, it follows that $O \not \in C_n$, i.e., $\|x_n\| < \|x_{n+1}\|$ for all $n \in \NN$. Therefore, since $z_n = (-a_n, 0) \to O$, for all $n$ sufficiently large, \begin{equation*}\label{example-non-decr-ineqs1}
        R_n = \|x_n - z_n\| \geq \|x_n\| - \|z_n\| > \|x_0\| - \|z_n\| \geq \frac{1}{2}\|x_0\| > 0,
    \end{equation*} and all the quantities of \cref{example-non-decr-ineqs} tend to $+\infty$ as announced.

    ``\cref{example-non-decr-prop1}'': It follows that \begin{equation}\label{example-non-decr-lines}
    \begin{aligned}
        C_n &= \{(x,y) \in \mathbb{R}^2 \mid y \geq c_n(x + a_n)\}, \\ \operatorname{bdry}(C_n) &= \{(x,y) \in \mathbb{R}^2 \mid y = c_n(x + a_n)\}.
    \end{aligned}
    \end{equation} Pick any $y := (a, b) \in \mathbb{R}_{--}\times \mathbb{R}$, then $y \in C_n$ for all $n$ sufficiently large as \begin{equation*}
        b \geq c_n(a + a_n) = -c_n(|a| - a_n) \to -\infty.
    \end{equation*} In particular, $\mathbb{R}_{--}\times \mathbb{R} \subseteq \underline{\lim}\,C_\NN$ by \cref{example-non-decr-lines}, and, in light of \cref{raik-*}, $$x_n \to z \in X.$$ But, since $d_{\{0\}\times\mathbb{R}}(x_{n+1}) = a_n 2^{-n} \to 0$, and $0<\|x_0\| < \|x_n\|$, we have $z \in \{0\}\times\mathbb{R}_{++}$. Consequently, note that, in light of \cref{example-non-decr-ineqs}, \begin{equation}\label{example-non-decr-equiv} R_n = \|x_n - z_n\| \sim \|z\|, \quad \text{and} \quad c_n \sim \frac{\|z\|}{a_n}. \end{equation} Similarly, if $y = (a, b) \in \mathbb{R}_{++}\times \mathbb{R}$, then, due to \cref{example-non-decr-equiv}, \begin{equation*} c_n(a + a_n) \sim \|z\| \frac{(a + a_n)}{a_n} \to +\infty.\end{equation*} In particular, $y \not \in C_n$ for all $n$ sufficiently large as $b < c_n(a + a_n)$. Thus $y \not \in \underline{\lim}\,C_\NN$. 
    
    Finally, consider $y = (0, b) \in \{0\}\times\mathbb{R}$. Then, due to \cref{example-non-decr-ineqs}, \begin{equation}\label{example-non-decr-ineqs-3}
        c_na_n \leq \frac{\sqrt{R^2_n - a^2_n2^{-2n}}}{1 + 2^{-n}} \leq \frac{R_n}{1 + 2^{-n}} \leq \frac{\|x_n\| - \|z_n\|}{1 + 2^{-n}} \leq \|z\|.
    \end{equation} Therefore, \cref{example-non-decr-equiv} and \cref{example-non-decr-ineqs-3} combined, imply that $y = (0, b) \in \underline{\lim}\,C_\NN$ if and only if $b \geq \|z\|$, which completes the proof of \cref{example-non-decr-prop1}.

    ``\cref{example-non-decr-prop2}'': Indeed, note that, by our construction, $$x_n \in \mathbb{R}_{++}^2 \subseteq N_{\overline{M}}(O) = M^{\ominus},$$ which implies that $P_{\overline{M}}(x_n) = O \not \in C_n$ for all $n \in \NN$. Therefore, $d_M(x_n) = \|x_n\|$ and $(d_M(x_n))_{n \in \NN}$ is strictly increasing. \end{proof}

\section{The proof for \cref{example3}}\label{app-example3-prop}
\begin{proof}
    As in the proof for \cref{aff-count}, we define $M := \mathbb{R} \times \{0\}$, and let $a_n$ and $b_n$ represent the $x$-coordinates of the points $x_{3n}$ and $y_{3n}$, respectively. Specifically, we have: \begin{equation*}
    x_{3n} = (a_n, d_M(x_{3n})), \quad y_{3n} = (b_n, 0).
    \end{equation*} Note that, by our construction, \begin{equation}\label{example3-prop-claim0}\begin{aligned}
        &(\forall n \in \NN) \quad x_{3n + 1} = x_{3n + 3}, \quad x_{3n + 2} = (a_{n+1}, -d_M(x_{3(n+1)})), \\ 
        &(\forall k \in \NN) \quad b_{2k} < a_{2k + 1} < a_{2k} \quad \text{and} \quad a_{2k + 1} < a_{2(k+1)} < b_{2k+1}.
    \end{aligned}\end{equation} Then, by induction, one can prove that \begin{equation}\label{example3-prop-claim1}
        (\forall k \in \NN) \quad -1 < 1 - \sqrt{2} \leq b_{2k} < b_{2(k+1)} < b_{2(k+1)+1} < b_{2k+1} < 0.\end{equation} Indeed, the point $y_{6k} = (b_{2k}, 0)$ represents the leftmost intersection between the circle $S_{6k}$, which is centered at $z_{6k} = z_{6(k+1)} = (1, 0)$ and passes through the point $x_{6k}$, and the $x$-axis (i.e., the set $M$). According to our construction, the point $x_{6k + 4} = x_{6(k + 1)}$ must lie within the circle $S_{6k}$, since it belongs to the minor arc between $x_{6k + 1} = x_{6k + 3} \in S_{6k}$ and $y_{6k + 3}$ which is entirely included inside the circle $S_{6k}$ (see \cref{fig:examples3-pic}). In particular, \begin{equation*}
            d(x_{6(k+1)},(1,0)) < d(x_{6k},(1,0)),
        \end{equation*} which leads to the inequality \begin{equation*} 1 - d(x_{6k},(1,0)) = b_{2k} < b_{2(k+1)} = 1 - d(x_{6(k+1)},(1,0)).\end{equation*} Likewise, one can show that $b_{2(k+1)+1} < b_{2k+1}$. Altogether, \cref{example3-prop-claim1} follows as \begin{equation*}
            (b_{2k})_{k \in \NN} \; \text{is strictly increasing,} \quad (b_{2k + 1})_{k \in \NN} \; \text{is strictly decreasing,}
        \end{equation*} $b_0 = 1 - \sqrt{2} > -1$, $b_1 < 0$ (see \cref{fig:examples3-pic}), and \cref{example3-prop-claim0} implies that $b_{2k} < b_{2k + 1}$ for all $k \in \NN$. That said, \cref{example3-prop-claim0} and \cref{example3-prop-claim1} combined show: \begin{equation}\label{example3-prop-claim2}
            (P_{M}(x_n))_{n \in \NN} \subseteq \left[1 - \sqrt{2}, 0\right]\times\{0\} \subset (-1, 1)\times\{0\}.
        \end{equation} Note that \cref{example3-prop-claim2} guarantees that the whole sequence $(x_{3n})_{n \in \NN}$ with reflections excluded is localized in $[1-\sqrt{2}, 0]\times\mathbb{R}_+$, which allows us to conclude that $\alpha_{n}:=\angle x_{3n}z_{3n}y_{3n}$ does not exceed $\pi/2$ in its absolute value. Thus, since \begin{equation*}
            \left(\forall x \in \left[-\frac{\pi}{2}, \frac{\pi}{2}\right]\right) \quad \frac{\sin(x/2)}{\sin(x)} = \frac{1}{2\cos(x/2)}\leq \frac{1}{\sqrt{2}},
        \end{equation*} we have \begin{equation*}
            d_M(x_{3n + 3}) = d_M(x_{3n+2}) = d_M(x_{3n+1}) = \frac{\sin(\alpha_n/2)}{\sin(\alpha_n)}d_M(x_{3n}) \leq\frac{1}{\sqrt{2}}d_M(x_{3n}),
        \end{equation*} which, in particular, implies that $d_M(x_{3n}) \leq 2^{-n/2} \to 0$.
        
        ``\cref{example3-prop-1}'': As for reflections with respect to the $x$-axis, $C_n$ becomes the closed upper or lower halfplane, with the $x$-axis serving as the hyperplane $\operatorname{bdry}(C_n)$ (see \cref{hsp-lemma}), we conclude that (see \cref{fig:examples3-pic}) \begin{equation*}\begin{aligned}
            (\forall k \in \NN) \quad &C_{6k} \cap C_{6k + 1} \cap C_{6k + 2} = (-\infty, 1]\times\{0\}, \\ &C_{6k + 3} \cap C_{6k + 4} \cap C_{6k + 5} = [-1, +\infty)\times\{0\},
        \end{aligned}\end{equation*} due to which, \cref{example3-prop-1} follows.

        ``\cref{example3-prop-2}'': In view of \cref{raik-*}, the inclusion \cref{example3-prop-claim2}, and the fact that $d_M(x_n) \to 0$, we deduce that \begin{equation*}
            x_n \to z \in [1 - \sqrt{2}, 0] \times \{0\} \subsetneq \ri(\underline{\lim} \, C_\NN).
        \end{equation*} Nevertheless, since $\underline{\lim} \, C_\NN = [-1, 1] \times \{0\}$ is not a closure of any shifted cone, it follows that $\operatorname{int}(\mathcal{K}) = \varnothing$ in light of \cref{lemma-cone-inf}. Yet, by \cref{cone-th}, the same closed convex cone $\mathcal{K}$ must satisfy the inclusion \begin{equation*}
            \operatorname{cone}(\underline{\lim} \, C_\NN - z) = \mathbb{R} \times \{0\} \subseteq \mathcal{K}, \quad \text{where } \lim_{n \to \infty}x_n = z \in \ri(\underline{\lim} \, C_\NN).
        \end{equation*} Altogether, the only possible cone to fit these restrictions is $\mathcal{K} = \mathbb{R} \times \{0\}$. \end{proof}
\end{appendices}

\end{document}